\documentclass[reqno,11pt]{amsart}
\usepackage{amssymb, amsmath,latexsym,amsfonts,amsbsy,amsthm,mathrsfs}
\usepackage{graphicx}
\usepackage{color}

\usepackage{appendix}
\usepackage[colorlinks,linkcolor=blue,anchorcolor=blue,citecolor=red]{hyperref}

\allowdisplaybreaks[3]

\newcommand{\non}{\nonumber}

\def\eqdefa{\buildrel\hbox{\footnotesize def}\over =}

\makeatletter

\newcommand{\Rmnum}[1]{\expandafter\@slowromancap\romannumeral #1@}

\newtheorem{athm}{\bf \t}[section]
\newenvironment{thm} [1] {\def\t{#1}\begin{athm} \bf \rm} {\end{athm}}
\newcommand{\bthm}{\begin{thm}}
\newcommand{\ethm}{\end{thm}}

\newtheorem{theorem}{Theorem}[section]
\newtheorem{lemma}{Lemma}[section]

\newtheorem{proposition}{Proposition}[section]

\newcommand{\beq}{\begin{equation}}
\newcommand{\eeq}{\end{equation}}
\newcommand{\ben}{\begin{eqnarray}}
\newcommand{\een}{\end{eqnarray}}
\newcommand{\beno}{\begin{eqnarray*}}
\newcommand{\eeno}{\end{eqnarray*}}
\newcommand{\bali}{\begin{aligned}}
\newcommand{\eali}{\end{aligned}}

\numberwithin{equation}{section}

\newcommand{\al}{\alpha}

\newcommand{\ve}{\varepsilon}

\newcommand{\f}{\frac}
\newcommand{\na}{\nabla}

\newcommand{\ud}{\mathrm{d}}

\newcommand{\vv}{\mathbf{v}}

\newcommand{\xx}{\mathbf{x}}

\newcommand{\nn}{\mathbf{n}}

\newcommand{\hh}{\mathbf{h}}
\newcommand{\mm}{\mathbf{m}}

\newcommand{\BB}{\mathbf{B}}
\newcommand{\GG}{\mathbf{G}}
\newcommand{\HH}{\mathbf{H}}
\newcommand{\WW}{\mathbf{W}}

\newcommand{\NN}{\mathbf{N}}
\newcommand{\SSS}{\mathbf{S}}

\newcommand{\DD}{\mathbf{D}}
\newcommand{\FF}{\mathbf{F}}

\newcommand{\II}{\mathbf{I}}

\newcommand{\CN}{\mathcal{N}}
\newcommand{\CF}{\mathcal{F}}
\newcommand{\CG}{\mathcal{G}}

\newcommand{\Ff}{\mathfrak{F}}
\newcommand{\Ef}{\mathfrak{E}}

\newcommand{\MH}{\mathscr{H}}
\newcommand{\ME}{\mathscr{E}}
\newcommand{\ML}{\mathscr{L}}
\newcommand{\MA}{\mathscr{A}}
\newcommand{\MB}{\mathscr{B}}

\newcommand{\MC}{\mathscr{C}}
\newcommand{\MP}{\mathscr{P}}
\newcommand{\MU}{\mathscr{U}}

\newcommand{\MD}{\mathscr{D}}

\newcommand{\MT}{\mathscr{T}}

\newcommand{\BS}{{\mathbb{S}^2}}
\newcommand{\BR}{{\mathbb{R}^3}}
\newcommand{\BOm}{\mathbf{\Omega}}

\newcommand{\pa}{\partial}

\begin{document}

\title[The uniaxial limit of Qian-Sheng¡¯s inertial $Q$-tensor theory]
{Rigorous justification of the uniaxial limit from Qian-Sheng¡¯s inertial $Q$-tensor theory to
the Ericksen-Leslie theory}

\author{Sirui Li}
\address{School of  Mathematics and Statistics, Guizhou University, Guiyang 550025, China}
\email{srli@gzu.edu.cn}

\author{Wei Wang}
\address{School of  Mathematical Sciences, Zhejiang University, Hangzhou 310027, China}
\email{wangw07@zju.edu.cn}

\maketitle
\begin{abstract}
In this paper, we rigorously justify the connection between Qian-Sheng's inertial $Q$-tensor model and the full Ericksen-Leslie model for the liquid crystal flow.
By using the Hilbert expansion method, we prove that, when the elastic coefficients tend to zero(also called the uniaxial limit), the solution to the Qian-Sheng's inertial model
will converge to the solution to the full inertial Ericksen-Leslie system.
\end{abstract}

\tableofcontents

\section{Introduction}
Liquid crystals are a state of matter with physical properties between liquid and solid, in which molecules tend to align a preferred direction. In nematic liquid phase, the molecules exhibit long-range orientational order but no positional order.
In physics, different order parameters are introduced to characterize the anisotropic behavior of liquid crystals, which lead to different models.
There are three kinds of widely accepted theories to model nematic liquid crystal flows: {\it the Ericksen-Leslie theory}, {\it the Landau-de Gennes theory} and {\it the Doi-Onsager theory}.
The first two are macroscopic theories which based on continuum mechanics, while the latter one are microscopic kinetic theory derived from the viewpoint of statistical mechanics.
As they are derived from different considerations and are widely used in liquid crystal studies,  to explore the connection between different theories is an important problem. In this paper, we aim to study the
rigorous connection between the Ericksen-Leslie model and the Qian-Sheng model--a representative model in the Landau-de Gennes framework.

Before introducing the Ericksen-Leslie model and the Qian-Sheng model, we list some notations and conventions.
Throughout this papet, the Einstein summation convention is utilized. The space of symmetric traceless tensors is defined as
\begin{align*}
\mathbb{S}^3_0\eqdefa\big\{Q\in \mathbb{R}^{3\times3}:~Q_{ij}=Q_{ji},~Q_{ii}=0\big\},
\end{align*}
which is endowed with the inner product
$Q_1:Q_2=Q_{1ij}Q_{2ij}.$
The set $\mathbb{S}^3_0$ is a five-dimensional subspace of $\mathbb{R}^{3\times3}.$
The matrix norm on $\mathbb{S}^3_0$ is defined as
$|Q|\eqdefa\sqrt{{\rm Tr} Q^2}=\sqrt{Q_{ij}Q_{ij}}$.
For two tensors $A,B\in\mathbb{S}^3_0$ we denote $(A\cdot B)_{ij}=A_{ik}B_{kj}$
and $A:B=A_{ij}B_{ij}$, and their commutator $[A,B]=A\cdot B-B\cdot A$.
For any $Q_1, Q_2\in L^2(\BR)^{3\times3}$, the corresponding inner product is defined by
\begin{align*}
\langle Q_1,Q_2\rangle\eqdefa\int_{\BR}Q_{1ij}(\xx):Q_{2ij}(\xx)\ud\xx.
\end{align*}
We denote by $\nn_1\otimes\nn_2$ the tensor product of two vectors $\nn_1$ and $\nn_2$,
and omit the symbol $\otimes$ for simplicity.
We use $f_{,i}$ to denote $\partial_if$ and $\II$ to denote the $3\times3$ identity tensor. In addition, the superscripted dot denotes the material derivative, i.e., $\dot{f}=(\partial_t+\vv\cdot\nabla)f$, where the fluid velocity $\vv$ can be understood from the context.

\subsection{Ericksen-Leslie theory}
The hydrodynamic theory of nematic liquid crystals was initiated in the seminal work of Ericksen \cite{E-61} and Leslie \cite{Les} in the 1960's. In this theory, the local state of molecular alignments is described by a unit vector $\nn\in\mathbb{S}^2$, called the director.
The corresponding total free energy, called the Oseen-Frank energy, is given by
\begin{align}\label{energy-OF}
E_F(\nn,\nabla\nn)=&\f {k_1} 2(\na\cdot\nn)^2+\f {k_2} 2(\nn{\cdot}(\na\times\nn))^2
+\f {k_3} 2|\nn{\times}(\na\times \nn)|^2\nonumber\\
&+\frac{k_2+k_4}2\big(\textrm{tr}(\na\nn)^2-(\na\cdot\nn)^2\big),
\end{align}
where $k_1, \ldots, k_4$ are constants depending on the material and the temperature.

The full inertial Ericksen-Leslie system can be given as follows:
\begin{align}
&\vv_t+\vv\cdot\nabla\vv=-\nabla{p}+\nabla\cdot\sigma,\label{eq:EL-v}\\
&\na\cdot\vv=0,\label{eq:incompre}\\
&\nn\times\big(I\ddot{\nn}-\hh+\gamma_1\NN+\gamma_2\DD\cdot\nn\big)=0,\label{eq:EL-n}
\end{align}
where $\vv$ is the fluid velocity, $p$ is the pressure penalizing the incompressible condition (\ref{eq:incompre}) of $\vv$, and $I$ is {\it the moment of inertial density} usually considered as a small parameter.
The inertial term $\ddot{\nn}$ is the material derivative of $\dot{\nn}$. Equations (\ref{eq:EL-v}) and (\ref{eq:EL-n})
reflect the conservation laws of linear momentum and angular momentum, respectively.
The stress tensor $\sigma$ consists of the viscous (Leslie) stress $\sigma^L$ and the elastic (Ericksen) stress $\sigma^E$, i.e., $\sigma=\sigma^L+\sigma^E$, which can be described by the following phenomenological constitutive relations:
\begin{align}
\sigma^L=&\alpha_1(\nn\nn:\DD)\nn\nn+\alpha_2\nn\NN+\alpha_3\NN\nn+\alpha_4\DD
+\alpha_5\nn\nn\cdot\DD+\alpha_6\DD\cdot\nn\nn, \label{eq:Leslie stress}\\
\sigma^E=&-\frac{\partial{E_F}}{\partial (\nabla\nn)}\cdot(\nabla\nn)^T,\label{eq:Ericksen}
\end{align}
where
\begin{align*}
\DD=\frac{1}{2}(\nabla\vv+(\nabla\vv)^T), \quad\BOm=\frac12(\nabla\vv-(\nabla\vv)^T),
\quad\NN=\dot{\nn}-\BOm\cdot\nn.
\end{align*}
Additionally, the molecular field $\hh$ is given by
\begin{align*}
\hh=-\frac{\delta{E_F}}{\delta{\nn}}=-\frac{\partial{E_F}}{\partial\nn}+
\nabla\cdot\frac{\partial{E_F}}{\partial(\nabla\nn)}.
\end{align*}

The six constants $\al_1, \cdots, \al_6$ in (\ref{eq:Leslie stress}) are called the Leslie viscosity coefficients. They and the coefficients $\gamma_1, \gamma_2$ together satisfy the following relations
\begin{eqnarray}
&\alpha_2+\alpha_3=\alpha_6-\alpha_5,\label{Leslie relation}\\
&\gamma_1=\alpha_3-\alpha_2,\quad \gamma_2=\alpha_6-\alpha_5.\label{Leslie-coeff}
\end{eqnarray}
The equality (\ref{Leslie relation}) is referred to as {\it Parodi's relation} derived from the Onsager reciprocal relation of irreversible thermodynamics.
The relations (\ref{Leslie relation})-(\ref{Leslie-coeff}) will guarantee that the full Ericksen-Leslie system (\ref{eq:EL-v})-(\ref{eq:EL-n}) fulfils the  energy dissipation law:
\begin{align}
\frac{\ud}{\ud{t}}\int_{\BR}&\Big(\frac{1}{2}|\vv|^2+\frac{I}{2}|\dot{\nn}|^2+E_F\Big)\ud\xx
=-\int_{\BR}\Big((\alpha_1+\frac{\gamma_2^2}{\gamma_1})(\DD:\nn\nn)^2
+\alpha_4|\DD|^2\qquad\nonumber\\
&\quad+\big(\alpha_5+\alpha_6-\frac{\gamma_2^2}{\gamma_1}\big)|\DD\cdot\nn|^2
+\frac{1}{\gamma_1}\big|\nn\times(\hh-I\ddot{\nn})\big|^2\Big)\ud\xx.\quad\label{EL_energy_law}
\end{align}

It is worth emphasizing that the inertial term $I$ in (\ref{eq:EL-n}) is responsible for the hyperbolic feature of the equation describing the molecular orientation. If the inertial term is neglected, then the system (\ref{eq:EL-v})-(\ref{eq:EL-n}) is immediately transformed into its non-inertial counterpart which is a parabolic-type system.

Concerning the non-inertial version of the Ericksen-Leslie theory, the well-posedness results can be refered to \cite{LL,LinW,WZZ2} and the references therein. In particular, under a natural physical condition on the Leslie coefficients, Wang-Zhang-Zhang \cite{WZZ2} proved the well-posedness of the system, and the global existence of weak solution in two-dimensional case was showed in \cite{HLW,WW}. Lin-Wang \cite{LinW2} proved the global existence of a weak solution for 3D case with the initial director field lying in the upper hemisphere. For more related works on the non-inertial Ericksen-Leslie system, for instance, see \cite{LinW,WXL,EKL} and the references therein.

On the other hand, there were also some analytical works devoting to the original inertial Ericksen-Leslie system. Very recently, Jiang-Luo \cite{JL1} established the well-posedness for the full inertial Ericksen-Leslie system in the context of classical solutions. Cai-Wang \cite{CW} studied the global well-posedness of classical solutions to the inertial Ericksen-Leslie model with positive $\gamma_1$.

\subsection{Landau-de Gennes theory}
Landau-de Gennes theory \cite{DG} is capable of providing a rather comprehensive description of the local behaviour of the medium, since it accounts for more complex phenomena of liquid crystals, such as line defects and biaxial configurations. This theory employs a symmetric and traceless tensorial order parameter $Q(\xx)$ to characterize the alignment behaviour of molecular orientations. Physically, $Q(\xx)$ could be understood as the second-order  traceless moment of $f$:
\begin{align*}
Q(\xx)=\int_\BS(\mm\mm-\frac{1}{3}\II)f(\xx,\mm)\ud\mm,
\end{align*}
where $f(\xx,\mm)$ represents the microscopic distribution of molecules with the orientation parallel to $\mm$ at material point $\xx$.
The tensor $Q(\xx)$ is said to be {\it isotropic} if all its eigenvalues are zero, {\it uniaxial} if it has two equal non-zero eigenvalues, and {\it biaxial} if its three  eigenvalues are distinct.

In the absence of boundary constraint and external field, the Landau-de Gennes free energy is given as follows:
\begin{align}
\mathcal{F}(Q,\nabla Q)
&=\int_{\BR}\Big\{-\frac{a}2\text{Tr}(Q^2)
-\frac{b}{3}\text{Tr}(Q^3)+\frac{c}{4}(\text{Tr}(Q^2))^2\nonumber\\
&\qquad+\frac{1}{2}\Big(L_1|\nabla Q|^2+L_2Q_{ij,j}Q_{ik,k}
+L_3Q_{ij,k}Q_{ik,j}\Big) \Big\}\ud\xx\nonumber\\
&\eqdefa\int_{\BR}\Big(f_b(Q)+f_e(\nabla Q)\Big)\ud\xx,\label{eq:Landau-energy}
\end{align}
where $a, b, c$ are non-negative parameters which may depend on the material and temperature, and $L_i(i=1,2,3)$
are material dependent elastic constants. $f_b$ is the bulk energy density describing the isotropic-nematic phase transition, while the elastic energy density $f_e$ penalizes spatial non-homogeneities.
The interested reader refers to \cite{DG,MN} for detailed introductions.

Up to now, some dynamic $Q$-tensor theories have been established to model nematic liquid crystal flows,
which are either derived from the molecular kinetic theory by closure approximations such as \cite{FLS, HLWZZ}
or directly obtained by a variational method such as Beris-Edwards model \cite{BE} and Qian-Sheng model \cite{QS}.
The well-posedness results of the Beris-Edwards system on whole space and bounded domain can be refered to \cite{PZ1,PZ2,HD} and \cite{ADL1, ADL2, LYW1}, respectively.
For the inertial Qian-Sheng model, De Anna and Zarnescu \cite{DZ} studied the local well-posdedness for bounded initial data and global well-posedness under the
assumptions of the small initial data. For the non-viscous version of the Qian-Sheng model, Feireisl et al. \cite{FRSZ}
proved a global existence of the dissipative solution which is inspired from that of incompressible Euler equations.

The Qian-Sheng model  \cite{QS} is a hydrodynamical model which reads as:
\begin{align}
J\ddot{Q}+\mu_1\dot{Q}&=\HH-\frac{\mu_2}{2}\DD+\mu_1[\BOm, Q],\label{eq:Q-general-intro1}\\
\frac{\pa{\vv}}{\pa{t}}+\vv\cdot\nabla\vv&=-\nabla{p}+\nabla\cdot\big(\sigma+\sigma^{d}\big),
\label{eq:Q-general-intro2}\\
\nabla\cdot\vv&=0,\label{eq:Q-general-intro3}
\end{align}
where $\dot{Q}=(\partial_t+\vv\cdot\nabla)Q$, $\ddot{Q}=(\partial_t+\vv\cdot\nabla)\dot{Q}$, the viscous stress $\sigma$, the distortion stress $\sigma^d$ and the molecular field $\HH$ are respectively given by
\begin{align} \label{vis-stressQ}\nonumber
 \sigma=&\beta_1 Q(Q:\DD)+\beta_4 \DD+\beta_5\DD\cdot Q+\beta_6Q\cdot\DD
+\beta_7(\DD\cdot Q^2+Q^2\cdot \DD)\\
&+\frac{\mu_2}{2}(\dot{Q}-[\BOm,Q])+\mu_1\big[Q,(\dot{Q}-[\BOm,Q])\big],\\ \nonumber
\sigma^d_{ij}=& -\frac{\partial\CF}{\partial Q_{kl,j}}\partial_iQ_{kl},\\ \label{vis-fieldQ}
\HH_{ij} =& -\Big(\frac{\delta{\mathcal{F}(Q,\nabla Q)}}{\delta Q}\Big)_{ij}=-\frac{\partial\CF}{\partial Q_{ij}}+\partial_k\Big(\frac{\partial\CF}{\partial Q_{ij,k}}\Big).
\end{align}
Moreover, in (\ref{eq:Q-general-intro1}), $J$ stands for the {\it inertial density} which is usually small.
 The viscosity coefficients $\beta_1, \beta_4, \beta_5, \beta_6, \beta_7, \mu_1$, and $\mu_2$ in (\ref{vis-stressQ})-(\ref{vis-fieldQ})  can be linked by the following relation:
\begin{align}
\beta_6-\beta_5=\mu_2.\label{Q-Parodi}
\end{align}

The system (\ref{eq:Q-general-intro1})-(\ref{eq:Q-general-intro3}) possesses an energy dissipation law, see (\ref{QS-dissip}) in Appendix.
Here we remark that, comparing with the original Qian-Sheng model in \cite{QS}, we add a new viscosity term $\beta_7(D_{ik}Q_{kl}Q_{lj}+Q_{ik}Q_{kl}D_{lj})$ in (\ref{vis-stressQ})
to ensure that  the energy of the system  will always dissipate without assuming any relation between $\beta_5$ and $\beta_6$.
Indeed, if $\beta_7=0$, we have to assume $\beta_5+\beta_6=0$, otherwise the energy may not dissipate, see Lemma \ref{dissip-rel-QS}.
However, the condition $\beta_5+\beta_6=0$ is so strong that it can not be satisfied by many liquid crystal materials.
Therefore, we introduce the $\beta_7$ term and assume that
\begin{align}\label{viscosity-cond}
\left\{
\begin{aligned}
&\beta_1,\beta_4,\mu_1>0,~\beta_4-\frac{\mu_2^2}{4\mu_1}>0, ~\beta_7\ge 0;\\
&(\beta_5+\beta_6)^2<8\beta_7\big(\beta_4-\frac{\mu_2^2}{4\mu_1}\big), ~\text{if}~ \beta_7\neq0; \beta_5+\beta_6=0, ~\text{if}~ \beta_7=0.\\
\end{aligned}
\right.
\end{align}
The detailed discussion of the dissipative relation is referred to Lemma \ref{dissip-rel-QS} in the Appendix.

\subsection{Motivations and main results}
The intricate connection between different dynamical theories for liquid crystals is not only of significance in mathematical literature,
but also directly related to many physical properties. The fundamental subject, generally involving the singular limit problem,
has drawn a lot of attention in physics and applied mathematics communities.
In this respect, the formal asymptotic expansions were first constructed by Kuzzu-Doi \cite{KD} to
derive the homogenous non-inertial Ericksen-Leslie system from the Doi-Onsager system and to determine the Leslie coefficients, under the small Deborah number limit.
E-Zhang \cite{EZ} extended Kuzzu-Doi's derivation and obtained the inhomogenous  non-inertial homogenous Ericksen-Leslie system. In particular, the Ericksen stress is derived from a body force.
Their formal derivation was rigorously justified by Wang-Zhang-Zhang \cite{WZZ1} under the small Deborah number limit.
Based on the same spirit, Li-Wang-Zhang \cite{LWZ} provided a strict derivation from the molecular-based $Q$-tensor system,
obtained from the molecular kinetic theory by the Bingham closure, to the non-inertial Ericksen-Leslie system.
Similar rigorous results were initiated by Wang et al. in \cite{WZZ3} concerning the Beris-Edwards system in Landau-de Gennes framework.
A unified formulation for liquid crystal modeling was put forward by Han et al. in \cite{HLWZZ} to establish relations between microscopic theories and  macroscopic theories.
There are also some interesting works which have explored the relations between different dynamical
theories for liquid crystals  in the framework of weak solutions, see \cite{LYW2}.

Recently, to better understand the limit of zero inertia for the full Ericksen-Leslie model, Jiang et al. \cite{JLTZ} first study a limit connecting
a scaled wave map with heat flow into the unit sphere $\BS$. Later on, Jiang et al. \cite{JL2,JLT1} investigate the zero inertial limit from the
full inertial Ericksen-Leslie model to the non-inertial one.

The main goal of this paper is to rigorously justify the connection between the inertial Qian-Sheng model and the full inertial Ericksen-Leslie model, in a sense of smooth solutions.

In contrast to the constants $a,b,c$, the elastic coefficients $L_i(i=1,2,3)$ in (\ref{eq:Landau-energy}) are usually regarded as being small,
so we consider the following rescaled energy functional with a small parameter $\ve$:
\begin{align}
\mathcal{F}_{\ve}(Q,\nabla Q)=&\int_{\BR}\Big(\frac{1}{\ve}f_b(Q)+f_e(\nabla Q)\Big)\ud\xx,\label{eq:Landau-energy-ve}
\end{align}
and $a, b, c, L_i(1\le i\le 3)\sim O(1)$.
We assume that the elastic coefficients $L_i$-s satisfy
\begin{align*}
L_1>0,~L_1+L_2+L_3>0,
\end{align*}
which will ensure that the elastic energy is strictly positive (see Lemma 2.5 in \cite{WZZ3}), i.e., there exists some constant $L_0=L_0(L_1,L_2,L_3)>0$ such that
\begin{align}\label{L:postive}
\int_{\BR}f_e(\nabla Q)\ud\xx\geq L_0\|\nabla Q\|_{L^2}.
\end{align}

Then the Qian-Sheng model with a small parameter $\ve$ can be written as:
\begin{align}
J\ddot{Q}^{\ve}+\mu_1\dot{Q}^{\ve}&=\HH^{\ve}-\frac{\mu_2}{2}\DD^{\ve}
+\mu_1[\BOm^{\ve}, Q^{\ve}],\label{eq:QS-general-ve1}\\
\frac{\pa{\vv^{\ve}}}{\pa{t}}+\vv^{\ve}\cdot\nabla\vv^{\ve}&=-\nabla{p^{\ve}}+\nabla\cdot\big(\sigma_{\ve}+\sigma^{d}_{\ve}\big),
\label{eq:QS-general-ve2}\\
\nabla\cdot\vv^{\ve}&=0,\label{eq:QS-general-ve3}
\end{align}
where $\dot{Q}^{\ve}=(\partial_t+\vv^{\ve}\cdot\nabla)Q^{\ve},~\ddot{Q}^{\ve}=(\partial_t+\vv^{\ve}\cdot\nabla)\dot{Q}^{\ve}$, and
\begin{align*}
\DD^{\ve}=&\frac12(\nabla\vv^{\ve}+(\nabla\vv^{\ve})^T),
~\BOm^{\ve}=\frac12(\nabla\vv^{\ve}-(\nabla\vv^{\ve})^T), \\
\sigma_{\ve}=&\beta_1 Q^{\ve}(Q^{\ve}:\DD^{\ve})+\beta_4 \DD^{\ve}+\beta_5\DD^{\ve}\cdot Q^{\ve}+\beta_6Q^{\ve}\cdot\DD^{\ve}+\beta_7(\DD^{\ve}\cdot {Q^{\ve}}^2+{Q^{\ve}}^2\cdot\DD^{\ve})\\
&+\frac{\mu_2}{2}(\dot{Q}^{\ve}-[\BOm^{\ve},Q^{\ve}])+\mu_1\big[Q^{\ve},(\dot{Q}^{\ve}-[\BOm^{\ve},Q^{\ve}])\big],\\
(\sigma^d_{\ve})_{ji}=&-\frac{\partial\CF_{\ve}}{\partial Q^{\ve}_{kl,j}}Q^{\ve}_{kl,i}
\eqdefa\sigma^d(Q^{\ve},Q^{\ve}).
\end{align*}
The tensor $\sigma^d(Q,\overline{Q})$ is denoted as
\begin{align*}
\sigma^d_{ji}(Q,\overline{Q})
=-\big(L_1Q_{kl,j}\overline{Q}_{kl,i}+L_2Q_{km,m}\overline{Q}_{kj,i}+L_3Q_{kj,l}\overline{Q}_{kl,i}\big).
\end{align*}
The molecular field $\HH^{\ve}$ is given by
\begin{align*}
\HH^{\ve}(Q)=-\frac{1}{\ve}\frac{\partial f_b}{\partial Q}+\partial_i\Big(\frac{\partial f_e}{\partial Q_{,i}}\Big)\eqdefa-\frac{1}{\ve}\MT(Q)-\ML(Q),
\end{align*}
where two operators $\MT$ and $\ML$ are respectively defined by
\begin{align*}
\MT(Q)=&-aQ-bQ^2+c|Q|^2Q+\frac13b|Q|^2\II,\\
(\ML(Q))_{kl}=&-\Big(L_1\Delta Q_{kl}+\frac12(L_2+L_3)(Q_{km,ml}+Q_{lm,mk}-\frac23\delta_{kl}Q_{ij,ij})\Big).
\end{align*}

For a given director field $\nn(t,\xx)$ and $s=\frac{b+\sqrt{b^2+24ac}}{4c}$, we define
\begin{align*}
\MP^{out}(Q)=&Q-(\nn\nn\cdot Q+Q\cdot\nn\nn)+2(Q:\nn\nn)\nn\nn,\\
\MH_{\nn}(Q)=&bs\Big(Q-(\nn\nn\cdot Q+Q\cdot\nn\nn)+\frac23(Q:\nn\nn)\II\Big)+2cs^2(Q:\nn\nn)(\nn\nn-\frac13\II),
\end{align*}
which will be explained in Subsection \ref{crit-liner}.
We also take the viscosity coefficients in the full inertial Ericksen-Leslie model as:
\begin{align}\label{leslie1-intro}
\left\{
\begin{aligned}
&\alpha_1=\beta_1s^2,\quad
\alpha_2=\frac12\mu_2s-\mu_1s^2, \\
&\alpha_3=\frac12\mu_2s+\mu_1s^2, \quad \alpha_4=\beta_4-\frac13(\beta_5+\beta_6)s+\frac29\beta_7s^2,\\
&\alpha_5=\beta_5s+\frac13\beta_7s^2,\quad
\alpha_6=\beta_6s+\frac13\beta_7s^2,
\end{aligned}
\right.
\end{align}
and the coefficients $\gamma_1$, $\gamma_2$ and the inertial coefficient $I$ are
\begin{align}\label{leslie2-intro}
\gamma_1=2\mu_1s^2,\quad\gamma_2=\mu_2s,\quad I=2Js^2.
\end{align}
In addition, the elastic constants in the Oseen-Frank energy are given by
\begin{align}\label{OF-LD-relation-intro}
k_1=k_3=2(L_1+L_2+L_3)s^2,\quad k_2=2L_1s^2,\quad k_4=L_3s^2.
\end{align}

Throughout this paper, we assume that the viscosity coefficient $\mu_1$ is large enough compared
with the inertial coefficient $J$, i.e., $\mu_1\gg J$, and the  condition (\ref{viscosity-cond}) holds, and the elastic constants $L_i(i=1,2,3)$ satisfy $L_1>0,~L_1+L_2+L_3>0$.
The main result of this paper is stated as follows.
\begin{theorem}\label{thm:main1}
Let $(\nn(t,\xx), \vv(t,\xx))$ be a smooth solution of the full inertial Ericksen-Leslie model (\ref{eq:EL-v})--(\ref{eq:EL-n}) on $[0,T]$
with the coefficients given by (\ref{leslie1-intro})-(\ref{OF-LD-relation-intro}), which satisfies
\beno
(\vv, \partial_t\nn, \nabla\nn)\in L^{\infty}([0,T];H^{k})\quad \textrm{for}\quad k\ge 20.
\eeno
Let $Q_0(t,\xx)=s\big(\nn(t,\xx)\nn(t,\xx)-\frac13\II\big)$, and the functions $\big(Q_1,Q_2,Q_3, \vv_1,\vv_2\big)$ are determined by Proposition \ref{prop:Hilbert}.
Suppose that the initial data $(Q^{\ve}(0,\xx), \partial_tQ^{\ve}(0,\xx), \vv^\ve(0,\xx))$ takes the form
\begin{align*}
&Q^\ve(0,\xx)=\sum^3_{k=0}\ve^kQ_{k}(0,\xx)+\ve^3Q_{R,0}^\ve(\xx),~~~
\vv^\ve(0,\xx)=\sum^2_{k=0}\ve^k\vv_{k}(0,\xx)+\ve^3\vv_{R,0}^\ve(\xx),\\
&\partial_tQ^\ve(0,\xx)=\sum^3_{k=0}\ve^k\partial_tQ_{k}(0,\xx)+\ve^3\partial_tQ_{R,0}^\ve(\xx),
\end{align*}
where $(Q^{\ve}_{R,0}(\xx), \partial_tQ^{\ve}_{R,0}(\xx), \vv^\ve_{R,0}(\xx))$ fulfils
\begin{align*}
\|\vv_{R,0}^\ve\|_{H^2}+\|Q_{R,0}^\ve\|_{H^3}+\|\partial_tQ_{R,0}^{\ve}\|_{H^2}+\ve^{-1}\|\MP^{out}(Q^\ve_{R,0})\|_{L^2}\le E_0.
\end{align*}
Then there exists $\ve_0>0$ and $E_1>0$ such that for all $\ve<\ve_0$, the inertial Qian-Sheng model (\ref{eq:QS-general-ve1})-(\ref{eq:QS-general-ve3}) has a unique solution
$(Q^\ve(t,\xx), \vv^\ve(t,\xx))$ on $[0,T]$ that has the Hilbert expansion
\begin{align*}
Q^\ve(t,\xx)=\sum^3_{k=0}\ve^kQ_k(t,\xx)+\ve^3Q^{\ve}_R(t,\xx),~~
\vv^\ve(t,\xx)=\sum^2_{k=0}\ve^k\vv_k(t,\xx)+\ve^3\vv^{\ve}_R(t,\xx),
\end{align*}
where, for any $t\in[0,T]$, $(Q^{\ve}_R,\vv^{\ve}_R)$ satisfies
\begin{align*}
\Ef(Q^{\ve}_R(t),\vv^{\ve}_R(t))\leq E_1.
\end{align*}
Here $\Ef$ is defined by
\begin{align}\label{energy-functional-Ef}
\Ef(Q,\vv)\eqdefa&\int_{\BR}\Big(|\vv|^2+|Q|^2+|\dot{Q}|^2+\frac{1}{\ve}\MH^{\ve}_{\nn}(Q):Q\Big)
+\ve^2\Big(|\nabla\vv|^2+|\partial_i\dot{Q}|^2\nonumber\\
&\quad+\frac{1}{\ve}\MH^{\ve}_{\nn}(\partial_iQ):\partial_iQ\Big)
+\ve^4\Big(|\Delta\vv|^2+|\Delta\dot{Q}|^2+\frac{1}{\ve}\MH^{\ve}_{\nn}(\Delta Q):\Delta Q\Big),
\end{align}
and $\dot{Q}=(\partial_t+\tilde{\vv}\cdot\nabla)Q,~\tilde{\vv}=\sum^2_{k=0}\ve^k\vv_k$, $\MH^\ve_\nn(Q)=\MH_{\nn}(Q)+\ve\ML(Q)$ in (\ref{energy-functional-Ef}) and the constant $E_1$ is independent of $\ve$.
\end{theorem}

Let us spend some words on the rough idea of proving the main result. We first make a formal expansion for the solution $(Q^{\ve},\vv^{\ve})$:
\begin{align*}
Q^{\ve}(t,\xx)=&Q_0(t,\xx)+\ve Q_1(t,\xx)+\ve^2Q_2(t,\xx)+\ve^3Q_3(t,\xx)+\ve^3Q_R(t,\xx),\\
\vv^{\ve}(t,\xx)=&\vv_0(t,\xx)+\ve\vv_1(t,\xx)+\ve^2\vv_2(t,\xx)+\ve^3\vv_R(t,\xx).
\end{align*}
If plugging the above expansion into the inertial Qian-Sheng system (\ref{eq:QS-general-ve1})-(\ref{eq:QS-general-ve3}), then we obtain a hierarchy of equations in Subsection \ref{Hilb-expan}.
The $O(\ve^{-1})$ equation gives $\MT(Q_0)=0$, which implies by Proposition \ref{critical-points} that
\begin{align*}
Q_0=s\Big(\nn\nn-\frac13\II\Big)
\end{align*}
for some $\nn\in\BS$ and $s=\frac{b+\sqrt{b^2+24ac}}{4c}$. For the $O(1)$ system, we can obtain that $(\vv_0,\nn)$ is
exactly a solution of the full inertial Ericksen-Leslie system with the coefficients given by (\ref{leslie1-intro})-(\ref{OF-LD-relation-intro}).
Moreover, it can be shown that the existence of $(Q_i,\vv_i)$ with $i\ge 1$ for $O(\ve^i)$ can be guaranteed
by the fact that $(Q_i,\vv_i)$ satisfies a linear dissipative system, see Proposition \ref{prop:Hilbert}.

The core part to rigorously justify the uniaxial limit is to prove the uniform (in $\ve$) bounds for the remainders$(Q_R,\vv_R)$.
For this end, we write the equation for $(Q_R,\vv_R)$ which roughly reads as:
\begin{align*}
J\ddot{Q}_R+\mu_1\dot{Q}_R=&-\frac{1}{\ve}\MH^{\ve}_{\nn}(Q_R)
-\frac{\mu_2}{2}\DD_R+\mu_1[\BOm_R,Q_0]+\widetilde{\FF}_R+\cdots,\\
\frac{\partial\vv_R}{\partial t}+\tilde{\vv}\cdot\nabla\vv_R=&-\nabla p_R+\nabla\cdot\Big(\frac{\mu_2}{2}(\dot{Q}_R-[\BOm_R,Q_0])+\mu_1\big[Q_0,(\dot{Q}_R-[\BOm_R,Q_0])\big]\Big)\\
&+\cdots.
\end{align*}
The main difficulty terms are the singular (in $\ve$) term $\frac{1}{\ve}\MH^{\ve}_{\nn}(Q_R)$, and the term
$\widetilde{\FF}_R$ which includes second order derivatives of $Q_R$ (see (\ref{remainder-Q-R}) for the precise definition).
To control the singular term, we have to include $\langle \frac{1}{\ve}\MH^{\ve}_{\nn}(Q_R), Q_R\rangle$ into the energy, see (\ref{energy-functional-Ef}). However,
the operator $\MH^{\ve}_{\nn}$ is dependent of $t$, and its time derivative will bring some difficult terms such as
\begin{align}\label{sigul-terms}
\frac{1}{\ve}\langle\dot{\overline{\nn\nn}}\cdot Q_R,Q_R\rangle,\quad \frac{1}{\ve}\langle(Q_R:\dot{\overline{\nn\nn}})\nn\nn,Q_R\rangle.
\end{align}
In the non-inertial case, (\ref{sigul-terms}) can be controlled with the help of the dissipation term $\frac{1}{\ve^2}\|\MH^{\ve}_{\nn}(Q_R)\|_{L^2}^2$, see \cite[Lemma 4.1]{WZZ3}.
However, it does not work in our case.
Another trouble term $\widetilde{\FF}_R$ can not be directly estimated either, as it contains second order derivatives of $Q_R$.

To overcome these difficulties, we  choose a delicate modified energy $\widetilde{\Ef}$ (see (\ref{compli-energ-func})),
and then use the symmetric and cancellation structures of the system to close the energy estimate. Some key steps for
the estimates are summarized in Lemma \ref{lem:singular}-\ref{lem:FR-2}, where the evolution equation of $Q_R$ will be frequently used.
Moreover, we can show that the energy functional $\widetilde{\Ef}$ is positive and $\widetilde{\Ef}\sim\Ef$ if $\mu_1\gg J$, and thus accomplish  the main steps of the proof
for Theorem \ref{thm:main1} in principle.

\section{The Hilbert expansion}
This section is devoted to deriving the Hilbert expansion for the inertial Qian-Sheng system (\ref{eq:QS-general-ve1})-(\ref{eq:QS-general-ve3}).
In particular, we will show that the $O(1)$ system is just the full Ericksen-Leslie system.
The existence of $O(\ve^k)(k\ge1)$ system in the Hilbert expansion will also be proved.

We first give some preliminary results about critical points and the linearized operator.

\subsection{Critical points and the linearized operator}\label{crit-liner}
 A tensor $Q_0$ is called a critical point of $F_b(Q)$ if $\MT(Q_0):=\frac{\partial F_b}{\partial Q}\big|_{Q=Q_0}=0$. The following characterization of critical points can be seen from \cite{Maj,WZZ3}.
\begin{proposition}\label{critical-points}
$\MT(Q)=0$ if and only if $Q=s(\nn\nn-\frac13\II)$ for some $\nn\in\mathbb{S}^2$, where $s=0$ or a solution of $2cs^2-bs+3a=0$, that is,
\begin{align*}
s_{1}=\frac{b+\sqrt{b^2+24ac}}{4c} ~or~ s_{2}=\frac{b-\sqrt{b^2+24ac}}{4c}.
\end{align*}
Moreover, the critical point $Q_0=s(\nn\nn-\frac13\II)$ is stable if $s=s_1$.
\end{proposition}

Given a critical point $Q_0=s(\nn\nn-\frac13\II)$, the linearized operator $\MH_{Q_0}$ of $\MT(Q)$ around $Q_0$ is given by
\begin{align*}
\MH_{Q_0}(Q)=aQ-b(Q_0\cdot Q+Q\cdot Q_0)+c|Q_0|^2Q+2(Q_0:Q)\Big(cQ_0+\frac{b}{3}\II\Big).
\end{align*}
Then a direct calculation yields
\begin{align}
\MH_{Q_0}(Q)=&bs\big(Q-(\nn\nn\cdot Q+Q\cdot\nn\nn)+\frac23(Q:\nn\nn)\II\big)+2cs^2(Q:\nn\nn)(\nn\nn-\frac13\II)\nonumber\\
\eqdefa&\MH_{\nn}(Q).
\end{align}

The kernel space of the linearized operator $\MH_{\nn}$, being a two-dimensional subspace of $\mathbb{S}^3_0$, can be defined as
\begin{align*}
{\rm Ker}\MH_{\nn}\eqdefa\{\nn\nn^{\perp}+\nn^{\perp}\nn\in\mathbb{S}^3_0: \nn^{\perp}\in\mathbb{V}_{\nn}\},
\end{align*}
for any given $\nn\in\mathbb{S}^2$, where $\mathbb{V}_{\nn}\eqdefa\{\nn^{\perp}\in\BR: \nn^{\perp}\cdot\nn=0\}$. Let $\mathscr{P}^{in}$ be the projection operator from $\mathbb{S}^3_0$ to
$\text{Ker}\MH_{\nn}$ and $\MP^{out}$ the projection operator from $\mathbb{S}^3_0$ to
$(\text{Ker}\MH_{\nn})^{\perp}$. Using the following simple fact that
\[
 |Q-(\nn\nn^{\perp}+\nn^{\perp}\nn)|^2
 =|Q|^2-2|Q\cdot\nn|^2+2|Q:\nn\nn|^2+|\nn^{\perp}-(\II-\nn\nn)\cdot Q\cdot\nn|^2,
\]
then the projection operators $\mathscr{P}^{in}$ and $\mathscr{P}^{out}$ are expressed as, respectively,
\begin{align}
 \MP^{in}(Q)=&\nn[(\II-\nn\nn)\cdot Q\cdot\nn]+[(\II-\nn\nn)\cdot Q\cdot\nn]\nn\nonumber\\
 =&(\nn\nn\cdot Q+Q\cdot\nn\nn)-2(Q:\nn\nn)\nn\nn,\label{projection_in1}\\
 \MP^{out}(Q)=&Q- \MP^{in}(Q)\nonumber\\
 =&Q-(\nn\nn\cdot Q+Q\cdot\nn\nn)+2(Q:\nn\nn)\nn\nn.\label{projection_out1}
\end{align}

The important properties of the linearized operator $\MH_{\nn}$ can be found in \cite{WZZ3}.
\begin{proposition}\label{linearized-oper-prop}
{\rm(i)} For any $\nn\in\BS$, it holds that $\MH_{\nn}\mathrm{Ker}\MH_{\nn}=0$, i.e., $\MH_{\nn}(Q)\in(\mathrm{Ker}\MH_{\nn})^{\perp}$.\\
{\rm(ii)} There exists a constant $C_0=c_0(a,b,c)>0$ such that for any $Q\in(\mathrm{Ker}\MH_{\nn})^{\perp}$,
\begin{align*}
\MH_{\nn}(Q):Q\geq c_0|Q|^2.
\end{align*}
{\rm(iii)} $\MH_{\nn}$ is a 1-1 map on $(\mathrm{Ker}\MH_{\nn})^{\perp}$ and its inverse $\MH^{-1}_{\nn}$ is given by
\begin{align}\label{HM-inverse}
\MH^{-1}_{\nn}(Q)=&\frac{1}{bs}\Big(Q-(\nn\nn\cdot Q+Q\cdot\nn\nn)+\frac23(Q:\nn\nn)\II\Big)\nonumber\\
&+\frac{4b+2cs}{bs(4cs-b)}(Q:\nn\nn)(\nn\nn-\frac13\II).
\end{align}
\end{proposition}

\subsection{The Hilbert expansion}\label{Hilb-expan}
Let $(Q^{\ve},\vv^{\ve})$ be a solution of the system (\ref{eq:QS-general-ve1})-(\ref{eq:QS-general-ve3}).
We perform the following Hilbert expansion:
\begin{align}
&Q^{\ve}=\sum^3_{k=0}\ve^kQ_k+\ve^3 Q_{R}\eqdefa\widetilde{Q}+\ve^3 Q_{R},\label{Qvar}\\
&\vv^{\ve}=\sum^2_{k=0}\ve^k\vv_k+\ve^3 \vv_{R}\eqdefa\tilde{\vv}+\ve^3 \vv_{R},\label{vvar}
\end{align}
where $Q_i(0\le i\le 3),\vv_j(0\le j\le2)$ do not depend on $\ve$, while $(Q_R,\vv_R)$ are called the remainder term which depend upon $\ve$.

As is shown in (\ref{expan:Q0-1})-(\ref{eq:expansion2-ve3}) below,
inserting the Hilbert expansion (\ref{Qvar})-(\ref{vvar}) into the system (\ref{eq:QS-general-ve1})-(\ref{eq:QS-general-ve3}) and equating like powers of $\ve$ leads to a hierarchy of equations. We will prove that $(Q_i, \vv_i)(0\le i\le 2)$ and $Q_3$ can be determined in this way: $Q_0$ must be a critical point of $\MT(Q)$, and the system of $(Q_0,\vv_0)$ could be reduced to the full inertial Ericksen-Leslie system, while $(Q_i, \vv_i)(1\le i\le 2)$ and $Q_3$ solve the linear equations obtained by using the projection operators.

For $Q_i\in\mathbb{R}^{3\times3}(i=1,2,3)$, we introduce the following definitions:
\begin{align*}
\MB(Q_1,Q_2)\eqdefa& Q_1\cdot Q_2+Q^T_2\cdot Q^T_1-\frac23(Q_1:Q_2)\II,\\
\MC(Q_1,Q_2,Q_3)\eqdefa& Q_1(Q_2:Q_3)+Q_2(Q_1:Q_3)+Q_3(Q_1:Q_2).
\end{align*}
Let $\widehat{Q}^{\ve}=Q_1+\ve Q_2+\ve^2Q_3$, just as the polynomial expansion technique adopted in \cite{WZZ3}, we get the expansion of $\MT(Q^{\ve})$ in $\ve$ as follows:
\begin{align}\label{CJQ-ve}
\MT(Q^{\ve})=&\MT(Q_0)+\ve\MH_{\nn}(Q_1)+\ve^2(\MH_{\nn}(Q_2)+\BB_1)+\ve^3(\MH_{\nn}(Q_3)+\BB_2)\nonumber\\
&+\ve^3\MH_{\nn}(Q_R)+\ve^4\MT^{\ve}_R,
\end{align}
where $\BB_1, \BB_2$ and $\BB^{\ve}$, independent of $Q_R$, are respectively
\begin{align*}
\BB_1=&-\frac{b}{2}\MB(Q_1,Q_1)+c\MC(Q_0,Q_1,Q_1),\\
\BB_2=&-b\MB(Q_1,Q_2)+2c\MC(Q_0,Q_1,Q_2),\\
\BB^{\ve}=&-\frac{b}{2}\sum\limits_{\mbox{\tiny$\begin{array}{c}
i+j\geq4\\
1\leq i,j\leq 3\end{array}$}}\ve^{i+j-4}\MB(Q_i,Q_j)\\
&+\frac{c}{3}\sum_{\mbox{\tiny$\begin{array}{c}
i+j+k\geq4\\
\text{at least two of}~i,j,k~\text{are not zero}\end{array}$}}
\ve^{i+j+k-4}\MC(Q_i,Q_j,Q_k),
\end{align*}
and the fourth order term $\MT^{\ve}_R$ in $\ve$ is given by
\begin{align}\label{TR}
\MT^{\ve}_R=&\BB^{\ve}-b\MB(\widehat{Q}^{\ve},Q_R)+c\MC(Q_R,\widehat{Q}^{\ve},Q_0)
+\frac{c}{2}\ve\MC(Q_R,\widehat{Q}^{\ve},\widehat{Q}^{\ve})\nonumber\\
&-\frac{b}{2}\ve^2\MB(Q_R,Q_R)+c\ve^2\MC(Q_R,Q_R,Q_0+\ve\widehat{Q}^{\ve})+c\ve^5\MC(Q_R,Q_R,Q_R).
\end{align}
For the sake of brevity, we also denote
\begin{align*}
&\HH_0=\MH_{\nn}(Q_1)+\ML(Q_0),\\
&\HH_1=\MH_{\nn}(Q_2)+\ML(Q_1)+\BB_1,\\
&\HH_2=\MH_{\nn}(Q_3)+\ML(Q_2)+\BB_2.
\end{align*}

We are now in a position to write down the expansion of the original system (\ref{eq:QS-general-ve1})-(\ref{eq:QS-general-ve3})
and collect the terms (independent of $Q_R$) with same order of $\ve$. Specifically, we have:\\
$\bullet$~{\bf The $O(\ve^{-1})$ system}\\
\begin{align}\label{expan:Q0-1}
\MT(Q_0)=0.
\end{align}
$\bullet$~{\bf Zero order term in $\ve$}\\
\begin{align}
J\ddot{Q}_0+\mu_1\dot{Q}_0=&-\HH_0-\frac{\mu_2}{2}\DD_0+\mu_1[\BOm_0, Q_0],\label{eq:expansion0-ve1}\\
\frac{\partial\vv_0}{\partial t}+\vv_0\cdot\nabla\vv_0=&-\nabla p_0+\nabla\cdot\Big(\beta_1Q_0(Q_0:\DD_0)
+\beta_4\DD_0+\beta_5\DD_0\cdot Q_0\nonumber\\
&+\beta_6Q_0\cdot\DD_0
+\beta_7(\DD_0\cdot Q^2_0+Q^2_0\cdot\DD_0)\nonumber\\
&+\frac{\mu_2}{2}\CN_0
+\mu_1[Q_0,\CN_0]+\sigma^d(Q_0,Q_0)\Big),\label{eq:expansion0-ve2}\\
\nabla\cdot\vv_0=&0,\label{eq:expansion0-ve3}
\end{align}
where
\begin{align*}
\ddot{Q}_0=(\partial_t+\vv_0\cdot\nabla)\dot{Q}_0,\quad \dot{Q}_0=(\partial_t+\vv_0\cdot\nabla)Q_0,\quad
\CN_0=\dot{Q}_0-[\BOm_0,Q_0].
\end{align*}
$\bullet$~{\bf First order term in $\ve$}\\
\begin{align}
J\ddot{Q}_1+\mu_1\dot{Q}_1=&-\HH_1-\frac{\mu_2}{2}\DD_1+\mu_1\big([\BOm_1, Q_0]+[\BOm_0, Q_1]
-\vv_1\cdot \nabla Q_0\big)\nonumber\\
&-J\Big(2\vv_1\cdot\nabla(\partial_tQ_0)+\partial_t\vv_1\cdot\nabla Q_0
+(\vv_1\cdot\nabla)\vv_0\cdot\nabla Q_0\nonumber\\
&+(\vv_0\cdot\nabla)\vv_1\cdot\nabla Q_0+(\vv_1\vv_0:\nabla^2)Q_0+(\vv_0\vv_1:\nabla^2)Q_0\Big),\label{eq:expansion1-ve1}\\
\frac{\partial\vv_1}{\partial t}+\vv_0\cdot\nabla\vv_1=&-\vv_1\cdot\nabla\vv_0-\nabla p_1+\nabla\cdot\Big(\beta_1\big(Q_0(Q_0:\DD_1)\nonumber\\
&+Q_0(Q_1:\DD_0)+Q_1(Q_0:\DD_0)\big)+\beta_4\DD_1\nonumber\\
&+\beta_5(\DD_0\cdot Q_1+\DD_1\cdot Q_0)+\beta_6(Q_0\cdot\DD_1+Q_1\cdot\DD_0)\nonumber\\
&+\beta_7\big(\DD_1\cdot Q^2_0+Q^2_0\cdot\DD_1+\DD_0\cdot Q_1\cdot Q_0+\DD_0\cdot Q_0\cdot Q_1\nonumber\\
&+Q_1\cdot Q_0\cdot\DD_0+Q_0\cdot Q_1\cdot\DD_0\big)+\frac{\mu_2}{2}\overline{\CN}_1\nonumber\\
&+\mu_1([Q_1,\CN_0]+[Q_0,\overline{\CN}_1])
+\sigma^d(Q_1,Q_0)+\sigma^d(Q_0,Q_1)\Big),\label{eq:expansion1-ve2}\\
\nabla\cdot\vv_1=&0,\label{eq:expansion1-ve3}
\end{align}
where
\begin{align*}
\ddot{Q}_1=&(\partial_t+\vv_0\cdot\nabla)\dot{Q}_1,\quad \dot{Q}_1=(\partial_t+\vv_0\cdot\nabla)Q_1,\\
\CN_1=&\dot{Q}_1-[\BOm_0,Q_1],\quad \overline{\CN}_1=\CN_1+\vv_1\cdot\nabla Q_0-[\BOm_1, Q_0].
\end{align*}
$\bullet$~{\bf Second order term in $\ve$}\\
\begin{align}
J\ddot{Q}_2+\mu_1\dot{Q}_2=&-\HH_2-\frac{\mu_2}{2}\DD_2+\mu_1\Big([\BOm_2, Q_0]+[\BOm_0, Q_2]
+[\BOm_1, Q_1]\nonumber\\
&-\vv_2\cdot\nabla\vv_0-\vv_1\cdot \nabla Q_1\Big)\nonumber\\
&-J\Big(2\vv_1\cdot\nabla(\partial_tQ_1)+2\vv_2\cdot\nabla(\partial_tQ_0)+\partial_t\vv_1\cdot\nabla Q_1\nonumber\\
&+\partial_t\vv_2\cdot\nabla Q_0+(\vv_1\cdot\nabla)\vv_1\cdot\nabla Q_0+(\vv_0\cdot\nabla)\vv_1\cdot\nabla Q_1\nonumber\\
&+(\vv_0\vv_1:\nabla^2)Q_1+(\vv_1\vv_0:\nabla^2)Q_1+(\vv_1\vv_1:\nabla^2)Q_0\nonumber\\
&+(\vv_2\vv_0:\nabla^2)Q_0+(\vv_0\vv_2:\nabla^2)Q_0\Big),\label{eq:expansion2-ve1}\\
\frac{\partial\vv_2}{\partial t}+\vv_0\cdot\nabla\vv_2=&-\vv_2\cdot\nabla\vv_0-\vv_1\cdot\nabla\vv_1-\nabla p_2+\nabla\cdot\Big(\beta_1\sum\limits_{i+j+k=2}Q_i(Q_j:\DD_k)\nonumber\\
&+\beta_4\DD_2+\beta_5(\DD_2\cdot Q_0+\DD_1\cdot Q_1+\DD_0\cdot Q_2)\nonumber\\
&+\beta_6(Q_0\cdot\DD_2+Q_1\cdot\DD_1+Q_2\cdot\DD_0)\nonumber\\
&+\beta_7\sum\limits_{i+j+k=2}\big(\DD_i\cdot Q_j\cdot Q_k+Q_i\cdot Q_j\cdot\DD_k\big)
+\frac{\mu_2}{2}\overline{\CN}_2\nonumber\\
&+\mu_1\big([Q_2,\CN_0]+[Q_1,\overline{\CN}_1]+[Q_0,\overline{\CN}_2]\big)\nonumber\\
&+\sigma^d(Q_2,Q_0)+\sigma^d(Q_1,Q_1)+\sigma^d(Q_0,Q_2)\Big),\label{eq:expansion2-ve2}\\
\nabla\cdot\vv_2=&0,\label{eq:expansion2-ve3}
\end{align}
where
\begin{align*}
\ddot{Q}_2=&(\partial_t+\vv_0\cdot\nabla)\dot{Q}_2,\quad \dot{Q}_2=(\partial_t+\vv_0\cdot\nabla)Q_2,\quad
\CN_2=\dot{Q}_2-[\BOm_0,Q_2],\\
\overline{\CN}_2=&\CN_2+\vv_2\cdot\nabla Q_0+\vv_1\cdot\nabla Q_1-[\BOm_2, Q_0]-[\BOm_1, Q_1].
\end{align*}

In the sequel, we will show how to solve $(Q_i, \vv_i)(0\le i\le 2)$ and $Q_3$. First of all, combining the equation (\ref{expan:Q0-1}) with Proposition \ref{critical-points}, we deduce that $Q_0$ is a critical point and
could be taken as
\begin{align}\label{Q-0}
Q_0(t,\xx)=s(\nn(t,\xx)\nn(t,\nn)-\frac13\II),
\end{align}
for some $\nn(t,\xx)\in\mathbb{S}^2$ and $s=s_1$.

\begin{proposition}\label{prop:EL}
Suppose that $(Q_0,\vv_0)$ is a smooth solution of the system (\ref{eq:expansion0-ve1})-(\ref{eq:expansion0-ve3}), then $(\nn, \vv_0)$ must be a
solution of the full inertial Ericksen-Leslie system (\ref{eq:EL-v})--(\ref{eq:EL-n}),
where the coefficients are determined by (\ref{leslie1-intro})-(\ref{OF-LD-relation-intro}).
\end{proposition}

\begin{proof}
Recalling the first property $\MH_{\nn}(Q_1)\in(\mathrm{Ker}\MH_\nn)^\bot$ in Proposition \ref{linearized-oper-prop}, we can deduce from the equation (\ref{eq:expansion0-ve1}) that
\begin{align}\label{Qvnv}
\Big(J\ddot{Q}_0+\mu_1\CN_0+\ML(Q_0)+\frac{\mu_2}{2}\DD_0\Big)
:(\nn\nn^{\perp}+\nn^{\perp}\nn)=0.
\end{align}
Substituting (\ref{Q-0}) into (\ref{Qvnv}), we get by a subtle calculation that
\begin{align*}
\ddot{Q}_0:(\nn\nn^{\perp}+\nn^{\perp}\nn)=&s(\ddot{\nn}\nn+2\dot{\nn}\dot{\nn}+\nn\ddot{\nn}):(\nn\nn^{\perp}+\nn^{\perp}\nn)\\
=&2s\ddot{\nn}\cdot\nn^{\perp},\\
\CN_0:(\nn\nn^{\perp}+\nn^{\perp}\nn)=&[s(\dot{\nn}\nn+\nn\dot{\nn})+s(\nn\nn\cdot\BOm_0-\BOm_0\cdot\nn\nn)]
:(\nn\nn^{\perp}+\nn^{\perp}\nn)\\
=&2s\NN\cdot\nn^{\perp},\\
\ML(Q_0):(\nn\nn^{\perp}+\nn^{\perp}\nn)=&-\frac{1}{s}\hh\cdot\nn^{\perp},\\
\DD_0:(\nn\nn^{\perp}+\nn^{\perp}\nn)=&2(\DD_0\cdot\nn)\cdot\nn^{\perp},
\end{align*}
from which it follows that
\begin{align*}
\nn^{\perp}\cdot\Big(2s^2J\ddot{\nn}+2s^2\mu_1\NN-\hh+s\mu_2\DD_0\cdot\nn\Big)=0,
\end{align*}
which implies
\begin{align}\label{eq:nn}
\nn\times\Big(I\ddot{\nn}-\hh+\gamma_1\NN+\gamma_2\DD_0\cdot\nn\Big)=0,
\end{align}
where
\begin{align*}
I=2s^2J,~~\gamma_1=2s^2\mu_1,~~\gamma_2=s\mu_2.
\end{align*}

Applying the definition of $\mathrm{Ker}\MH_\nn$ and (\ref{Q-0}) yields
\begin{align*}
\CN_0=&\frac{\partial Q_0}{\partial t}+\vv_0\cdot\nabla Q_0+Q_0\cdot\BOm_0-\BOm_0\cdot Q_0\\
=&s(\nn\NN+\NN\nn)\in\mathrm{Ker}\MH_\nn.
\end{align*}
Consequently, we have
\begin{align*}
\sigma_0\eqdefa&\beta_1Q_0(Q_0:\DD_0)+\beta_4\DD_0+\beta_5\DD_0\cdot Q_0+\beta_6Q_0\cdot\DD_0\\
&+\beta_7(\DD_0\cdot Q^2_0+Q^2_0\cdot\DD_0)
+\frac{\mu_2}{2}\CN_0+\mu_1(Q_0\cdot\CN_0-\CN_0\cdot Q_0)\\
=&\beta_1s^2(\nn\nn:\DD_0)\nn\nn-\frac13\beta_1s^2(\nn\nn:\DD_0)\II+\beta_4\DD_0+\beta_5s\DD_0\cdot\nn\nn\\
&+\beta_6s\nn\nn\cdot\DD_0-\frac13(\beta_5+\beta_6)s\DD_0
+\frac13\beta_7s^2(\DD_0\cdot\nn\nn+\nn\nn\cdot\DD_0)\\
&+\frac29\beta_7s^2\DD_0+\frac12\mu_2s(\nn\NN+\NN\nn)+\mu_1s^2(\nn\NN-\NN\nn)\\
=&\beta_1s^2(\nn\nn:\DD_0)\nn\nn+(\frac12\mu_2s-\mu_1s^2)\NN\nn+(\frac12\mu_2s+\mu_1s^2)\nn\NN\\
&+\big(\beta_4-\frac13(\beta_5+\beta_6)s+\frac29\beta_7s^2\big)\DD_0
+(\beta_5s+\frac13\beta_7s^2)\nn\nn\cdot\DD_0\\
&+(\beta_6s+\frac13\beta_7s^2)\DD_0\cdot\nn\nn
+\text{pressure terms}\\
=&\sigma^L+\text{pressure terms}.
\end{align*}
In addition, from Lemma 3.5 in \cite{WZZ3} we know that $\sigma^E=\sigma^d(Q_0,Q_0)$.
Here $\sigma^E$ and $\sigma^L$ (see (\ref{eq:Leslie stress}) and (\ref{eq:Ericksen})) are just the elastic stress and the viscous stress in the full inertial Ericksen-Leslie system, respectively. This completes the proof of Proposition \ref{prop:EL}.
\end{proof}

\subsection{Existence of the Hilbert expansion}
In this subsection, we are going to elucidate the existence of the Hilbert expansion. In other words, we will show how to solve $(Q_1,\vv_1)$ and $Q_3$ from the system (\ref{eq:expansion1-ve1})-(\ref{eq:expansion2-ve3}). To be more specific, we have the following Proposition \ref{prop:Hilbert}.
\begin{proposition}\label{prop:Hilbert}
Let $(\nn,\vv_0)$ be a smooth solution of the full inertial Ericksen-Leslie system (\ref{eq:EL-v})-(\ref{eq:EL-n}) on $[0,T]$ and satisfy
\begin{align*}
(\vv_0, \partial_t\nn, \nabla\nn)\in L^{\infty}([0,T];H^{k})\quad \textrm{for} \quad k\ge 20.
\end{align*}
Then there exists the solution $(Q_i,\vv_i)(i=0,1,2)$ and $Q_3\in(\mathrm{Ker}\MH_{\nn})^{\perp}$ of the system
(\ref{eq:expansion1-ve1})-(\ref{eq:expansion2-ve3}) satisfying
\begin{align*}
(\vv_i,\partial_tQ_i,\nabla Q_i)\in L^{\infty}([0,T];H^{k-4i})(i=0,1,2),\quad Q_3\in L^{\infty}([0,T];H^{k-11}).
\end{align*}
\end{proposition}

Before proving Proposition \ref{prop:Hilbert}, we need the following Lemma \ref{diss:relation} from \cite{WZZ2} and Lemma \ref{lem:s5-1}.
\begin{lemma}\label{diss:relation}
The following dissipation relation holds
\begin{align}\label{diss:ineq}
\hat{\beta}_1|\nn\nn:\DD|^2+\hat{\beta}_2|\DD|^2+\hat{\beta}_3|\nn\cdot\DD|^2\geqq 0
\end{align}
for any symmetric traceless matric $\DD$ and unit vector $\nn$, if and only if
\begin{align}\label{diss:coeff}
\hat{\beta}_2\geqq 0,\quad 2\hat{\beta}_2+\hat{\beta}_3\geqq 0,\quad \frac32\hat{\beta}_2+\hat{\beta}_3+\hat{\beta}_1\geqq 0.
\end{align}
\end{lemma}

\begin{lemma}\label{lem:s5-1}
Assume that $Q_1=Q^{\top}_1+Q^{\bot}_1$ with $Q^{\top}_1\in\text{Ker}\MH_{\nn}$ and $Q^{\bot}_1\in(\text{Ker}\MH_{\nn})^{\perp}$. Then it follows that
\begin{align}
\MP^{out}(\dot{Q}_1)=&L(Q^{\top}_1)+R,\quad
\MP^{in}(\dot{Q}_1)=\dot{Q}^{\top}_1+L(Q^{\top}_1)+R,\label{mp-dQ1}\\
\MP^{out}(\ddot{Q}_1)=&L(\dot{Q}^{\top}_1)+L(Q^{\top}_1)+R,\quad
\MP^{in}(\ddot{Q}_1)=\ddot{Q}^{\top}_1+L(\dot{Q}^{\top}_1)+L(Q^{\top}_1)+R,\label{mp-ddQ1}
\end{align}
where $\dot{Q}^{\top}_1\eqdefa(\partial_t+\vv_0\cdot\nabla)Q^{\top}_1$ and
$\ddot{Q}^{\top}_1\eqdefa(\partial_t+\vv_0\cdot\nabla)\dot{Q}^{\top}_1$. In addition, $L(\cdot)$ represents the linear function with the coefficients belonging to $L^{\infty}([0,T];H^{k-1})$ and $R\in L^{\infty}([0,T];H^{k-3})$ some function depending only on $\nn, \vv_0, Q^{\bot}_1$.
\end{lemma}
\begin{proof}
The proof of (\ref{mp-dQ1}) see \cite{WZZ3} for the details. It remains to prove (\ref{mp-ddQ1}).
Let $Q^{\top}_1=\nn\nn^{\bot}+\nn^{\bot}\nn$ with $\nn^{\bot}\cdot\nn=0$, then it follows that
\begin{align*}
\dot{Q}^{\top}_1=&\nn\dot{\nn}^{\bot}+\dot{\nn}\nn^{\bot}+\dot{\nn}^{\bot}\nn+\nn^{\bot}\dot{\nn},\\
\ddot{Q}^{\top}_1=&2(\dot{\nn}\dot{\nn}^{\bot}+\dot{\nn}^{\bot}\dot{\nn})
+\nn\ddot{\nn}^{\bot}+\ddot{\nn}^{\bot}\nn+\ddot{\nn}\nn^{\bot}+\nn^{\bot}\ddot{\nn},
\end{align*}
where $\dot{\nn}=(\partial_t+\vv_0\cdot\nabla)\nn$ and $\dot{\nn}^{\bot}=(\partial_t+\vv_0\cdot\nabla)\nn^{\bot}$.
Note that
\begin{align*}
&\nn^{\bot}\cdot\nn=\dot{\nn}\cdot\nn=0,\quad\dot{\overline{(\nn^{\bot}\cdot\nn)}}=\dot{\nn}^{\bot}\cdot\nn+\nn^{\bot}\cdot\dot{\nn}=0,\\
&\ddot{\overline{(\nn^{\bot}\cdot\nn)}}=\ddot{\nn}^{\bot}\cdot\nn
+2\dot{\nn}^{\bot}\cdot\dot{\nn}+\nn^{\bot}\cdot\ddot{\nn}=0.
\end{align*}
By a simple computation, we have
\begin{align*}
(\II-\nn\nn)\cdot\ddot{Q}^{\top}_1\cdot\nn
&=(\delta_{ij}-n_in_j)\Big(2(\dot{n}_j\dot{n}^{\bot}_k+\dot{n}^{\bot}_j\dot{n}_k)+n_j\ddot{n}^{\bot}_k+\ddot{n}^{\bot}_jn_k
+\ddot{n}_jn^{\bot}_k+n^{\bot}_j\ddot{n}_k\Big)n_k\\
&=2\dot{n}_i\dot{n}^{\bot}_kn_k+\ddot{n}^{\bot}_i+n^{\bot}_i\ddot{n}_kn_k-n_i\ddot{n}^{\bot}_kn_k\\
&=\ddot{n}^{\bot}_i+(n_in^{\bot}_k+n^{\bot}_in_k)\ddot{n}_k-2\dot{n}_in^{\bot}_k\dot{n}_k+2n_i\dot{n}^{\bot}_k\dot{n}_k.
\end{align*}
Consequently, using the fact $\nn^{\bot}=Q^{\top}_1\cdot\nn$ and the definition of the projection operator $\MP^{in}$, we obtain
\begin{align*}
\MP^{in}(\ddot{Q}^{\top}_1)=&\nn\big((\II-\nn\nn)\cdot\ddot{Q}^{\top}_1\cdot\nn\big)
+\big((\II-\nn\nn)\cdot\ddot{Q}^{\top}_1\cdot\nn\big)\nn\\
=&\nn\ddot{\nn}^{\bot}+\ddot{\nn}^{\bot}\nn+2(\dot{Q}^{\top}_1:\nn\dot{\nn})\nn\nn+L(Q^{\top}_1),
\end{align*}
from which it yields that
\begin{align*}
\MP^{out}(\ddot{Q}_1)=&\MP^{out}(\ddot{Q}^{\top}_1)+R\\
=&2(\dot{\nn}\nn\cdot\dot{Q}^{\top}_1+\dot{Q}^{\top}_1\cdot\nn\dot{\nn})
+2(\dot{Q}^{\top}_1:\nn\dot{\nn})\nn\nn+L(Q^{\top}_1)+R\\
=&L(\dot{Q}^{\top}_1)+L(Q^{\top}_1)+R.
\end{align*}
Therefore, we can deduce that
\begin{align*}
\MP^{in}(\ddot{Q}_1)=&\ddot{Q}_1-\MP^{out}(\ddot{Q}_1)=\ddot{Q}^{\top}_1+L(\dot{Q}^{\top}_1)+L(Q^{\top}_1)+R.
\end{align*}
\end{proof}

\noindent{\bf Proof of Proposition \ref{prop:Hilbert}}. Suppose that $(\nn,\vv_0)$ is a smooth solution of the full inertial Ericksen-Leslie model (\ref{eq:EL-v})-(\ref{eq:EL-n}) on $[0,T]$ such that
\[
 (\vv_0, \partial_t\nn, \nabla\nn)\in L^{\infty}([0,T];H^k), ~\text{for}~ k\geq20.
\]
Thanks to $Q_0=s(\nn(t,\xx)\nn(t,\xx)-\frac13\II)$, we know $Q_0\in L^{\infty}([0,T],H^{k+1})$.
Note that we could solve $Q^{\perp}_1$ from (\ref{eq:expansion0-ve1}), and  easily get
$Q^{\perp}_1\in L^{\infty}([0,T];H^{k-1})$ by Proposition \ref{linearized-oper-prop}.
Thus, the existence of $(Q_1,\vv_1)$ can be reduced to solving $(Q^{\top}_1,\vv_1)$.

The key observation is that $(Q^{\top}_1,\vv_1)$ satisfies a linear dissipative system, although the system seems
 nonlinear at a first glance due to the term $\HH_1$ in (\ref{eq:expansion1-ve1}) which contains $\BB_1$.
 For this end, we derive the linear system of $(\vv_1,Q^{\top}_1)$. We denote
\begin{align*}
\widehat{\BB}_1(Q,\overline{Q})=-b\big(Q\cdot\overline{Q}-\frac13(Q:\overline{Q}\II)\big)
+c\big(2(Q:Q_0)\overline{Q}+(Q:\overline{Q})Q_0\big).
\end{align*}
Thus we have
\begin{align*}
\BB_1=&\widehat{\BB}_1(Q_1,Q_1)=\widehat{\BB}_1(Q^{\top}_1,Q^{\top}_1)+\widehat{\BB}_1(Q^{\top}_1,Q^{\bot}_1)
+\widehat{\BB}_1(Q^{\bot}_1,Q^{\top}_1)+\widehat{\BB}_1(Q^{\bot}_1,Q^{\bot}_1)\\
=&\widehat{\BB}_1(Q^{\top}_1,Q^{\top}_1)+L(Q^{\top}_1,\vv_1).
\end{align*}
By a simple calculation we get
\begin{align}\label{BB1}
\widehat{\BB}_1(Q^{\top}_1,Q^{\top}_1)\in(\text{Ker}\MH_{\nn})^{\perp}.
\end{align}

We denote
\begin{align*}
&\MA=\MP^{in}(\ML(Q^{\top}_1)),\quad \MC_1=\MP^{in}([\BOm_1, Q_0]),\quad
\MD_1=\MP^{in}(\DD_1),\\
&\MU=\MP^{in}(\dot{\vv}_1\cdot\nabla Q_0),\quad\MC_2=\MP^{out}([\BOm_1, Q_0]),\quad \MD_2=\MP^{out}(\DD_1).
\end{align*}
Taking the projection $\MP^{in}$ on both sides of the equation (\ref{eq:expansion1-ve1}), notice that $\MH_{\nn}(Q_2)\in (\text{Ker}\MH_{\nn})^{\perp}$ and $\ML(Q_1)=\ML(Q^{\top}_1)+R$, from Lemma \ref{lem:s5-1} and (\ref{BB1}) we obtain that
\begin{align*}
J\ddot{Q}^{\top}_1+\mu_1\dot{Q}^{\top}_1=&-\MA-\frac{\mu_2}{2}\MD_1+\mu_1\MC_1-J\MU+L(\dot{Q}^{\top}_1)+L(Q^{\top}_1,\vv_1)+R.
\end{align*}
Note that, due to (\ref{BB1}), the nonlinear term $\widehat{\BB}_1(Q^{\top}_1,Q^{\top}_1)$ vanishes in the above equation.

Thus, we have the following closed linear system of $(Q^{\top}_1,\vv_1)$:
\begin{align}
J\ddot{Q}^{\top}_1+\mu_1\dot{Q}^{\top}_1=&-\MA-\frac{\mu_2}{2}\MD_1+\mu_1\MC_1-J\MU+L(\dot{Q}^{\top}_1)+L(Q^{\top}_1,\vv_1)+R,\label{hilbert-ex1}\\
\frac{\partial\vv_1}{\partial t}+\vv_0\cdot\nabla\vv_1=&-\nabla p_1+\nabla\cdot\Big(\beta_1Q_0(Q_0:\DD_1)
+\beta_4\DD_1+\beta_5\DD_1\cdot Q_0\nonumber\\
&+\beta_6Q_0\cdot\DD_1+\beta_7(\DD_1\cdot Q^2_0+Q^2_0\cdot\DD_1)\nonumber\\
&+\frac{\mu_2}{2}(\dot{Q}^{\top}_1-[\BOm_1,Q_0])+\mu_1\big[Q_0,(\dot{Q}^{\top}_1-[\BOm_1,Q_0])\big]\nonumber\\
&+\sigma^d(Q^{\top}_1,Q_0)+\sigma^d(Q_0,Q^{\top}_1)+L(Q^{\top}_1,\vv_1)+R\Big),\label{hilbert-ex2}\\
\nabla\vv_1=&0.\label{hilbert-ex3}
\end{align}

In order to prove the unique solvability of the linear system (\ref{hilbert-ex1})-(\ref{hilbert-ex3}), we need to present an a priori estimate for the following energy
\begin{align*}
\ME(t)\eqdefa&\|\vv_1\|^2_{L^2}+\langle Q^{\top}_1,\ML(Q^{\top}_1)\rangle+\|\dot{Q}^{\top}_1\|^2_{L^2}
+\|Q^{\top}_1\|^2_{L^2},
\end{align*}
that is to prove the energy inequality
\begin{align}
\frac{\ud}{\ud t}\ME(t)\leq C(\ME(t)+\|R(t)\|_{L^2}),\label{ME-energy}
\end{align}
where the solution $(Q^{\top}_1,\vv_1)$ satisfy $(\vv_1,\partial_tQ^{\top}_1,\nabla Q^{\top}_1)\in L^{\infty}([0,T];H^{k-4})$.

First of all, from the equation (\ref{hilbert-ex1}) and (\ref{L:postive}) we have
\begin{align}\label{Qtop1-L2}
&J\langle\ddot{Q}^{\top}_1,Q^{\top}_1\rangle+\mu_1\langle\dot{Q}^{\top}_1,Q^{\top}_1\rangle\nonumber\\
&=\Big\langle-\ML(Q^{\top}_1)-\frac{\mu_2}{2}\DD_1+\mu_1[\BOm_1,Q_0],Q^{\top}_1\Big\rangle
-J\langle\dot{\vv}_1\cdot\nabla Q_0,Q^{\top}_1\rangle\nonumber\\
&\quad+\big\langle L(\dot{Q}^{\top}_1)+L(Q^{\top}_1,\vv_1)+R,Q^{\top}_1\big\rangle\nonumber\\
&\leq-\frac{\ud}{\ud t}\langle\vv_1\cdot\nabla Q_0,Q^{\top}_1\rangle
+\delta\|\nabla\vv_1\|^2_{L^2}+C_{\delta}(\|\vv_1\|^2_{L^2}\nonumber\\
&\quad+\|\dot{Q}^{\top}_1\|^2_{L^2}+\|Q^{\top}_1\|^2_{H^1}+\|R\|^2_{L^2}),
\end{align}
where we have been obliged to estimate the term $-J\langle\dot{\vv}_1\cdot\nabla Q_0,Q^{\top}_1\rangle$. In fact, from integration by parts we know that
\begin{align*}
&-J\langle\dot{\vv}_1\cdot\nabla Q_0,Q^{\top}_1\rangle\\
&=-\frac{\ud}{\ud t}\langle\vv_1\cdot\nabla Q_0,Q^{\top}_1\rangle
+\langle\vv_1\cdot(\partial_t+\vv_0\cdot\nabla)\nabla Q_0,Q^{\top}_1\rangle
+\langle\vv_1\cdot\nabla Q_0,\dot{Q}^{\top}_1\rangle\\
&\leq
-\frac{\ud}{\ud t}\langle\vv_1\cdot\nabla Q_0,Q^{\top}_1\rangle
+C(\|\vv_1\|^2_{L^2}+\|Q^{\top}_1\|^2_{L^2}+\|\dot{Q}^{\top}_1\|^2_{L^2}).
\end{align*}

It can be observed that, for any $Q\in\mathbb{S}^3_0$, there holds
\begin{align}\label{ddQ-Q}
\langle\ddot{Q},Q\rangle=&\int_{\BR}\partial_t\dot{Q}_{ij}Q_{ij}+v_k\partial_k\dot{Q}_{ij}Q_{ij}\ud\xx\nonumber\\
=&\int_{\BR}\Big(\partial_t(\dot{Q}_{ij}Q_{ij})+v_k\partial_k(\dot{Q}_{ij}Q_{ij})-\dot{Q}_{ij}\dot{Q}_{ij}\Big)\ud\xx\nonumber\\
=&\frac{\ud}{\ud t}\int_{\BR}\dot{Q}:Q\ud\xx-\int_{\BR}|\dot{Q}|^2\ud\xx.
\end{align}
From (\ref{Qtop1-L2}) and (\ref{ddQ-Q}) we thus obtain
\begin{align}\label{Qtop1-L22}
&\frac{\ud}{\ud t}\int_{\BR}\Big(J\dot{Q}^{\top}_1:Q^{\top}_1
+J(\vv_1\cdot\nabla Q_0):Q^{\top}_1
+\frac{\mu_1}{2}|Q^{\top}_1|^2\Big)\ud\xx\nonumber\\
&\leq\delta\|\nabla\vv_1\|^2_{L^2}+C_{\delta}(\|\vv_1\|^2_{L^2}+\|\dot{Q}^{\top}_1\|^2_{L^2}
+\|Q^{\top}_1\|^2_{H^1}+\|R\|^2_{L^2}).
\end{align}

Taking advantage of the linear system (\ref{hilbert-ex1})-(\ref{hilbert-ex3}) and integration by parts over $\BR$, we know
\begin{align}\label{hilbert-estimate}
&\langle\partial_t\vv_1,\vv_1\rangle+J\langle\ddot{Q}^{\top}_1, \dot{Q}^{\top}_1\rangle\nonumber\\
&=\underbrace{-\Big\langle\beta_1Q_0(Q_0:\DD_1)
+\beta_4\DD_1+\beta_5\DD_1\cdot Q_0+\beta_6Q_0\cdot\DD_1,\nabla\vv_1\Big\rangle}_{I_1}\nonumber\\
&\quad\underbrace{-\Big\langle\beta_7(\DD_1\cdot Q^2_0+Q^2_0\cdot\DD_1),\nabla\vv_1\Big\rangle}_{I_2}\nonumber\\
&\quad\underbrace{-\Big\langle\frac{\mu_2}{2}(\dot{Q}^{\top}_1-[\BOm_1,Q_0])
+\mu_1\big[Q_0,(\dot{Q}^{\top}_1-[\BOm_1,Q_0])\big],\nabla\vv_1\Big\rangle}_{I_3}\nonumber\\
&\quad\underbrace{-\mu_1\langle\dot{Q}^{\top}_1-\MC_1,\dot{Q}^{\top}_1\rangle-\Big\langle\MA+\frac{\mu_2}{2}\MD_1,\dot{Q}^{\top}_1\Big\rangle}_{I_4}
\underbrace{-J\langle\MU,\dot{Q}^{\top}_1\rangle}_{I_5}\nonumber\\
&\quad+\underbrace{\Big\langle L(\dot{Q}^{\top}_1)+L(Q^{\top}_1,\vv_1)+R, \dot{Q}^{\top}_1\Big\rangle
+\Big\langle\sigma^d(Q^{\top}_1,Q_0)+\sigma^d(Q_0,Q^{\top}_1),\nabla\vv_1\Big\rangle}_{I_6}\nonumber\\
&\quad+\underbrace{\langle L(Q^{\top}_1,\vv_1)+R,\nabla\vv_1\rangle}_{I_7}.
\end{align}
We next estimate the right-hand side of (\ref{hilbert-estimate}) term by term. Using $Q_0=s(\nn\nn-\frac13\II)$ and the relation $\beta_6-\beta_5=\mu_2$ in (\ref{Q-Parodi}), note that $\langle[\DD_1,Q_0],\DD_1\rangle=0$, we obtain that
\begin{align*}
I_1+I_2=&-\Big\langle\beta_1Q_0(Q_0:\DD_1)+\beta_4\DD_1+\frac{\beta_5+\beta_6}{2}(Q_0\cdot\DD_1+\DD_1\cdot Q_0),\DD_1\Big\rangle\\
&-\Big\langle\beta_7(\DD_1\cdot Q^2_0+Q^2_0\cdot\DD_1),\DD_1\Big\rangle\\
&+\Big\langle\Big(\frac{\beta_5+\beta_6}{2}-\beta_5\Big)\DD_1\cdot Q_0
+\Big(\frac{\beta_5+\beta_6}{2}-\beta_6\Big)Q_0\cdot\DD_1, \DD_1+\BOm_1\Big\rangle\\
=&-\beta_1s^2\|\nn\nn:\DD_1\|^2_{L^2}-\Big(\beta_4-\frac{s(\beta_5+\beta_6)}{3}
+\frac29\beta_7s^2\Big)\|\DD_1\|^2_{L^2}\\
&-\Big(s(\beta_5+\beta_6)+\frac23\beta_7s^2\Big)\|\nn\cdot\DD_1\|^2_{L^2}
+\frac{\mu_2}{2}\langle[\BOm_1,Q_0],\DD_1\rangle.
\end{align*}
Making use of $\MP^{in}(\dot{Q}^{\top}_1)=\dot{Q}^{\top}_1+L(Q^{\top}_1)$ and the self-adjoint property of the projection operator yields
that
\begin{align}\label{LL-ddQ1}
-\langle\MA,\dot{Q}^{\top}_1\rangle
=&-\langle\ML(Q^{\top}_1),\dot{Q}^{\top}_1+L(Q^{\top}_1)\rangle\nonumber\\
=&-\langle\ML(Q^{\top}_1),\partial_tQ^{\top}_1\rangle
-\langle\ML(Q^{\top}_1),\vv_0\cdot\nabla Q^{\top}_1\rangle
-\langle\ML(Q^{\top}_1),L(Q^{\top}_1)\rangle\nonumber\\
\leq&-\frac12\frac{\ud}{\ud t}\langle Q^{\top}_1,\ML(Q^{\top}_1)\rangle+C\|Q^{\top}_1\|^2_{H^1}.
\end{align}
Here we have employed the following fact that, for any $Q\in \mathbb{S}^3_0$,
\begin{align}\label{MLQ-vNQ}
&-\langle\ML(Q),\vv_0\cdot\nabla Q\rangle\nonumber\\
&=\int_{\BR}v_{0j}Q_{kl,j}\Big(L_1\Delta Q_{kl}
+\frac12(L_2+L_3)\big(Q_{km,ml}+Q_{lm,mk}-\frac23\delta_{kl}Q_{ij,ij}\big)\Big)\ud\xx\nonumber\\
&=\int_{\BR}\Big(-L_1v_{0j}Q_{kl,mj}Q_{kl,m}-\frac12(L_2+L_3)(v_{0j}Q_{kl,lj}Q_{km,m}+v_{0j}Q_{kl,kj}Q_{lm,m})\nonumber\\
&\qquad\quad-L_1v_{0j,m}Q_{kl,j}Q_{kl,m}-\frac12(L_2+L_3)(v_{0j,l}Q_{kl,j}Q_{km,m}+v_{0j,k}Q_{kl,j}Q_{lm,m})\Big)\ud\xx\nonumber\\
&=\int_{\BR}\Big(-L_1v_{0j,m}Q_{kl,j}Q_{kl,m}-\frac12(L_2+L_3)(v_{0j,l}Q_{kl,j}Q_{km,m}+v_{0j,k}Q_{kl,j}Q_{lm,m})\Big)\ud\xx\nonumber\\
&\leq C\|Q\|^2_{H^1}.
\end{align}
In addition, we have
\begin{align}\label{DD-ddQ1}
-\frac{\mu_2}{2}\langle\MD_1,\dot{Q}^{\top}_1\rangle
=&-\frac{\mu_2}{2}\langle\DD_1,\MP^{in}(\dot{Q}^{\top}_1)\rangle\nonumber\\
=&-\frac{\mu_2}{2}\langle\DD_1,\dot{Q}^{\top}_1+L(Q^{\top}_1)\rangle\nonumber\\
\leq&-\frac{\mu_2}{2}\langle\DD_1,\dot{Q}^{\top}_1\rangle
+\delta\|\nabla\vv_1\|^2_{L^2}+C_{\delta}\|Q^{\top}_1\|^2_{L^2}.
\end{align}

For terms $I_3$ and $I_4$, we notice that
\begin{align*}
\mu_1\langle\MC_1,\dot{Q}^{\top}_1\rangle=&\mu_1\langle[\BOm_1,Q_0],\MP^{in}(\dot{Q}^{\top}_1)\rangle
=\mu_1\langle[\BOm_1,Q_0],\dot{Q}^{\top}_1+L(Q^{\top}_1)\rangle\\
\leq&\mu_1\langle[\BOm_1,Q_0],\dot{Q}^{\top}_1\rangle
+\delta\|\nabla\vv_1\|^2_{L^2}+C_{\delta}\|Q^{\top}_1\|^2_{L^2},
\end{align*}
then from (\ref{LL-ddQ1}) and (\ref{DD-ddQ1}), we get
\begin{align*}
I_3+I_4\leq&-\frac{\mu_2}{2}\langle(\dot{Q}^{\top}_1-[\BOm_1,Q_0]),\DD_1\rangle
-\mu_1\Big\langle\big[Q_0,\BOm_1\big],(\dot{Q}^{\top}_1-[\BOm_1,Q_0])\Big\rangle\\
&-\mu_1\langle\dot{Q}^{\top}_1-[\BOm_1,Q_0],\dot{Q}^{\top}_1\rangle
-\frac12\frac{\ud}{\ud t}\langle Q^{\top}_1,\ML(Q^{\top}_1)\rangle
-\frac{\mu_2}{2}\langle\DD_1,\dot{Q}^{\top}_1\rangle\\
&+\delta\|\nabla\vv_1\|^2_{L^2}+C_{\delta}\|Q^{\top}_1\|^2_{H^1}\\
=&-\mu_2\langle\dot{Q}^{\top}_1-[\BOm_1,Q_0],\DD_1\rangle-\frac{\mu_2}{2}\langle[\BOm_1,Q_0],\DD_1\rangle
-\mu_1\|\dot{Q}^{\top}_1-[\BOm_1,Q_0]\|^2_{L^2}\\
&-\frac12\frac{\ud}{\ud t}\langle Q^{\top}_1,\ML(Q^{\top}_1)\rangle
+\delta\|\nabla\vv_1\|^2_{L^2}+C_{\delta}\|Q^{\top}_1\|^2_{H^1}\\
\leq&-\mu_1\Big\|\dot{Q}^{\top}_1-[\BOm_1,Q_0]+\frac{\mu_2}{2\mu_1}\DD_1\Big\|^2_{L^2}
+\frac{\mu^2_2}{4\mu_1}\|\DD_1\|^2_{L^2}
-\frac{\mu_2}{2}\langle[\BOm_1,Q_0],\DD_1\rangle\\
&-\frac12\frac{\ud}{\ud t}\langle Q^{\top}_1,\ML(Q^{\top}_1)\rangle+\delta\|\nabla\vv_1\|^2_{L^2}+C_{\delta}\|Q^{\top}_1\|^2_{H^1}.
\end{align*}

For term $I_5$, using the equation (\ref{hilbert-ex1}) and basic properties of the projection operator $\MP^{in}$, and integration by parts, we deduce that
\begin{align*}
I_5=&-J\langle\dot{\vv}_1\cdot\nabla Q_0,\MP^{in}(\dot{Q}^{\top}_1)\rangle\\
=&-J\frac{\ud}{\ud t}\langle\vv_1\cdot\nabla Q_0,\MP^{in}(\dot{Q}^{\top}_1)\rangle
+J\langle\vv_1\cdot\dot{\overline{\nabla Q_0}},\MP^{in}(\dot{Q}^{\top}_1)\rangle\\
&+J\Big\langle\vv_1\cdot\nabla Q_0,(\partial_t+\vv_0\cdot\nabla)\MP^{in}(\dot{Q}^{\top}_1)\Big\rangle\\
\leq&-J\frac{\ud}{\ud t}\langle\vv_1\cdot\nabla Q_0,\MP^{in}(\dot{Q}^{\top}_1)\rangle
+J\langle\MP^{in}(\vv_1\cdot\nabla Q_0),\ddot{Q}^{\top}_1\rangle
+C(\|\vv_1\|^2_{L^2}+\|\dot{Q}^{\top}_1\|^2_{L^2})\\
=&-J\frac{\ud}{\ud t}\langle\vv_1\cdot\nabla Q_0,\MP^{in}(\dot{Q}^{\top}_1)\rangle
-\mu_1\langle\MP^{in}(\vv_1\cdot\nabla Q_0),\dot{Q}^{\top}_1\rangle
-\langle\vv_1\cdot\nabla Q_0,\ML(Q^{\top}_1)\rangle\\
&-\Big\langle\vv_1\cdot\nabla Q_0,\frac{\mu_2}{2}\DD_1+\mu_1[\BOm_1,Q_0]\Big\rangle
-J\langle\vv_1\cdot\nabla Q_0,\dot{\vv}_1\cdot\nabla Q_0\rangle\\
&+\Big\langle\MP^{in}(\vv_1\cdot\nabla Q_0),L(\dot{Q}^{\top}_1)+L(Q^{\top}_1,\vv_1)+R\Big\rangle
+C(\|\vv_1\|^2_{L^2}+\|\dot{Q}^{\top}_1\|^2_{L^2})\\
\leq&-J\frac{\ud}{\ud t}\langle\vv_1\cdot\nabla Q_0,\MP^{in}(\dot{Q}^{\top}_1)\rangle
-\frac{J}{2}\frac{\ud}{\ud t}\|\vv_1\cdot\nabla Q_0\|^2_{L^2}\\
&+\delta\|\nabla\vv_1\|^2_{L^2}+C_{\delta}(\|\vv_1\|^2_{L^2}+\|\dot{Q}^{\top}_1\|^2_{L^2}
+\|Q^{\top}_1\|^2_{H^1}+\|R\|^2_{L^2}),
\end{align*}
where $\dot{\overline{\nabla Q_0}}=(\partial_t+\vv_0\cdot\nabla)\nabla Q_0$ and we have utilized the following estimate
\begin{align*}
-J\langle\vv_1\cdot\nabla Q_0,\dot{\vv}_1\cdot\nabla Q_0\rangle
=&-\frac{J}{2}\frac{\ud}{\ud t}\|\vv_1\cdot\nabla Q_0\|^2_{L^2}
+J\langle\vv_1\cdot\nabla Q_0,\vv_1\cdot\dot{\overline{\nabla Q_0}}\rangle\\
\leq&-\frac{J}{2}\frac{\ud}{\ud t}\|\vv_1\cdot\nabla Q_0\|^2_{L^2}+C\|\vv_1\|^2_{L^2}.
\end{align*}
For terms $I_6$ and $I_7$, we have
\begin{align*}
I_6+I_7\leq&\delta\|\nabla\vv_1\|^2_{L^2}+C_{\delta}(\|\vv_1\|^2_{L^2}+\|\dot{Q}^{\top}_1\|^2_{L^2}
+\|Q^{\top}_1\|^2_{H^1}+\|R\|^2_{L^2}).
\end{align*}

Putting all the above estimates together and using Lemma \ref{diss:relation}, we obtain that
\begin{align}\label{Qv1-energy}
&\frac12\frac{\ud}{\ud t}\int_{\BR}
\bigg(|\vv_1|^2+J\Big(|\dot{Q}^{\top}_1|^2+|\vv_1\cdot\nabla Q_0|^2\Big)
+Q^{\top}_1:\ML(Q^{\top}_1)\bigg)\ud\xx\nonumber\\
&\quad+\frac{\ud}{\ud t}\langle\vv_1\cdot\nabla Q_0,\MP^{in}(\dot{Q}^{\top}_1)\rangle\nonumber\\
&\leq-\tilde{\beta}_1\|\nn\nn:\DD_1\|^2_{L^2}
-\tilde{\beta}_2\|\DD_1\|^2_{L^2}-\tilde{\beta}_3\|\nn\cdot\DD_1\|^2_{L^2}-5\delta\|\nabla\vv_1\|^2_{L^2}
\nonumber\\
&\quad+4\delta\|\nabla\vv_1\|^2_{L^2}+C_{\delta}(\|\vv_1\|^2_{L^2}+\|\dot{Q}^{\top}_1\|^2_{L^2}
+\|Q^{\top}_1\|^2_{H^1}+\|R\|^2_{L^2})\nonumber\\
&\leq-\delta\|\nabla\vv_1\|^2_{L^2}+C_{\delta}(\|\vv_1\|^2_{L^2}+\|\dot{Q}^{\top}_1\|^2_{L^2}
+\|Q^{\top}_1\|^2_{H^1}+\|R\|^2_{L^2}),
\end{align}
where the coefficients $\tilde{\beta}_i(i=1,2,3)$ are given by
\begin{align}\label{tilde-beta}
\left\{
\begin{aligned}
&\tilde{\beta}_1=\beta_1s^2,\quad\tilde{\beta}_2=\beta_4-5\delta-\frac{s(\beta_5+\beta_6)}{3}
+\frac29\beta_7s^2-\frac{\mu^2_2}{4\mu_1},\\
&\tilde{\beta}_3=s(\beta_5+\beta_6)+\frac23\beta_7s^2,
\end{aligned}
\right.
\end{align}
and $\delta>0$ is small enough, such that $\tilde{\beta}_1, \tilde{\beta}_2, \tilde{\beta}_3$ satisfy the relation (\ref{diss:coeff})
(notice that  (\ref{diss:coeff}) holds with strictly positive sign when $\delta=0$). Notice that
\begin{align*}
&  |\vv_1|^2+J\Big(|\dot{Q}^{\top}_1|^2+|\vv_1\cdot\nabla Q_0|^2\Big)+2J(\vv_1\cdot\nabla Q_0):\MP^{in}(\dot{Q}^{\top}_1)\\
  &= |\vv_1|^2+J\Big(|\MP^{in}(\dot{Q}^{\top}_1)+\vv_1\cdot\nabla Q_0|^2+|\MP^{out}(\dot{Q}^{\top}_1)|^2\Big)\\
  &\ge\frac12 |\vv_1|^2+ C(|\nabla Q_0|)|\dot{Q}^{\top}_1|^2.
\end{align*}
Therefore, combining (\ref{Qtop1-L22}) and (\ref{Qv1-energy}), and choosing suitable $M>0$, such that
\begin{align*}
 M( \frac12 |\vv_1|^2+ C(|\nabla Q_0|)|\dot{Q}^{\top}_1|^2)+J\dot{Q}^{\top}_1:Q^{\top}_1 +J(\vv_1\cdot\nabla Q_0): Q^{\top}_1+\frac{\mu_1}{2}|Q^{\top}_1|^2\\
 \ge C(\|\nabla Q_0\|_{L^\infty})(|\vv_1|^2+|\dot{Q}^{\top}_1|^2+|Q^{\top}_1|^2),
 \end{align*}
we obtain the following energy estimate
\begin{align*}
\frac{\ud}{\ud t}\ME(t)\leq C(\|\nabla Q_0\|_{L^\infty})(\ME(t)+\|R(t)\|_{L^2}).
\end{align*}

The estimate of the higher-order derivative for $(\vv_1, Q_1)$ can be also established by introducing a similar energy functional. Therefore, the solution $(\vv_1,Q_1)$ is uniquely determined. In a similar argument, we can solve $(\vv_2, Q_2)$ and $Q_3$ by (\ref{eq:expansion2-ve1})-(\ref{eq:expansion2-ve2}). Here we omit the details.

\section{The estimate for the remainder}\label{estim-remainder}
The main task of this section is to derive the remainder system and the uniform estimates for the remainder. The previous Proposition \ref{prop:Hilbert} tells us that $(\vv_i,\partial_tQ_i,\nabla Q_i)\in L^{\infty}([0,T];H^{k-4i})$ for $i=0,1,2$ and $Q_3\in L^{\infty}([0,T];H^{k-11})$. Hence, in what follows, $\vv_i$ and $Q_i$ will be treated as known functions. We denote by $C$ a constant depending on $\displaystyle\sum_{i=0}^2\sup_{t\in [0,T]}\|\vv_i(t)\|_{H^{k-4i}}$ and $\displaystyle\sum_{i=0}^3\sup_{t\in [0,T]}\|Q_i(t)\|_{H^{k+1-4i}}$, and independent of $\ve$.
\subsection{The system for the remainder}
Recalling the Hilbert expansions (\ref{Qvar})-(\ref{vvar}), then we have
\begin{align}\label{QRvR}
Q_R={\ve^{-3}}(Q^{\ve}-\widetilde{Q}),\quad \vv_R={\ve^{-3}}(\vv^{\ve}-\tilde{\vv}),
\end{align}
where $Q_R$ and $\vv_R$ depend on $\ve$. In order to derive the system of the remainder (\ref{QRvR}), we denote
\begin{align*}
\widetilde{\DD}=&\DD_0+\ve\DD_1+\ve^2\DD_2,\quad\widetilde{\BOm}=\BOm_0+\ve\BOm+\ve^2\BOm_2, \\
\dot{Q}_R=&(\partial_t+\tilde{\vv}\cdot\nabla)Q_R,\quad
\ddot{Q}_R=(\partial_t+\tilde{\vv}\cdot\nabla)\dot{Q}_R.
\end{align*}
From (\ref{CJQ-ve})-(\ref{TR}) and the definitions of $\HH_i(i=0,1,2)$, the molecular field $\HH(Q^{\ve})$ can be expanded into
\begin{align*}
\HH(Q^{\ve})=-{\ve}^{-1}\MT(Q^{\ve})-\ML(Q^{\ve})=-\HH_0-\ve\HH_1-\ve^2\HH_2-\ve^2\HH_R-\ve^3\MT^{\ve}_R,
\end{align*}
where $\HH_R=\MH^{\ve}_{\nn}(Q_R)\eqdefa\MH_{\nn}(Q_R)+\ve\ML(Q_R)$.

Therefore, from (\ref{eq:QS-general-ve1})-(\ref{eq:QS-general-ve3}) and (\ref{expan:Q0-1})-(\ref{eq:expansion2-ve3}), the system for the remainder can be derived as follows:
\begin{align}
J\ddot{Q}_R+\mu_1\dot{Q}_R=&-{\ve}^{-1}\MH^{\ve}_{\nn}(Q_R)
-\frac{\mu_2}{2}\DD_R+\mu_1[\BOm_R,Q_0]+\FF_R+\widetilde{\FF}_R,\label{remainder-Q-R}\\
\frac{\partial\vv_R}{\partial t}+\tilde{\vv}\cdot\nabla\vv_R=&-\nabla p_R+\nabla\cdot\Big(\beta_1Q_0(Q_0:\DD_R)+\beta_4\DD_R+\beta_5Q_0\cdot\DD_R\nonumber\\
&+\beta_6\DD_R\cdot Q_0+\beta_7(\DD_R\cdot Q^2_0+Q^2_0\cdot\DD_R)\nonumber\\
&+\frac{\mu_2}{2}(\dot{Q}_R-[\BOm_R,Q_0])+\mu_1\big[Q_0,(\dot{Q}_R-[\BOm_R,Q_0])\big]\Big)\nonumber\\
&+\nabla\cdot\GG_R+\GG'_R,
\label{remainder-v-R}\\
\nabla\cdot\vv_R=&0.\label{remainder-div}
\end{align}
The term $\FF_R$ is given by
\begin{align*}
\FF_R=\FF_1+\FF_2+\FF_3+\FF_4+\FF_5,
\end{align*}
where $\FF_1$ is independent of $(\vv_R,Q_R)$,
\begin{align*}
\FF_1=&J\Big(-\partial^2_tQ_3-2\vv_0\cdot\nabla \partial_tQ_3-2\vv_1\cdot\nabla\partial_t(Q_2+\ve Q_3)-2\vv_2\cdot\nabla\partial_t\widehat{Q}^{\ve}-\partial_t\vv_0\cdot\nabla Q_3\\
&-\partial_t\vv_1\cdot\nabla(Q_2+\ve Q_3)-\partial_t\vv_2\cdot\nabla\widehat{Q}^{\ve}
-\sum_{i+j+k\geq3}\ve^{i+j+k-3}\vv_i\cdot\nabla(\vv_j\cdot\nabla Q_k)\Big)\\
&+\mu_1\Big(-\partial_tQ_3-\vv_0\cdot\nabla Q_3-\vv_1\cdot\nabla(Q_2+\ve Q_3)-\vv_2\cdot\nabla\widehat{Q}^{\ve}\Big)\\
&+\mu_1\Big(\sum_{i+j\geq3}\ve^{i+j-3}(\BOm_i\cdot Q_j-Q_j\cdot\BOm_i)\Big)-\BB^{\ve}-\ML(Q_3)\\
\equiv&-J\FF_{11}-\mu_1\FF_{12}-\BB^{\ve}-\ML(Q_3),
\end{align*}
and $\FF_2, \FF_3$ linearly depend on $(\vv_R,Q_R)$,
\begin{align*}
\FF_2=&\mu_1(\widetilde{\BOm}\cdot Q_R-Q_R\cdot\widetilde{\BOm})
-\Big(-b\MB(\widehat{Q}^{\ve},Q_R)+c\MC(Q_R,\widehat{Q}^{\ve},Q_0)
+\frac{c}{2}\ve\MC(Q_R,\widehat{Q}^{\ve},\widehat{Q}^{\ve})\Big),\\
\FF_3=&
-J\vv_R\cdot\nabla(\partial_t\widetilde{Q}+\tilde{\vv}\cdot\nabla\widetilde{Q})
+\mu_1\Big(-\vv_R\cdot\nabla\widetilde{Q}+\ve\BOm_R\cdot\widehat{Q}^{\ve}-\ve\widehat{Q}^{\ve}\cdot\BOm_R\Big)\\
\equiv&-J\FF_{31}-\mu_1\FF_{32},
\end{align*}
and $\FF_4, \FF_5$ nonlinearly depend on $(\vv_R, Q_R)$,
\begin{align*}
\FF_4=&-\ve^3J
\vv_R\cdot\nabla (\vv_R\cdot\nabla\widetilde{Q})
+\ve^3\mu_1\Big(-\vv_R\cdot\nabla Q_R+\BOm_R\cdot Q_R-Q_R\cdot\BOm_R\Big)\\
\equiv&-\ve^3J\FF_{41}-\ve^3\mu_1\FF_{42},\\
\FF_5=&-\Big(-\frac{b}{2}\ve^2\MB(Q_R,Q_R)+c\ve^2\MC(Q_R,Q_R,\widetilde{Q})
+c\ve^5\MC(Q_R,Q_R,Q_R)\Big).
\end{align*}
The term $\widetilde{\FF}_R$, including the derivative term with respect to time $t$, is given by
\begin{align*}
\widetilde{\FF}_R=&-J(\partial_t+\tilde{\vv}\cdot\nabla)(\vv_R\cdot\nabla\widetilde{Q})\\
&+\ve^3J\Big(-(\partial_t+\tilde{\vv}\cdot\nabla)(\vv_R\cdot\nabla Q_R)-\vv_R\cdot\nabla \dot{Q}_R
-\ve^3\vv_R\cdot\nabla(\vv_R\cdot\nabla Q_R)\Big).
\end{align*}

On the other hand, the term $\GG'_R$ takes the following form
\begin{align*}
\GG'_R=-\vv_1\cdot\nabla\vv_2-\vv_2\cdot\nabla\vv_1-\ve\vv_2\cdot\nabla\vv_2
-\vv_R\cdot\nabla\tilde{\vv}-\ve^3\vv_R\cdot\nabla\vv_R.
\end{align*}
Similarly, the term $\GG_R$ can be written as
\begin{align*}
\GG_R=\GG_1+\GG_2+\GG_3,
\end{align*}
where $\GG_1$ is given by
\begin{align*}
\GG_1=&\sum_{i+j+k\geq3}\ve^{i+j+k-3}\bigg(\beta_1Q_i(Q_j:\DD_k)
+\beta_7\Big(\DD_i\cdot Q_j\cdot Q_k+Q_i\cdot Q_j\cdot\DD_k\Big)\bigg)\\
&+\sum_{i+j\geq3}\ve^{i+j-3}\Big(\beta_5\DD_i\cdot Q_j+\beta_6Q_i\cdot\DD_j+\sigma^d(Q_i,Q_j)\Big)
+\frac{\mu_2}{2}\FF_{12}\\
&+\mu_1\bigg(\sum_{i+j\geq3}\ve^{i+j-3}[Q_i,\partial_tQ_j]
+\sum_{i+j+k\geq3}\ve^{i+j+k-3}\Big[Q_i,\big(\vv_j\cdot\nabla Q_k-[\BOm_j,Q_k]\big)\Big]\bigg),
\end{align*}
and $\GG_2, \GG_3$ are given by
\begin{align*}
\GG_2=&\beta_1\Big(\widetilde{Q}(Q_R:\widetilde{\DD})+Q_R(\widetilde{Q}:\widetilde{\DD})
+\ve Q_0(\widehat{Q}^{\ve}:\DD_R)+\ve\widehat{Q}^{\ve}(\widetilde{Q}:\DD_R)\Big)\\
&+\beta_5(\widetilde{\DD}\cdot Q_R+\ve\DD_R\cdot\widehat{Q}^{\ve})
+\beta_6(\ve\widehat{Q}^{\ve}\cdot\DD_R+Q_R\cdot\widetilde{\DD})\\
&+\beta_7\Big(\widetilde{\DD}\cdot Q_R\cdot\widetilde{Q}
+\widetilde{\DD}\cdot\widetilde{Q}\cdot Q_R+\ve\DD_R\cdot\widehat{Q}^{\ve}\cdot\widetilde{Q}
+\ve\DD_R\cdot Q_0\cdot \widehat{Q}^{\ve}\Big)\\
&+\beta_7\Big(\widetilde{Q}\cdot Q_R\cdot\widetilde{\DD}
+Q_R\cdot\widetilde{Q}\cdot\widetilde{\DD}+\ve\widehat{Q}^{\ve}\cdot Q_0\cdot\DD_R
+\ve\widetilde{Q}\cdot \widehat{Q}^{\ve}\cdot\DD_R\Big)\\
&+\frac{\mu_2}{2}(\FF_{32}-[\widetilde{\BOm},Q_R])
+\mu_1\Big[Q_R,\big(\partial_t\widetilde{Q}+\tilde{\vv}\cdot\nabla\widetilde{Q}-[\widetilde{\BOm},\widetilde{Q}]\big)\Big]\\
&+\mu_1\Big[\widetilde{Q},\big(\vv_R\cdot\nabla\widetilde{Q}-[\widetilde{\BOm},Q_R]\big)\Big]
+\mu_1\big[\ve\widehat{Q}^{\ve},(\dot{Q}_R-[\BOm_R,Q_0])\big]\\
&-\mu_1\big[\widetilde{Q},[\BOm_R,\ve\widehat{Q}^{\ve}]\big]
+\sigma^d(\widetilde{Q},Q_R)+\sigma^d(Q_R,\widetilde{Q}),\\
\GG_3=&\ve^3\bigg(\beta_1\Big(\widetilde{Q}(Q_R:\DD_R)+Q_R(\widetilde{Q}:\DD_R)+Q_R(Q_R:\widetilde{\DD})
+\ve^3Q_R(Q_R:\DD_R)\Big)\\
&+\beta_7\Big(\widetilde{\DD}\cdot Q_R\cdot Q_R+\DD_R\cdot\widetilde{Q}\cdot Q_R
+\DD_R\cdot Q_R\cdot\widetilde{Q}+\ve^3\DD_R\cdot Q_R\cdot Q_R\Big)\\
&+\beta_7\Big(\widetilde{Q}\cdot Q_R\cdot\DD_R+Q_R\cdot\widetilde{Q}\cdot\DD_R
+Q_R\cdot Q_R\cdot\widetilde{\DD}+\ve^3 Q_R\cdot Q_R\cdot\DD_R\Big)\\
&+\beta_5\DD_R\cdot Q_R+\beta_6Q_R\cdot\DD_R
+\frac{\mu_2}{2}\FF_{42}
+\mu_1\Big[\widetilde{Q},\big(\vv_R\cdot\nabla Q_R-[\BOm_R,Q_R]\big)\Big]\\
&+\mu_1\Big[Q_R,\big(\dot{Q}_R+\vv_R\cdot\nabla\widetilde{Q}-[\widetilde{\BOm},Q_R]-[\BOm_R,\widetilde{Q}]\big)\Big]\\
&+\mu_1\ve^3\Big[Q_R,\big(\vv_R\cdot\nabla Q_R-[\BOm_R,Q_R]\big)\Big]
+\sigma^d(Q_R,Q_R)\bigg).
\end{align*}

\subsection{Uniform estimates for the remainder}
In this subsection, we derive the uniform estimates for the remainder. We assume that $(\vv_R,Q_R)$ is a smooth solution of the remainder system (\ref{remainder-Q-R})-(\ref{remainder-div}) and introduce the following energy functional:
\begin{align}
\Ef(t)\eqdefa&\int_{\BR}\Big(|\vv_R|^2+|\dot{Q}_R|^2
+|Q_R|^2+{\ve}^{-1}\MH^{\ve}_{\nn}(Q_R):Q_R\Big)\nonumber\\
&+\ve^2\Big(|\nabla\vv_R|^2+|\partial_i\dot{Q}_R|^2
+{\ve}^{-1}\MH^{\ve}(\partial_iQ_R):\partial_iQ_R\Big)\nonumber\\
&+\ve^4\Big(|\Delta\vv_R|^2+|\Delta\dot{Q}_R|^2
+{\ve}^{-1}\MH^{\ve}(\Delta Q_R):\Delta Q_R\Big)\ud\xx,
\label{Ef:energy}\\
\Ff(t)\eqdefa&\int_{\BR}\delta\Big(|\nabla\vv_R|^2
+\ve^2|\Delta\vv_R|^2+\ve^4|\nabla\Delta\vv_R|^2\Big)\ud\xx.\label{Ff:energy}
\end{align}

By using the definitions of $\Ef$ and $\Ff$, we can immediately obtain that
\begin{lemma}\label{lem:energy}
The following estimates hold
\begin{align*}
\|(\ve\nabla^2Q_R,\ve^2\nabla^3Q_R)\|_{L^2}
  +\|(\vv_R,\ve\nabla\vv_R,\ve^2\nabla^2\vv_R)\|_{L^2}\leq &~C\Ef^{\frac12},\\
\|Q_R\|_{H^1}+\|(\dot{Q}_R,\ve\nabla\dot{Q}_R,\ve^2\Delta\dot{Q}_R)\|_{L^2}\leq &~C\Ef^{\frac12},\\
\|(\nabla\vv_R,\ve\nabla^2\vv_R,\ve^2\nabla^3\vv_R)\|_{L^2}
    \leq&~ C\Ff^{\frac12}.
\end{align*}
\end{lemma}

In order to establish the estimates of the remainder terms $(\FF_R,\GG_R)$ and $\GG'_R$, it is desirable to utilize the following inequality:
\begin{align}\label{simple-ineq}
\|fg\|_{H^k}\leq C\|f\|_{H^2}\|g\|_{H^k},~~k=0,1,2.
\end{align}
\begin{lemma}\label{FFR-goodterm}
For the remainder term $\FF_R$, the following estimate holds
\begin{align*}
\|(\FF_R,\ve\nabla\FF_R,\ve^2\Delta\FF_R)\|_{L^2}\leq C(1+\Ef^{\frac12}+\ve\Ef+\ve^3\Ef^{\frac32}+\ve\Ff^{\frac12}+\ve\Ef^{\frac12}\Ff^{\frac12}).
\end{align*}
\end{lemma}
\begin{proof}
Applying Lemma \ref{lem:energy}, we see at once that
\begin{align}
&\|(\FF_1,\ve\nabla\FF_1,\ve^2\Delta\FF_1)\|_{L^2}\leq C,\nonumber\\
&\|(\FF_2,\ve\nabla\FF_2,\ve^2\Delta\FF_2)\|_{L^2}\leq C\Ef^{\frac12},\nonumber\\
&\|(\FF_3,\ve\nabla\FF_3,\ve^2\Delta\FF_3)\|_{L^2}\leq C(\Ef^{\frac12}+\ve\Ff^{\frac{1}{2}}).\nonumber
\end{align}
Using the inequality (\ref{simple-ineq}), we have
\begin{align*}
\|\FF_4\|_{H^k}\leq&
 C\ve\|\vv_R\|_{H^k}\big(\|\ve^2\nabla\vv_R\|_{H^2}
 +\|\ve^2\vv_R\|_{H^2}\big)\\
&+C\ve\|\vv_R\|_{H^k}\|\ve^2\nabla Q_R\|_{H^2}+C\ve^2\|\ve Q_R\|_{H^2}\|\nabla\vv_R\|_{H^k},
\end{align*}
which implies
\begin{align*}
\|(\FF_4,\ve\nabla\FF_4,\ve^2\Delta\FF_4)\|_{L^2}\leq C\ve(\Ef+\Ef^{\frac12}\Ff^{\frac12}).
\end{align*}
Similarly, from Lemma \ref{lem:energy} and (\ref{simple-ineq}) again, we can infer that
\begin{align*}
\|(\FF_5,\ve\nabla\FF_5,\ve^2\Delta\FF_5)\|_{L^2}\leq C\ve(\Ef+\ve^2\Ef^{\frac{3}{2}}).
\end{align*}
The proof is finished.
\end{proof}

\begin{lemma}\label{GG-lem}
For the remainder term $\GG_R$, the following estimates hold
\begin{align*}
&\|(\GG_R,\ve\nabla\GG_R,\ve^2\Delta\GG_R)\|_{L^2}\leq C(1+\Ef^{\frac12}+\ve\Ef+\ve^3\Ef^{\frac32}+\ve\Ff^{\frac12}+\ve^2\Ef^{\frac12}\Ff^{\frac12}
+\ve^4\Ef\Ff^{\frac12}),\\
&\|(\GG'_R,\ve\nabla\GG'_R,\ve^2\Delta\GG'_R)\|_{L^2}\leq C(1+\Ef^{\frac12}+\Ff^{\frac12}+\ve\Ef^{\frac12}\Ff^{\frac12}).
\end{align*}
\end{lemma}
\begin{proof}
It is straightforward to show from Lemma \ref{lem:energy} that
\begin{align*}
&\|(\GG_1,\ve\nabla\GG_1,\ve^2\Delta\GG_1)\|_{L^2}\leq C,\\
&\|(\GG_2,\ve\nabla\GG_2,\ve^2\Delta\GG_2)\|_{L^2}\leq C(\Ef^{\frac12}+\ve\Ff^{\frac12}).
\end{align*}
By the inequality (\ref{simple-ineq}), we obtain
\begin{align*}
\|\GG_3\|_{H^k}\leq& C\ve^2\|\ve Q_R\|_{H^2}\Big(\|\nabla\vv_R\|_{H^k}+\|Q_R\|_{H^k}+\ve^3\|Q_R:\DD_R\|_{H^k}
+\|\dot{Q}_R\|_{H^k}\\
&+\|\vv_R\|_{H^k}+\ve^3(\|\vv_R\cdot\nabla Q_R\|_{H^k}+\|\BOm_R\cdot Q_R\|_{H^k})\Big)\\
&+C\ve\|\ve^2\nabla Q_R\|_{H^2}(\|\vv_R\|_{H^k}+\|\nabla Q_R\|_{H^k}),
\end{align*}
which gives
\begin{align*}
\|(\GG_3,\ve\nabla\GG_3,\ve^2\Delta\GG_3)\|_{L^2}\leq C(\ve\Ef+\ve^3\Ef^{\frac32}+\ve^2\Ef^{\frac12}\Ff^{\frac12}
+\ve^4\Ef\Ff^{\frac12}).
\end{align*}
Then the conclusion follows.
\end{proof}

We point out that for $m=0,1,2$, highly singular terms $\langle\frac{1}{\ve}\MH^{\ve}_{\nn}(\partial^m_iQ_R),\partial^m_iQ_R\rangle$ in (\ref{Ef:energy}) come from the $L^2$-inner products $\langle\frac{1}{\ve}\partial^m_i\MH^{\ve}_{\nn}(Q_R),\partial^m_i\dot{Q}_R\rangle$. Fortunately, the following Lemma \ref{lem:singular} will play a crucial role in dealing with these singular terms, which makes the whole machinery work.
\begin{lemma}\label{lem:singular}
Assume that $(\vv_R,Q_R)$ is a smooth solution of the remainder system (\ref{remainder-Q-R})-(\ref{remainder-div}). Then for any $\delta>0$, there exists a constant $C$ depending on $\nn, \nabla_{t,x}\nn, \tilde{\vv}$ and $\widetilde{Q}$,
such that
\begin{align}
 \frac{1}{\ve}\Big\langle\dot{\overline{\nn\nn}}\cdot Q_R,Q_R\Big\rangle
&\leq-J\frac{\ud}{\ud t}\Big\langle\MH^{-1}_{\nn}(\dot{\overline{\nn\nn}}\cdot Q_R^{\top}),\dot{Q}_R+\vv_R\cdot\nabla Q^{\ve}\Big\rangle\nonumber\\
&\quad+C(1+\Ef+\ve^2\Ef^2)+(\delta+C\ve^2\Ef)\Ff,\label{lem:singular-1}\\
\frac{1}{\ve}\Big\langle(Q_R:\dot{\overline{\nn\nn}})\nn\nn,Q_R\Big\rangle
&\leq-J\frac{\ud}{\ud t}\Big\langle\MH^{-1}_{\nn}\big(Q_R^{\top}:\dot{\overline{\nn\nn}}(\nn\nn-\frac13\II)\big),
\dot{Q}_R+\vv_R\cdot\nabla Q^{\ve}\Big\rangle\nonumber\\
&\quad+C(1+\Ef+\ve^2\Ef^2)+(\delta+C\ve^2\Ef)\Ff,\label{lem:singular-2}
\end{align}
where $\dot{\overline{\nn\nn}}\eqdefa(\partial_t+\tilde{\vv}\cdot\nabla)(\nn\nn)$, $Q^{\ve}=\widetilde{Q}+\ve^3Q_R$ and $\MH^{-1}_{\nn}$ is defined by (\ref{HM-inverse}).
Moreover, for $m=1,2$, the following estimates hold
\begin{align}
\ve^{2m-1}\Big\langle\dot{\overline{\nn\nn}}\cdot \partial^m_iQ_R,\partial^m_iQ_R\Big\rangle
\leq&C\Ef,\label{m-nn-mQR-1}\\
\ve^{2m-1}\Big\langle(\partial^m_iQ_R:\dot{\overline{\nn\nn}})\nn\nn,\partial^m_iQ_R\Big\rangle
\leq&C\Ef,\label{m-nn-mQR-2}
\end{align}
where $\partial^m_i$ represents the $m$-th order partial derivative operator with respect to the component $x_i$.
\end{lemma}
\begin{proof}
We assume $Q_R=Q_R^{\top}+Q_R^{\bot}$ with $Q_R^{\top}\in\text{Ker}\MH_{\nn}$ and $Q_R^{\bot}\in(\text{Ker}\MH_{\nn})^{\perp}$. Then we obtain
\begin{align*}
\langle\dot{\overline{\nn\nn}}\cdot Q_R,Q_R\rangle
=&\langle\dot{\overline{\nn\nn}}\cdot Q_R^{\top},Q_R^{\top}\rangle
+{2}\langle\dot{\overline{\nn\nn}}\cdot Q_R^{\top},Q_R^{\bot}\rangle
+\langle\dot{\overline{\nn\nn}}\cdot Q_R^{\bot},Q_R^{\bot}\rangle.
\end{align*}
Note that there holds $\dot{\overline{\nn\nn}}\cdot Q_R^{\top}\in (\text{Ker}\MH_{\nn})^{\perp}$ since
$\dot{\overline{\nn\nn}}=\dot{\nn}\nn+\nn\dot{\nn}\in\text{Ker}\MH_{\nn}$.
Then we have $\frac{1}{\ve}\langle\dot{\overline{\nn\nn}}\cdot Q_R^{\top},Q_R^{\top}\rangle=0$. Using Proposition \ref{linearized-oper-prop}, it follows that
\begin{align*}
\frac{1}{\ve}\langle\dot{\overline{\nn\nn}}\cdot Q_R^{\bot},Q_R^{\bot}\rangle
&\leq C\frac{1}{\ve}\|Q_R^{\bot}\|^2_{L^2}
\leq C\frac{1}{\ve}\langle\MH_{\nn}(Q_R),Q_R\rangle\\
&\leq C\Big(\frac{1}{\ve}\langle\MH^{\ve}_{\nn}(Q_R),Q_R\rangle
-\langle\ML(Q_R),Q_R\rangle\Big)\leq C\Ef.
\end{align*}
It can be seen from Proposition \ref{linearized-oper-prop} that
\begin{align*}
\frac{1}{\ve}\langle\dot{\overline{\nn\nn}}\cdot Q_R^{\top},Q_R^{\bot}\rangle
&=\Big\langle\MH^{-1}_{\nn}(\dot{\overline{\nn\nn}}\cdot Q_R^{\top}),\frac{1}{\ve}\MH_{\nn}(Q_R)\Big\rangle\\
&=\Big\langle\MH^{-1}_{\nn}(\dot{\overline{\nn\nn}}\cdot Q_R^{\top}),\frac{1}{\ve}\MH^{\ve}_{\nn}(Q_R)\Big\rangle
-\Big\langle\MH^{-1}_{\nn}(\dot{\overline{\nn\nn}}\cdot Q_R^{\top}),\ML(Q_R)\Big\rangle\\
&\leq
\Big\langle\MH^{-1}_{\nn}(\dot{\overline{\nn\nn}}\cdot Q_R^{\top}),\frac{1}{\ve}\MH^{\ve}_{\nn}(Q_R)\Big\rangle
+C(\|\nabla Q_R\|^2_{L^2}+\|Q_R\|^2_{L^2}).
\end{align*}
From the equation (\ref{remainder-Q-R}), we have
\begin{align*}
&\Big\langle\MH^{-1}_{\nn}(\dot{\overline{\nn\nn}}\cdot Q_R^{\top}),\frac{1}{\ve}\MH^{\ve}_{\nn}(Q_R)\Big\rangle\\
&=\underbrace{-J\Big\langle\MH^{-1}_{\nn}(\dot{\overline{\nn\nn}}\cdot Q_R^{\top}),\ddot{Q}_R\Big\rangle}_{\mathcal{M}_1}
\underbrace{-\mu_1\Big\langle\MH^{-1}_{\nn}(\dot{\overline{\nn\nn}}\cdot Q_R^{\top}),\dot{Q}_R\Big\rangle}_{\mathcal{M}_2}\\
&\quad
\underbrace{-\frac{\mu_2}{2}\Big\langle\MH^{-1}_{\nn}(\dot{\overline{\nn\nn}}\cdot Q_R^{\top}),\DD_R\Big\rangle}_{\mathcal{M}_3}
+\underbrace{\mu_1\Big\langle\MH^{-1}_{\nn}(\dot{\overline{\nn\nn}}\cdot Q_R^{\top}),[\BOm_R,Q_0]\Big\rangle}_{\mathcal{M}_4}\\
&\quad
+\underbrace{\Big\langle\MH^{-1}_{\nn}(\dot{\overline{\nn\nn}}\cdot Q_R^{\top}),\FF_R\Big\rangle}_{\mathcal{M}_5}
+\underbrace{\Big\langle\MH^{-1}_{\nn}(\dot{\overline{\nn\nn}}\cdot Q_R^{\top}),\widetilde{\FF}_R\Big\rangle}_{\mathcal{M}_6}.
\end{align*}
Using integration by parts, we get
\begin{align*}
\mathcal{M}_1=&-J\frac{\ud}{\ud t}\Big\langle\MH^{-1}_{\nn}(\dot{\overline{\nn\nn}}\cdot Q_R^{\top}),\dot{Q}_R\Big\rangle
+J\Big\langle(\partial_t+\tilde{\vv}\cdot\nabla)\MH^{-1}_{\nn}(\dot{\overline{\nn\nn}}\cdot Q_R^{\top}),\dot{Q}_R\Big\rangle\\
\leq&-J\frac{\ud}{\ud t}\Big\langle\MH^{-1}_{\nn}(\dot{\overline{\nn\nn}}\cdot Q_R^{\top}),\dot{Q}_R\Big\rangle
+C(\|Q_R\|^2_{L^2}+\|\dot{Q}_R\|^2_{L^2}).
\end{align*}
From Lemma \ref{lem:energy}, we can easily estimate that
\begin{align*}
\mathcal{M}_2\leq&C\|Q_R\|_{L^2}\|\dot{Q}_R\|_{L^2}\leq C\Ef,\\
\mathcal{M}_3+\mathcal{M}_4\leq&C\|Q_R\|_{L^2}\|\nabla\vv_R\|\leq C\Ef^{\frac12}\Ff^{\frac12},\\
\mathcal{M}_5\leq &C\|Q_R\|_{L^2}\|\FF_R\|_{L^2}\leq C\Ef^{\frac12}\|\FF_R\|_{L^2}.
\end{align*}
For the term $\mathcal{M}_6$, we have
\begin{align*}
\mathcal{M}_6=&-J\frac{\ud}{\ud t}\Big\langle\MH^{-1}_{\nn}(\dot{\overline{\nn\nn}}\cdot Q_R^{\top}),
\vv_R\cdot\nabla\widetilde{Q}+\ve^3\vv_R\cdot\nabla Q_R\Big\rangle\\
&+\underbrace{J\Big\langle(\partial_t+\tilde{\vv}\cdot\nabla)\MH^{-1}_{\nn}(\dot{\overline{\nn\nn}}\cdot Q_R^{\top}),
\vv_R\cdot\nabla\widetilde{Q}\Big\rangle}_{\widetilde{\mathcal{M}}_1}\\
&+\underbrace{\ve^3J\Big\langle(\partial_t+\tilde{\vv}\cdot\nabla)\MH^{-1}_{\nn}(\dot{\overline{\nn\nn}}\cdot Q_R^{\top}),
\vv_R\cdot\nabla Q_R\Big\rangle}_{\widetilde{\mathcal{M}}_2}\\
&\underbrace{+\ve^3J\Big\langle(\vv_R\cdot\nabla)\MH^{-1}_{\nn}(\dot{\overline{\nn\nn}}\cdot Q_R^{\top}),
 \dot{Q}_R\Big\rangle}_{\widetilde{\mathcal{M}}_3}\\
&+\underbrace{\ve^6J\Big\langle(\vv_R\cdot\nabla)\MH^{-1}_{\nn}(\dot{\overline{\nn\nn}}\cdot Q_R^{\top}),
\vv_R\cdot\nabla Q_R\Big\rangle}_{\widetilde{\mathcal{M}}_4}.
\end{align*}
Using Lemma \ref{lem:energy}, we can infer that
\begin{align*}
\widetilde{\mathcal{M}}_1\leq&C(\|Q_R\|_{L^2}+\|\dot{Q}\|_{L^2})\|\vv_R\|_{L^2}
\leq C\Ef,\\
\widetilde{\mathcal{M}}_2\leq&C\ve^3(\|Q_R\|_{L^2}+\|\dot{Q}\|_{L^2})\|\vv_R\|_{H^2}\|\nabla Q_R\|_{L^2}
\leq C\ve\Ef^{\frac32},\\
\widetilde{\mathcal{M}}_3\leq&C\ve^3\|\vv_R\|_{H^2}\|Q_R\|_{H^1}\|\dot{Q}_R\|_{L^2}
\leq C\ve\Ef^{\frac32},\\
\widetilde{\mathcal{M}}_4\leq&C\ve^6\|\vv_R\|^2_{H^2}\|Q_R\|^2_{H^1}
\leq C\ve^2\Ef^2.
\end{align*}
Thus, we obtain the following estimate
\begin{align*}
\frac{1}{\ve}\langle\dot{\overline{\nn\nn}}\cdot Q_R,Q_R\rangle
&\leq-J\frac{\ud}{\ud t}\Big\langle\MH^{-1}_{\nn}(\dot{\overline{\nn\nn}}\cdot Q_R^{\top}),\dot{Q}_R\Big\rangle
-J\frac{\ud}{\ud t}\Big\langle\MH^{-1}_{\nn}(\dot{\overline{\nn\nn}}\cdot Q_R^{\top}),
\vv_R\cdot\nabla Q^{\ve}\Big\rangle\\
&\quad
+C(1+\Ef+\ve^2\Ef^2)+(\delta+C\ve^2\Ef)\Ff.
\end{align*}

Similarly, we get
\begin{align*}
\langle Q_R:\dot{\overline{\nn\nn}},Q_R:\nn\nn\rangle
=\langle Q^{\top}_R:\dot{\overline{\nn\nn}},Q^{\bot}_R:\nn\nn\rangle
+\langle Q^{\bot}_R:\dot{\overline{\nn\nn}},Q^{\bot}_R:\nn\nn\rangle.
\end{align*}
Therefore, the analogous argument leads to the second estimate (\ref{lem:singular-2}).

For the case of $m=1$ in (\ref{m-nn-mQR-1}) and (\ref{m-nn-mQR-2}), we first
assume that $\partial_iQ_R=(\partial_iQ_R)^{\top}+(\partial_iQ_R)^{\bot}$ with $(\partial_iQ_R)^{\top}\in\text{Ker}\MH_{\nn}$ and $(\partial_iQ_R)^{\bot}\in(\text{Ker}\MH_{\nn})^{\perp}$.
Then we have
\begin{align}\label{ve-nQR-pQR}
\big\langle\dot{\overline{\nn\nn}}\cdot \partial_iQ_R,\partial_iQ_R\big\rangle
=&
2\big\langle\dot{\overline{\nn\nn}}\cdot (\partial_iQ_R)^{\top},(\partial_iQ_R)^{\bot}\big\rangle
+\big\langle\dot{\overline{\nn\nn}}\cdot (\partial_iQ_R)^{\bot},(\partial_iQ_R)^{\bot}\big\rangle.
\end{align}
By Proposition \ref{linearized-oper-prop}, the third term in (\ref{ve-nQR-pQR}) can be estimated as
\begin{align*}
\ve\big\langle\dot{\overline{\nn\nn}}\cdot (\partial_iQ_R)^{\bot},(\partial_iQ_R)^{\bot}\big\rangle
&\leq C\ve\|(\partial_iQ_R)^{\bot}\|^2_{L^2}
\leq C\ve\langle\MH_{\nn}(\partial_iQ_R),\partial_iQ_R\rangle\\
&\leq C\Big(\ve\langle\MH^{\ve}_{\nn}(\partial_iQ_R),\partial_iQ_R\rangle
-\ve^2\langle\ML(\partial_iQ_R),\partial_iQ_R\rangle\Big)\\
&\leq C\Ef.
\end{align*}
For the second term in (\ref{ve-nQR-pQR}), using Proposition \ref{linearized-oper-prop}, we obtain
\begin{align*}
\ve\big\langle\dot{\overline{\nn\nn}}\cdot (\partial_iQ_R)^{\top},(\partial_iQ_R)^{\bot}\big\rangle
&=\ve\Big\langle\MH^{-1}_{\nn}\big(\dot{\overline{\nn\nn}}\cdot (\partial_iQ_R)^{\bot}\big),\MH_{\nn}(\partial_iQ_R)\Big\rangle\\
&\leq C\ve(\|(\partial_iQ_R)^{\bot}\|^2_{L^2}+\|\partial_iQ_R\|^2_{L^2})\leq C\Ef.
\end{align*}
Likewise, we can prove that
\begin{align*}
\ve\big\langle\partial_iQ_R:\dot{\overline{\nn\nn}},\partial_iQ_R:\nn\nn\big\rangle
\leq C\Ef.
\end{align*}

For the case of $m=2$, we suppose that $\Delta Q_R=(\Delta Q_R)^{\top}+(\Delta Q_R)^{\bot}$ with $(\Delta Q_R)^{\top}\in\text{Ker}\MH_{\nn}$
and $(\Delta Q_R)^{\bot}\in(\text{Ker}\MH_{\nn})^{\perp}$. Adopting an analogous argument yields (\ref{m-nn-mQR-1}) and (\ref{m-nn-mQR-2}) for $m=2$.
\end{proof}

We next deal with the estimates for the remainder term $\widetilde{\FF}_R$. For convenience, the remainder term $\widetilde{\FF}_R$, involving the derivatives with respect to time $t$, is denoted by
\begin{align*}
\widetilde{\FF}_R=&-J(\partial_t+\tilde{\vv}\cdot\nabla)(\vv_R\cdot\nabla\widetilde{Q})\\
&+\ve^3J\Big(-(\partial_t+\tilde{\vv}\cdot\nabla)(\vv_R\cdot\nabla Q_R)-\vv_R\cdot\nabla \dot{Q}_R
-\ve^3\vv_R\cdot\nabla(\vv_R\cdot\nabla Q_R)\Big)\\
\eqdefa&~\widetilde{\FF}_1+\widetilde{\FF}_2.
\end{align*}

\begin{lemma}\label{lem:FR-1}
For the remainder term $\widetilde{\FF}_R$, it follows that
\begin{align}
\langle\widetilde{\FF}_R,Q_R\rangle\leq
-J\frac{\ud}{\ud t}\langle\vv_R\cdot\nabla\widetilde{Q},Q_R\rangle+C(\Ef+\ve^2\Ef^2).
\label{FFR-QR}
\end{align}
\end{lemma}
\begin{proof}
Using integration by parts, it is easy to calculate that
\begin{align*}
\langle\widetilde{\FF}_1,Q_R\rangle=&-J
\Big\langle(\partial_t+\tilde{\vv}\cdot\nabla)(\vv_R\cdot\nabla\widetilde{Q}),Q_R\Big\rangle\\
=&-J\frac{\ud}{\ud t}\langle\vv_R\cdot\nabla\widetilde{Q},Q_R\rangle+J\langle\vv_R\cdot\nabla\widetilde{Q},\dot{Q}_R\rangle\\
\leq&-J\frac{\ud}{\ud t}\langle\vv_R\cdot\nabla\widetilde{Q},Q_R\rangle+C\Ef.
\end{align*}
In virtue of the incompressibility $\nabla\cdot\vv_R=0$, the following fact holds
\begin{align*}
\frac{\ud}{\ud t}\big\langle\vv_R\cdot\nabla Q_R,Q_R\big\rangle=0,
\end{align*}
which combines with Lemma \ref{lem:energy}, we get
\begin{align*}
\langle\widetilde{\FF}_2,Q_R\rangle=&-\ve^3J\Big\langle(\partial_t+\tilde{\vv}\cdot\nabla)(\vv_R\cdot\nabla Q_R),Q_R\Big\rangle
+\ve^3J\langle\vv_R\cdot\nabla Q_R,\dot{Q}_R\rangle\\
&+\ve^6J\langle\vv_R\cdot\nabla Q_R,\vv_R\cdot\nabla Q_R\rangle\\
=&-\ve^3J\frac{\ud}{\ud t}\big\langle\vv_R\cdot\nabla Q_R,Q_R\big\rangle+2\ve^3J\big\langle\vv_R\cdot\nabla Q_R,\dot{Q}_R\big\rangle\\
&+\ve^6J\langle\vv_R\cdot\nabla Q_R,\vv_R\cdot\nabla Q_R\rangle\\
\leq&C\ve^3\|\vv_R\|_{H^2}\|\nabla Q_R\|_{L^2}\|\dot{Q}_R\|_{L^2}
+C\ve^6\|\vv_R\|^2_{L^2}\|\nabla Q_R\|^2_{H^2}\\
\leq&C(\ve\Ef^{\frac32}+\ve^2\Ef^2).
\end{align*}
Consequently, we conclude the proof of the lemma.
\end{proof}

\begin{lemma}\label{lem:FR-2}
For the remainder term $\widetilde{\FF}_R$ and $m=0,1,2$, there holds
\begin{align}
&\ve^{2m}\langle\partial^m_i\widetilde{\FF}_R,\partial^m_i\dot{Q}_R\rangle\nonumber\\
&\leq
-\ve^{2m}\frac{J}{2}\frac{\ud}{\ud t}\big\|\partial^m_i(\vv_R\cdot\nabla Q^{\ve})\big\|^2_{L^2}
-\ve^{2m}J\frac{\ud}{\ud t}\Big\langle\partial^m_i(\vv_R\cdot\nabla Q^{\ve}),\partial^m_i\dot{Q}_R\Big\rangle\nonumber\\
&\quad
+C(1+\Ef+\ve^2\Ef^2+\ve^8\Ef^5)
+(\delta+C\ve^2\Ef)\Ff,\label{P3-FFR-dQR}
\end{align}
where $\partial^m_i$ represents the $m$-th order partial derivative operator with respect to the component $x_i$, and $Q^{\ve}=\widetilde{Q}+\ve^3Q_R$.
\end{lemma}

\begin{proof}
We only provide here the arguments of (\ref{P3-FFR-dQR}) for the case $m=0$. We relegate the proof of the cases $m=1,2$ in (\ref{P3-FFR-dQR}) to Appendix so as not to destroy the main body of this paper.

Firstly,  we control the term $\langle\widetilde{\FF}_1, \dot{Q}_R\rangle$.
Note that there holds $\langle\vv_R\cdot\nabla Q_0,\MH_{\nn}(Q_R)\rangle=0$ since $\vv_R\cdot\nabla Q_0\in {\rm Ker}\MH_{\nn}$ and $\MH_{\nn}(Q_R)\in ({\rm Ker}\MH_{\nn})^{\bot}$. Then we have
\begin{align}\label{vv-MHQR}
\Big\langle\vv_R\cdot\nabla\widetilde{Q},\frac{1}{\ve}\MH^{\ve}_{\nn}(Q_R)\Big\rangle
=&\Big\langle\vv_R\cdot\nabla\widehat{Q}^{\ve},\MH_{\nn}(Q_R)+\ve\ML(Q_R)\Big\rangle\nonumber\\
\leq&C\|\vv_R\|_{L^2}(\|Q_R\|_{L^2}+\ve\|Q_R\|_{H^2}) \leq C\Ef,
\end{align}
where $\widetilde{Q}=Q_0+\ve\widehat{Q}^{\ve}=Q_0+\ve(Q_1+\ve Q_2+\ve^2 Q_3)$.

 From the equation (\ref{remainder-Q-R}) and the bound (\ref{vv-MHQR}), utilizing integration by parts and Lemma \ref{lem:energy} yields
\begin{align}\label{FFdotQ-R}
\langle\widetilde{\FF}_1,\dot{Q}_R\rangle=&-J\frac{\ud}{\ud t}\langle\vv_R\cdot\nabla\widetilde{Q},\dot{Q}_R\rangle
+J\langle\vv_R\cdot\nabla\widetilde{Q},\ddot{Q}_R\rangle\nonumber\\
=&-J\frac{\ud}{\ud t}\langle\vv_R\cdot\nabla\widetilde{Q},\dot{Q}_R\rangle
-\mu_1\langle\vv_R\cdot\nabla\widetilde{Q},\dot{Q}_R\rangle
-\Big\langle\vv_R\cdot\nabla\widetilde{Q},\frac{1}{\ve}\MH^{\ve}_{\nn}(Q_R)\Big\rangle\nonumber\\
&+\Big\langle\vv_R\cdot\nabla\widetilde{Q},-\frac{\mu_2}{2}\DD_R+\mu_1[\BOm_R,Q_0]\Big\rangle
+\langle\vv_R\cdot\nabla\widetilde{Q},\FF_R+\widetilde{\FF}_R\rangle\nonumber\\
\leq&-J\frac{\ud}{\ud t}\langle\vv_R\cdot\nabla\widetilde{Q},\dot{Q}_R\rangle
+C(\Ef+\Ef^{\frac12}\Ff^{\frac12})\nonumber\\
&+C\Ef^{\frac12}\|\FF_R\|_{L^2}
+\langle\vv_R\cdot\nabla\widetilde{Q},\widetilde{\FF}_R\rangle.
\end{align}
It is easy to check that
\begin{align}\label{vvRQ-FF3}
\langle\vv_R\cdot\nabla\widetilde{Q},\widetilde{\FF}_1\rangle
=-\frac{J}{2}\frac{\ud}{\ud t}\big\|\vv_R\cdot\nabla\widetilde{Q}\big\|^2_{L^2}.
\end{align}
By using integration by parts, we deduce from Lemma \ref{lem:energy} that
\begin{align}\label{vvRQ-FF4}
\langle\vv_R\cdot\nabla\widetilde{Q},\widetilde{\FF}_2\rangle
=&-\ve^3J\frac{\ud}{\ud t}\big\langle\vv_R\cdot\nabla\widetilde{Q},\vv_R\cdot\nabla Q_R\big\rangle
+\ve^3J\big\langle\vv_R\cdot\nabla(\vv_R\cdot\nabla\widetilde{Q}), \dot{Q}_R\big\rangle\nonumber\\
&+\underbrace{\ve^3J\Big\langle(\partial_t+\tilde{\vv}\cdot\nabla)(\vv_R\cdot\nabla\widetilde{Q}),\vv_R\cdot\nabla Q_R\Big\rangle}_{\SSS_1}
\nonumber\\
&+\ve^6J\Big\langle\vv_R\cdot\nabla(\vv_R\cdot\nabla\widetilde{Q}),\vv_R\cdot\nabla Q_R\Big\rangle\nonumber\\
\leq&
-\ve^3J\frac{\ud}{\ud t}\big\langle\vv_R\cdot\nabla\widetilde{Q},\vv_R\cdot\nabla Q_R\big\rangle
+\SSS_1\nonumber\\
&+C\ve^3\|\vv_R\|_{H^2}\|\vv_R\|_{H^1}\|\dot{Q}_R\|_{L^2}
+C\ve^6\|\vv_R\|^2_{H^2}\|\vv_R\|_{H^1}\|\nabla Q_R\|_{L^2}\nonumber\\
\leq&
-J\ve^3\frac{\ud}{\ud t}\big\langle\vv_R\cdot\nabla\widetilde{Q},\vv_R\cdot\nabla Q_R\big\rangle+\SSS_1\nonumber\\
&+C(\ve\Ef^{\frac32}+\ve^2\Ef^2+\ve\Ef\Ff^{\frac12}+\ve^2\Ef^{\frac32}\Ff^{\frac12}).
\end{align}
Then from (\ref{FFdotQ-R})-(\ref{vvRQ-FF4}) and Lemma \ref{FFR-goodterm} we obtain
\begin{align}\label{vv-tildeQR-FFR}
\langle\widetilde{\FF}_1,\dot{Q}_R\rangle
\leq&
-J\frac{\ud}{\ud t}\langle\vv_R\cdot\nabla\widetilde{Q},\dot{Q}_R\rangle
-\frac{J}{2}\frac{\ud}{\ud t}\big\|\vv_R\cdot\nabla\widetilde{Q}\big\|^2_{L^2}-\ve^3J\frac{\ud}{\ud t}\big\langle\vv_R\cdot\nabla\widetilde{Q},\vv_R\cdot\nabla Q_R\big\rangle\nonumber\\
&
+\SSS_1+C(1+\Ef+\ve^2\Ef^2+\Ef^{\frac12}\Ff^{\frac12})+C\ve^2\Ef\Ff.
\end{align}

Now we derive the estimate of $\langle\widetilde{\FF}_2,\dot{Q}_R\rangle$. Similarly we obtain from  (\ref{remainder-Q-R}) that
\begin{align}
\langle\widetilde{\FF}_2,\dot{Q}_R\rangle=&-\ve^3J\frac{\ud}{\ud t}\langle\vv_R\cdot\nabla Q_R,\dot{Q}_R\rangle
+\ve^3J\langle\vv_R\cdot\nabla Q_R,\ddot{Q}_R\rangle\nonumber\\
&-\ve^3J\langle\vv_R\cdot\nabla\dot{Q}_R,\dot{Q}_R\rangle
+\ve^6J\big\langle\vv_R\cdot\nabla Q_R,\vv_R\cdot\nabla\dot{Q}_R\big\rangle\nonumber\\
=&
-\ve^3J\frac{\ud}{\ud t}\langle\vv_R\cdot\nabla Q_R,\dot{Q}_R\rangle
\underbrace{-\ve^3\mu_1\langle\vv_R\cdot\nabla Q_R,\dot{Q}_R\rangle}_{\mathcal{A}_1}\nonumber\\
&\underbrace{-\ve^3\Big\langle\vv_R\cdot\nabla Q_R,\frac{1}{\ve}\MH^{\ve}_{\nn}(Q_R)\Big\rangle}_{\mathcal{A}_2}
+\underbrace{\ve^3\Big\langle\vv_R\cdot\nabla Q_R,
 -\frac{\mu_2}{2}\DD_R+\mu_1[\BOm_R,Q_0]\Big\rangle}_{\mathcal{A}_3}\nonumber\\
&+\ve^3\langle\vv_R\cdot\nabla Q_R,\FF_R+\widetilde{\FF}_R\rangle
+\underbrace{\ve^6J\big\langle\vv_R\cdot\nabla Q_R,\vv_R\cdot\nabla\dot{Q}_R\big\rangle}_{\WW_1}.\nonumber
\end{align}
By Lemma \ref{lem:energy}, we have
\begin{align*}
\mathcal{A}_1
\leq& C\ve^3\|\vv_R\|_{H^2}\|\nabla Q_R\|_{L^2}\|\dot{Q}_R\|_{L^2}
\leq C\ve\Ef^{\frac32},\\
\mathcal{A}_2
=&-\ve^2\Big\langle\vv_R\cdot\nabla Q_R,\MH_{\nn}(Q_R)+\ve\ML(Q_R)\Big\rangle\\
\leq&C\ve^2\|\vv_R\|_{H^2}\|\nabla Q_R\|_{L^2}(\|Q_R\|_{L^2}+\|\ve Q_R\|_{H^2})
\leq C(\ve^2\Ef^{\frac32}+\ve\Ef\Ff^{\frac12}),\\
\mathcal{A}_3
\leq&C\ve^3\|\vv_R\|_{H^2}\|\nabla Q_R\|_{L^2}\|\nabla\vv_R\|_{L^2}
\leq C\ve\Ef\Ff^{\frac12}.
\end{align*}
By Lemma \ref{FFR-goodterm}, we get
\begin{align*}
\ve^3\langle\vv_R\cdot\nabla Q_R,\FF_R\rangle
\leq &\ve^3\|\vv_R\|_{H^2}\|\nabla Q_R\|_{L^2}\|\FF_R\|_{L^2}
\leq \ve\Ef\|\FF_R\|_{L^2}\\
\leq &C(1+\Ef+\ve^2\Ef^2+\ve^8\Ef^5)+C\ve^2\Ef\Ff.
\end{align*}
Using integration by parts, it follows immediately by Lemma \ref{lem:energy} that
\begin{align*}
&\ve^3\langle\vv_R\cdot\nabla Q_R,\widetilde{\FF}_R\rangle\\
&=-\ve^6\frac{J}{2}\frac{\ud}{\ud t}\|\vv_R\cdot\nabla Q_R\|^2_{L^2}
-\SSS_1-\WW_1
-\ve^9J\Big\langle\vv_R\cdot\nabla Q_R,\vv_R\cdot\nabla(\vv_R\cdot\nabla Q_R)\Big\rangle\\
&=-\ve^6\frac{J}{2}\frac{\ud}{\ud t}\|\vv_R\cdot\nabla Q_R\|^2_{L^2}
-\SSS_1-\WW_1.
\end{align*}
Thus we find
\begin{align}\label{ve3-vQR-FFR}
\langle\widetilde{\FF}_2,\dot{Q}_R\rangle
\leq&-\ve^3J\frac{\ud}{\ud t}\langle\vv_R\cdot\nabla Q_R,\dot{Q}_R\rangle
-\ve^6\frac{J}{2}\frac{\ud}{\ud t}\|\vv_R\cdot\nabla Q_R\|^2_{L^2}\nonumber\\
&-\SSS_1
+C(1+\Ef+\ve^2\Ef^2+\ve^8\Ef^5)+C\ve^2\Ef\Ff.
\end{align}
Recalling $Q^{\ve}=\widetilde{Q}+\ve^3Q_R$, we have
\begin{align}\label{energ:identity}
&\frac{\ud}{\ud t}\|\vv_R\cdot\nabla\widetilde{Q}\big\|^2_{L^2}
+2\ve^3\frac{\ud}{\ud t}\big\langle\vv_R\cdot\nabla\widetilde{Q},\vv_R\cdot\nabla Q_R\big\rangle
+\ve^6\frac{\ud}{\ud t}\|\vv_R\cdot\nabla Q_R\|^2_{L^2}\nonumber\\
&=\frac{\ud}{\ud t}\|\vv_R\cdot\nabla(\widetilde{Q}+\ve^3Q_R)\big\|^2_{L^2}
=\frac{\ud}{\ud t}\|\vv_R\cdot\nabla Q^{\ve}\big\|^2_{L^2}.
\end{align}

Therefore, summarizing (\ref{vv-tildeQR-FFR}) and (\ref{ve3-vQR-FFR}), and using (\ref{energ:identity}), we obtain
\begin{align*}
\langle\widetilde{\FF}_R,\dot{Q}_R\rangle
\leq&
-\frac{J}{2}\frac{\ud}{\ud t}\big\|\vv_R\cdot\nabla Q^{\ve}\big\|^2_{L^2}
-J\frac{\ud}{\ud t}\langle\vv_R\cdot\nabla Q^{\ve},\dot{Q}_R\rangle\\
&
+C(1+\Ef+\ve^2\Ef^2+\ve^8\Ef^5)+(\delta+C\ve^2\Ef)\Ff.
\end{align*}
\end{proof}

\subsection{The uniform energy estimate}
In this subsection, we derive the uniform energy estimate for the remainder system. 
\begin{proposition}\label{prop:energy}
Let $(\vv_R,Q_R)$ be a smooth solution of the remainder
system (\ref{remainder-Q-R})-(\ref{remainder-div}) on $[0,T]$, then for any $t\in [0,T]$,  it holds that
\begin{align}
\frac{d }{d t}\widetilde{\Ef}(t)+\Ff(t)\le C(1+\Ef+\ve^2\Ef^2+\ve^8\Ef^5)+C(\ve+\ve^2\Ef^{\frac12}+\ve^2\Ef)\Ff,\non
\end{align}
where the energy functional $\widetilde{\Ef}(t)$ is defined by
\begin{align}\label{compli-energ-func}
\widetilde{\Ef}(t)\eqdefa\widetilde{\Ef}_0(t)+\widetilde{\Ef}_1(t)+\widetilde{\Ef}_2(t),
\end{align}
and $\widetilde{\Ef}_i(t)(i=0,1,2)$ are given as follows:
\begin{align*}
\left\{
\begin{aligned}
\widetilde{\Ef}_0(t)=&\frac12\int_{\BR}\Big(|\vv_R|^2+J|\dot{Q}_R+\vv_R\cdot\nabla Q^{\ve}+Q_R|^2+(\mu_1-J)|Q_R|^2\\
&\qquad+\frac{1}{\ve}\MH^{\ve}_{\nn}(Q_R):Q_R
\Big)\ud\xx
+J\Big\langle\MH^{-1}_{\nn}\CG(Q_R^{\top}),\dot{Q}_R+\vv_R\cdot\nabla{Q}^\ve\Big\rangle,\\
\widetilde{\Ef}_1(t)=&\frac{\ve^2}{2}\int_{\BR}\Big(|\partial_i\vv_R|^2+J|\partial_i\dot{Q}_R+\partial_i(\vv_R\cdot\nabla Q^{\ve})|^2+\frac{1}{\ve}\MH^{\ve}_{\nn}(\partial_iQ_R):\partial_iQ_R\Big)\ud\xx,\\
\widetilde{\Ef}_2(t)=&\frac{\ve^4}{2}\int_{\BR}\Big(|\Delta\vv_R|^2+J|\Delta\dot{Q}_R+\Delta(\vv_R\cdot\nabla Q^{\ve})|^2
+\frac{1}{\ve}\MH^{\ve}_{\nn}(\Delta Q_R):\Delta Q_R\Big)\ud\xx.
\end{aligned}
\right.
\end{align*}
Here $\CG(Q)\eqdefa2bs\dot{\overline{\nn\nn}}\cdot Q-4cs^2Q:\dot{\overline{\nn\nn}}(\nn\nn-\frac13\II)$ and $Q^{\ve}=\widetilde{Q}+\ve^3Q_R$.
\end{proposition}

\begin{proof}
{\it Step 1. $L^2$-estimate}. On the one hand, multiplying the equation (\ref{remainder-Q-R}) by $Q_R$, taking the trace and integrating over the space $\BR$ and using the fact that $\MH^{\ve}_{\nn}(Q_R):Q_R\geq 0$ yields
\begin{align}\label{QR-LL2}
&J\langle\ddot{Q}_R,Q_R\rangle+\mu_1\langle\dot{Q}_R,Q_R\rangle\nonumber\\
&=-\Big\langle\frac{1}{\ve}\MH^{\ve}_{\nn}(Q_R),Q_R\Big\rangle
-\frac{\mu_2}{2}\langle\DD_R,Q_R\rangle+\mu_1\langle[\BOm_R,Q_0],Q_R\rangle
+\langle\FF_R+\widetilde{\FF}_R,Q_R\rangle\nonumber\\
&\leq C\|\nabla\vv_R\|_{L^2}\|Q_R\|_{L^2}+\Ef^{\frac12}\|\FF_R\|_{L^2}
+\langle\widetilde{\FF}_R,Q_R\rangle.
\end{align}
Considering the previous equality (\ref{ddQ-Q}),
then (\ref{QR-LL2}) can be reduced to
 \begin{align}\label{QR-L2}
&\frac12\frac{\ud}{\ud t}\int_{\BR}\Big(2J\dot{Q}_R:Q_R+\mu_1|Q_R|^2\Big)\ud\xx\nonumber\\
&\leq C(\Ef+\Ef^{\frac12}\Ff^{\frac12})+\Ef^{\frac12}\|\FF_R\|_{L^2}+\langle\widetilde{\FF}_R,Q_R\rangle.
\end{align}

On the other hand, multiplying the equation (\ref{remainder-Q-R}) by $\dot{Q}_R$ and the equation (\ref{remainder-v-R}) by $\vv_R$, integrating by parts over the space $\BR$, we hence obtain
\begin{align}\label{ddotQR-L2}
&\langle\dot{\vv}_R,\vv_R\rangle
+J\langle\ddot{Q}_R,\dot{Q}_R\rangle\nonumber\\
&=
\underbrace{-\Big\langle\beta_1Q_0(Q_0:\DD_R)+\beta_4\DD_R+\beta_5\DD_R\cdot Q_0+\beta_6Q_0\cdot\DD_R,\nabla\vv_R\Big\rangle}_{\mathcal{I}_1}\nonumber\\
&\quad
\underbrace{-\Big\langle\beta_7(\DD_R\cdot Q^2_0+Q^2_0\cdot\DD_R),\nabla\vv_R\Big\rangle}_{\mathcal{I}_2}
\underbrace{-\frac{\mu_2}{2}\big\langle\dot{Q}_R-[\BOm_R,Q_0],\nabla\vv_R\big\rangle}_{\mathcal{I}_3}\nonumber\\
&\quad
\underbrace{-\mu_1\Big\langle\big[Q_0,(\dot{Q}_R-[\BOm_R,Q_0])\big],\nabla\vv_R\Big\rangle}_{\mathcal{I}_4}
\underbrace{-\frac{\mu_2}{2}\langle\DD_R,\dot{Q}_R\rangle}_{\mathcal{I}_5}\nonumber\\
&\quad
\underbrace{-\mu_1\big\langle\dot{Q}_R-[\BOm_R,Q_0],\dot{Q}_R\big\rangle}_{\mathcal{I}_6}
\underbrace{-\Big\langle\frac{1}{\ve}\MH^{\ve}_{\nn}(Q_R),\dot{Q}_R\Big\rangle}_{\mathcal{I}_7}\nonumber\\
&\quad
+\langle\nabla\cdot\GG_R+\GG'_R,\vv_R\rangle+\langle\FF_R+\widetilde{\FF}_R,\dot{Q}_R\rangle.
\end{align}
Now we estimate (\ref{ddotQR-L2}) term by term as follows. We will use frequently a simple fact that $\langle A, B\rangle=0$ if the tensor $A$ is symmetric but $B$ skew symmetric.
Remembering the relation $\beta_6-\beta_5=\mu_2$, and noting $Q_0=s(\nn\nn-\frac13\II)$, it follows that
\begin{align}\label{I1I2}
\mathcal{I}_1+\mathcal{I}_2
=&-\Big\langle\beta_1Q_0(Q_0:\DD_R)+\beta_4\DD_R+\frac{\beta_5+\beta_6}{2}(Q_0\cdot\DD_R+\DD_R\cdot Q_0),\DD_R\Big\rangle\nonumber\\
&-\Big\langle\beta_7(\DD_R\cdot Q^2_0+Q^2_0\cdot\DD_R),\DD_R\Big\rangle\nonumber\\
&+\Big\langle\Big(\frac{\beta_5+\beta_6}{2}-\beta_5\Big)\DD_R\cdot Q_0
+\Big(\frac{\beta_5+\beta_6}{2}-\beta_6\Big)Q_0\cdot\DD_R, \DD_R+\BOm_R\Big\rangle\nonumber\\
=&-\beta_1s^2\|\nn\nn:\DD_R\|^2_{L^2}-\Big(\beta_4-\frac{s(\beta_5+\beta_6)}{3}
+\frac29\beta_7s^2\Big)\|\DD_R\|^2_{L^2}\nonumber\\
&-\Big(s(\beta_5+\beta_6)+\frac23\beta_7s^2\Big)\|\nn\cdot\DD_R\|^2_{L^2}
+\underbrace{\frac{\mu_2}{2}\Big\langle[\DD_R,Q_0],\BOm_R\Big\rangle}_{\mathcal{I}'_1}.
\end{align}
Due to the symmetry of the commutator $[\BOm_R,Q_0]$, it follows that
\begin{align*}
\mathcal{I}'_{1}+\mathcal{I}_{3}+\mathcal{I}_{5}
&=\frac{\mu_2}{2}\langle[\DD_R,Q_0],\BOm_R\rangle
-\frac{\mu_2}{2}\big\langle\dot{Q}_R-[\BOm_R,Q_0],\DD_R\big\rangle
-\frac{\mu_2}{2}\langle\DD_R,\dot{Q}_R\rangle\\
&=-\mu_2\big\langle\dot{Q}_R-[\BOm_R,Q_0],\DD_R\big\rangle.
\end{align*}
Simultaneously, we have
\begin{align*}
\mathcal{I}_4+\mathcal{I}_6
=&-\mu_1\Big\langle\big[Q_0,(\dot{Q}_R-[\BOm_R,Q_0])\big],\BOm_R\Big\rangle
-\mu_1\big\langle\dot{Q}_R-[\BOm_R,Q_0],\dot{Q}_R\big\rangle\\
=&-\mu_1\Big\langle[Q_0,\BOm_R],\dot{Q}_R-[\BOm_R,Q_0]\Big\rangle
-\mu_1\big\langle\dot{Q}_R-[\BOm_R,Q_0],\dot{Q}_R\big\rangle\\
=&-\mu_1\|\dot{Q}_R-[\BOm_R,Q_0]\|^2_{L^2}.
\end{align*}
It may be observed that
\begin{align*}
\mathcal{I}'_{1}+\mathcal{I}_{3}+\mathcal{I}_{4}+\mathcal{I}_5+\mathcal{I}_6
&=-\mu_1\Big\|\dot{Q}_R-[\BOm_R,Q_0]+\frac{\mu_2}{2\mu_1}\DD_R\Big\|^2_{L^2}
+\frac{\mu^2_2}{4\mu_1}\|\DD_R\|^2_{L^2},
\end{align*}
which combines with (\ref{I1I2}) and the dissipation relation (\ref{diss:ineq}) yields
\begin{align}\label{I1-I6}
&\mathcal{I}_1+\mathcal{I}_2+\mathcal{I}_3+\mathcal{I}_4+\mathcal{I}_5+\mathcal{I}_6\nonumber\\
&=-\tilde{\beta}_1\|\nn\nn:\DD_R\|^2_{L^2}
-\tilde{\beta}_2\|\DD_R\|^2_{L^2}-\tilde{\beta}_3\|\nn\cdot\DD_R\|^2_{L^2}-4\delta\|\DD_R\|^2_{L^2}
\nonumber\\
&\leq-4\delta\|\DD_R\|^2_{L^2},
\end{align}
where $\delta>0$ is small enough, such that $\tilde{\beta}_i(i=1,2,3)$ given by (\ref{tilde-beta}) satisfy (\ref{diss:coeff}).

For the term $\mathcal{I}_7$, using $Q_R:\II=\text{Tr} Q_R=0$,  we can write
\begin{align*}
\frac{\ud}{\ud t}\langle\frac{1}{\ve}\MH^{\ve}_{\nn}(Q_R),Q_R\rangle
=&\frac{2}{\ve}\langle\MH^{\ve}_{\nn}(Q_R),\dot{Q}_R\rangle
+\frac{1}{\ve}\Big\langle bs\big(\dot{\overline{\nn\nn}}\cdot Q_R+Q_R\cdot\dot{\overline{\nn\nn}}\big)\\
&-2cs^2\big(Q_R:\dot{\overline{\nn\nn}}(\nn\nn)+(Q_R:\nn\nn)\dot{\overline{\nn\nn}}\big),Q_R\Big\rangle\\
=&-2\mathcal{I}_7
+\frac{2}{\ve}\Big\langle bs\dot{\overline{\nn\nn}}\cdot Q_R-2cs^2Q_R:\dot{\overline{\nn\nn}}(\nn\nn),Q_R\Big\rangle,
\end{align*}
which implies from Lemma \ref{lem:singular} that
\begin{align}\label{key-estimatehqr}
\mathcal{I}_7
&\leq-\frac{1}{2}\frac{\ud}{\ud t}\langle\frac{1}{\ve}\MH^{\ve}_{\nn}(Q_R),Q_R\rangle
-J\frac{\ud}{\ud t}\Big\langle\MH^{-1}_{\nn}\CG(Q_R^{\top}),\dot{Q}_R+\vv_R\cdot\nabla Q^{\ve}\Big\rangle\nonumber\\
&\quad+C(1+\Ef+\ve^2\Ef^2)+(\delta+C\ve^2\Ef)\Ff.
\end{align}

Hence, summarizing (\ref{ddotQR-L2}) and the estimates (\ref{I1-I6})-(\ref{key-estimatehqr}), we get
\begin{align}\label{dQR-vR-L2}
&\frac12\frac{\ud}{\ud t}\int_{\BR}\Big(|\vv_R|^2+J|\dot{Q}_R|^2
+\frac{1}{\ve}\MH^{\ve}_{\nn}(Q_R):Q_R
\Big)\ud\xx\nonumber\\
&\quad+J\frac{\ud}{\ud t}\Big\langle\MH^{-1}_{\nn}\CG(Q_R^{\top}),\dot{Q}_R+\vv_R\cdot\nabla Q^{\ve}\Big\rangle
+4\delta\|\nabla\vv_R\|^2_{L^2}\nonumber\\
&\leq
C\Big(\|\GG'_R\|_{L^2}\Ef^{\frac12}+\|\FF_R\|_{L^2}\Ef^{\frac12}+\|\GG_R\|_{L^2}\Ff^{\frac12}\Big)
+\langle\widetilde{\FF}_R,\dot{Q}_R\rangle\nonumber\\
&\quad+C(1+\Ef+\ve^2\Ef^2)+(\delta+C\ve^2\Ef)\Ff.
\end{align}
Then, adding (\ref{QR-L2}) to (\ref{dQR-vR-L2}), and using Lemma \ref{FFR-goodterm}-\ref{GG-lem} and Lemma \ref{lem:FR-1}-\ref{lem:FR-2}, we obtain
\begin{align}\label{L2-estimate}
&\frac{\ud}{\ud t}\widetilde{\Ef}_0(t)+4\delta\|\nabla\vv_R\|^2_{L^2}\nonumber\\
&\leq
C(1+\Ef+\ve^2\Ef^2+\ve^8\Ef^5)+\delta\Ff+C(\ve+\ve^2\Ef^{\frac12}+\ve^2\Ef)\Ff.
\end{align}

{\it Step 2. $H^1$-estimate}. We act the derivative $\partial_i$ on the equation (\ref{remainder-Q-R}) and take $L^2$-inner product with $\partial_i\dot{Q}_R$. Again by acting $\partial_i$ on the equation (\ref{remainder-v-R}) and taking $L^2$-inner product with $\partial_i\vv_R$, we then have
\begin{align*}
&\ve^2\Big\langle\partial_t(\partial_i\vv_R),\partial_i\vv_R\Big\rangle
+\ve^2J\Big\langle\partial_t(\partial_i\dot{Q}_R),\partial_i\dot{Q}_R\Big\rangle\\
&=
\underbrace{-\ve^2\bigg\langle\partial_i\Big(\beta_1Q_0(Q_0:\DD_R)+\beta_4\DD_R
+\beta_5\DD_R\cdot Q_0
+\beta_6Q_0\cdot\DD_R\Big),\nabla\partial_i\vv_R\bigg\rangle}_{\mathcal{J}_{1}}\\
&\quad
\underbrace{-\ve^2\beta_7\Big\langle\partial_i(\DD_R\cdot Q^2_0+Q^2_0\cdot\DD_R),\nabla\partial_i\vv_R\Big\rangle}_{\mathcal{J}_{2}}\\
&\quad
\underbrace{-\ve^2\frac{\mu_2}{2}\Big\langle\partial_i(\dot{Q}_R-[\BOm_R,Q_0]),\nabla\partial_i\vv_R\Big\rangle}_{\mathcal{J}_3}
\underbrace{-\ve^2\mu_1\Big\langle\partial_i\big[Q_0,(\dot{Q}_R-[\BOm_R,Q_0])\big],\nabla\partial_i\vv_R\Big\rangle}_{\mathcal{J}_4}\\
&\quad
\underbrace{-\ve^2\langle\partial_i\tilde{\vv}\cdot\nabla\vv_R, \partial_i\vv_R\rangle
-\ve^2\langle\partial_i\GG_R,\nabla\partial_i\vv_R\rangle+\ve^2\langle\partial_i\GG'_R,\partial_i\vv_R\rangle}_{\mathcal{J}_5}\\
&\quad
\underbrace{-\ve^2\frac{\mu_2}{2}\langle\partial_i\DD_R,\partial_i\dot{Q}_R\rangle}_{\mathcal{J}_6}
\underbrace{-\ve^2\mu_1\langle\partial_i\dot{Q}_R-\partial_i[\BOm_R,Q_0],\partial_i\dot{Q}_R\rangle}_{\mathcal{J}_7}
\underbrace{-\ve^2\Big\langle\frac{1}{\ve}\partial_i\MH^{\ve}_{\nn}(Q_R),\partial_i\dot{Q}_R\Big\rangle}_{\mathcal{J}_8}\\
&\quad
\underbrace{-\ve^2\langle\partial_i\tilde{\vv}\cdot\nabla\dot{Q}_R,\partial_i\dot{Q}_R\rangle\rangle}_{\mathcal{J}_9}
+\ve^2\langle\partial_i\FF_R+\partial_i\widetilde{\FF}_R,\partial_i\dot{Q}_R\rangle.
\end{align*}
Via employing the analogous method in (\ref{I1I2}), we derive that
\begin{align*}
\mathcal{J}_{1}+\mathcal{J}_2
\leq&-\ve^2\Big\langle\beta_1Q_0(Q_0:\partial_i\DD_R)+\beta_4\partial_i\DD_R
+\beta_5\partial_i\DD_R\cdot Q_0+\beta_6Q_0\cdot\partial_i\DD_R,\nabla\partial_i\vv_R\Big\rangle\\
&
-\ve^2\Big\langle\beta_7(\partial_i\DD_R\cdot Q^2_0+Q^2_0\cdot\partial_i\DD_R),\nabla\partial_i\vv_R\Big\rangle
+C\|\ve\nabla\vv_R\|_{L^2}\|\ve\nabla\partial_i\vv_R\|_{L^2}\\
\leq&-\ve^2\Big\langle\beta_1Q_0(Q_0:\partial_i\DD_R)+\beta_4\partial_i\DD_R
+\frac{\beta_5+\beta_6}{2}(Q_0\cdot\partial_i\DD_R+\partial_i\DD_R\cdot Q_0),\partial_i\DD_R\Big\rangle\\
&
-\ve^2\Big\langle\beta_7(\partial_i\DD_R\cdot Q^2_0+Q^2_0\cdot\partial_i\DD_R),\partial_i\DD_R\Big\rangle
+\underbrace{\ve^2\frac{\mu_2}{2}\langle[\partial_i\DD_R,Q_0],\nabla\partial_i\vv_R\rangle}_{\mathcal{J}'_{1}}
+C\Ef^{\frac12}\Ff^{\frac12}.
\end{align*}
Direct calculation enables us to get
\begin{align*}
\mathcal{J}'_{1}+\mathcal{J}_3+\mathcal{J}_6
&\leq
\ve^2\frac{\mu_2}{2}\langle[\partial_i\DD_R,Q_0],\partial_i\BOm_R\rangle
-\ve^2\mu_2\langle\partial_i\DD_R,\partial_i\dot{Q}_R\rangle\\
&\quad+\ve^2\frac{\mu_2}{2}\langle[\partial_i\BOm_R,Q_0],\partial_i\DD_R\rangle
+C\|\ve\nabla\vv_R\|_{L^2}\|\ve\nabla\partial_i\vv_R\|_{L^2}\\
&\leq -\ve^2\mu_2\langle\partial_i\dot{Q}_R-[\partial_i\BOm_R,Q_0],\partial_i\DD_R\rangle
+C\Ef^{\frac12}\Ff^{\frac12}.
\end{align*}
For the estimates of $\mathcal{J}_4$ and $\mathcal{J}_7$, it is easily to deduce that
\begin{align*}
\mathcal{J}_4+\mathcal{J}_7\leq&
-\ve^2\mu_1\Big\langle\big[Q_0,(\partial_i\dot{Q}_R-[\partial_i\BOm_R,Q_0])\big],\nabla\partial_i\vv_R\Big\rangle\\
&-\ve^2\mu_1\langle\partial_i\dot{Q}_R-[\partial_i\BOm_R,Q_0],\partial_i\dot{Q}_R\rangle\\
&+C\Big(\big\|\ve(\dot{Q}_R-[\BOm_R,Q_0])\big\|_{L^2}
+\big\|\ve(\partial_i\dot{Q}_R-[\BOm_R,\partial_iQ_0])\big\|_{L^2}\Big)\|\ve\nabla\partial_i\vv_R\|_{L^2}\\
&+C\|\ve\nabla\vv_R\|_{L^2}\|\ve\partial_i\dot{Q}_R\|_{L^2}\\
\leq&
-\ve^2\mu_1\big\|\partial_i\dot{Q}_R-[\partial_i\BOm_R,Q_0]\big\|^2_{L^2}
+C(\Ef^{\frac12}\Ff^{\frac12}+\Ef).
\end{align*}
Noticing the following equality
\begin{align*}
&-\ve^2\mu_1\big\|\partial_i\dot{Q}_R-[\partial_i\BOm_R,Q_0]\big\|^2_{L^2}
-\ve^2\mu_2\langle\partial_i\dot{Q}_R-[\partial_i\BOm_R,Q_0],\partial_i\DD_R\rangle\\
&\quad=-\ve^2\mu_1\big\|\partial_i\dot{Q}_R-[\partial_i\BOm_R,Q_0]+\frac{\mu_2}{2\mu_1}\partial_i\DD_R\big\|^2_{L^2}
+\frac{\mu^2_2}{4\mu_1}\|\partial_i\DD_R\|^2_{L^2},
\end{align*}
and taking advantage of the dissipation relation (\ref{diss:ineq}), then we can infer that
\begin{align*}
&\mathcal{J}_1+\mathcal{J}_2+\mathcal{J}_{3}+\mathcal{J}_4+\mathcal{J}_6+\mathcal{J}_7\\
&\leq
-\ve^2\beta_1s^2\|\nn\nn:\partial_i\DD_R\|^2_{L^2}
-\ve^2\Big(\beta_4-\frac{s(\beta_5+\beta_6)}{3}\Big)\|\partial_i\DD_R\|^2_{L^2}\\
&\quad
-\ve^2s(\beta_5+\beta_6)\|\nn\cdot\partial_i\DD_R\|^2_{L^2}
-\ve^2\mu_2\langle\partial_i\dot{Q}_R-[\partial_i\BOm_R,Q_0],\partial_i\DD_R\rangle\\
&\quad-\ve^2\mu_1\big\|\partial_i\dot{Q}_R-[\partial_i\BOm_R,Q_0]\big\|^2_{L^2}
+C(\Ef^{\frac12}\Ff^{\frac12}+\Ef)\\
&\leq
-\ve^2\tilde{\beta}_1\|\nn\nn:\partial_i\DD_R\|^2_{L^2}-\ve^2\tilde{\beta}_2\|\partial_i\DD_R\|^2_{L^2}
-\ve^2\tilde{\beta}_3\|\nn\cdot\partial_i\DD_R\|^2_{L^2}\\
&\quad
-4\ve^2\delta\|\partial_i\DD_R\|^2_{L^2}
+C(\Ef^{\frac12}\Ff^{\frac12}+\Ef)\\
&\leq
-4\ve^2\delta\|\partial_i\DD_R\|^2_{L^2}+C\Ef+\delta\Ff,
\end{align*}
where $\delta>0$ is small enough such that the coefficients $\tilde{\beta}_i(i=1,2,3)$ given by (\ref{tilde-beta}) satisfy the relation (\ref{diss:coeff}).
In addition, the terms $\mathcal{J}_{5}$ and $\mathcal{J}_{9}$ can be controlled as
\begin{align*}
\mathcal{J}_{5}+\mathcal{J}_{9}
\leq&C\Big(\|\ve\nabla\vv_R\|^2_{L^2}+\|\ve\partial_i\GG_R\|_{L^2}\|\ve\nabla\partial_i\vv_R\|_{L^2}\\
&+\|\ve\partial_i\GG'_R\|_{L^2}\|\ve\partial_i\vv_R\|_{L^2}
+\|\ve\nabla\dot{Q}_R\|_{L^2}\|\ve\partial_i\dot{Q}_R\|_{L^2}\Big)\\
\leq&C\Ef+C(\|\ve\partial_i\GG_R\|_{L^2}\Ff^{\frac12}+\|\ve\partial_i\GG'_R\|_{L^2}\Ef^{\frac12}).
\end{align*}
We next deal with the term $\mathcal{J}_{8}$. First, we can observe that
\begin{align*}
\mathcal{J}_{8}\leq&-\ve\Big\langle\MH^{\ve}_{\nn}(\partial_iQ_R),\partial_i\dot{Q}_R\Big\rangle
+\ve\|Q_R\|_{L^2}\|\partial_i\dot{Q}_R\|_{L^2}\\
\leq&-\ve\Big\langle\MH^{\ve}_{\nn}(\partial_iQ_R),\dot{\overline{\partial_iQ_R}}\Big\rangle
-\ve\Big\langle\MH^{\ve}_{\nn}(\partial_iQ_R),\partial_i\tilde{\vv}\cdot\nabla Q_R\Big\rangle
+C\Ef^{\frac12}\Ff^{\frac12}\\
\leq&\underbrace{-\ve\Big\langle\MH^{\ve}_{\nn}(\partial_iQ_R),\dot{\overline{\partial_iQ_R}}\Big\rangle}_{\mathcal{J}'_{8}}
+C(\Ef+\Ef^{\frac12}\Ff^{\frac12}).
\end{align*}
Using Lemma \ref{lem:singular}, we get
\begin{align*}
\ve\frac{\ud}{\ud t}\langle\MH^{\ve}_{\nn}(\partial_iQ_R),\partial_iQ_R\rangle
=&2\ve\langle\MH^{\ve}_{\nn}(\partial_iQ_R),\dot{\overline{\partial_iQ_R}}\rangle
+\frac{1}{\ve}\Big\langle bs\big(\dot{\overline{\nn\nn}}\cdot \partial_iQ_R+\partial_iQ_R\cdot\dot{\overline{\nn\nn}}\big)\\
&-2cs^2\big(\partial_iQ_R:\dot{\overline{\nn\nn}}(\nn\nn)+(\partial_iQ_R:\nn\nn)\dot{\overline{\nn\nn}}\big),\partial_iQ_R\Big\rangle\\
\leq&-2\mathcal{J}'_8+C\Ef,
\end{align*}
which implies
\begin{align}\label{mathcalJ-8}
\mathcal{J}_8\leq-\frac{\ve}{2}\frac{\ud}{\ud t}\langle\MH^{\ve}_{\nn}(\partial_iQ_R),\partial_iQ_R\rangle
+\delta\Ff+C\Ef.
\end{align}

Summarizing the above estimates, we get
\begin{align}
&\ve^2\Big\langle\partial_t(\partial_i\vv_R),\partial_i\vv_R\Big\rangle
+\ve^2J\Big\langle\partial_t(\partial_i\dot{Q}_R),\partial_i\dot{Q}_R\Big\rangle\nonumber\\
&\quad
+\frac{\ve}{2}\frac{\ud}{\ud t}\langle\MH^{\ve}_{\nn}(\partial_iQ_R),\partial_iQ_R\rangle
+4\ve^2\delta\|\partial_i\nabla\vv_R\|^2_{L^2}\nonumber\\
&\leq
C\Big(\|\ve\partial_i\GG_R\|_{L^2}\Ff^{\frac12}+\|\ve\partial_i\GG'_R\|_{L^2}\Ef^{\frac12}
+\|\ve\partial_i\FF_R\|_{L^2}\Ef^{\frac12}\Big)\nonumber\\
&\quad+\ve^2\langle\partial_i\widetilde{\FF}_R,\partial_i\dot{Q}_R\rangle
+C\Ef+\delta\Ff.\nonumber
\end{align}
Then using Lemma \ref{FFR-goodterm}-\ref{GG-lem} and Lemma \ref{lem:FR-2}, we obtain
\begin{align}\label{H1-estimate}
&\frac{\ud}{\ud t}\widetilde{\Ef}_1(t)+4\ve^2\delta\|\partial_i\nabla\vv_R\|^2_{L^2}\nonumber\\
&\leq
C(1+\Ef+\ve^2\Ef^2+\ve^8\Ef^5)+\delta\Ff+C(\ve+\ve^2\Ef^{\frac12}+\ve^2\Ef)\Ff.
\end{align}

{\it Step 3. $H^2$-estimate}. Similar to Step 2, one can deduce that
\begin{align}\label{H2-estimate-1}
&\ve^4\Big\langle\partial_t(\Delta\vv_R),\Delta\vv_R\Big\rangle
+\ve^4J\Big\langle\partial_t(\Delta\dot{Q}_R),\Delta\dot{Q}_R\Big\rangle\nonumber\\
&\quad+\frac{\ve^3}{2}\frac{\ud}{\ud t}\langle\MH^{\ve}_{\nn}(\Delta Q_R),\Delta Q_R\rangle
+4\ve^4\delta\|\nabla\Delta\vv_R\|^2_{L^2}\nonumber\\
&\leq
C\Big(\|\ve^2\Delta\GG_R\|_{L^2}\Ff^{\frac12}+\|\ve^2\Delta\GG'_R\|_{L^2}\Ef^{\frac12}
+\|\ve^2\Delta\FF_R\|_{L^2}\Ef^{\frac12}\Big)\nonumber\\
&\quad
+\ve^4\langle\Delta\widetilde{\FF}_R,\Delta\dot{Q}_R\rangle+C\Ef+\delta\Ff.
\end{align}
The proof of (\ref{H2-estimate-1})  is delegated in the Appendix. Likewise, using Lemma \ref{FFR-goodterm}-\ref{GG-lem} and Lemma \ref{lem:FR-2} yields
\begin{align}\label{H2-estimate}
&\frac{\ud}{\ud t}\widetilde{\Ef}_2(t)+4\ve^4\delta\|\nabla\Delta\vv_R\|^2_{L^2}\nonumber\\
&\leq
C(1+\Ef+\ve^2\Ef^2+\ve^8\Ef^5)+\delta\Ff+C(\ve+\ve^2\Ef^{\frac12}+\ve^2\Ef)\Ff.
\end{align}

Combining (\ref{L2-estimate}), (\ref{H1-estimate}) and (\ref{H2-estimate}), we finish the proof of Proposition \ref{prop:energy}.
\end{proof}

The following lemma shows that $\widetilde{\Ef}(t)$ defined by (\ref{compli-energ-func}) and $\Ef(t)$ defined by (\ref{energy-functional-Ef}) can be controlled by each other.
\begin{lemma}\label{equiv-energy-functionals}
If $\mu_1\gg J$, then there exist constants $c_0>0$ and $C_0>0$, such that
\begin{align}\label{equiv-Efs}
c_0(1-\ve \Ef(t))\Ef(t)\leq \widetilde{\Ef}(t)\leq C_0(1+\ve \Ef(t)) \Ef(t).
\end{align}
\end{lemma}
\begin{proof}
It suffices to prove the first inequality in (\ref{equiv-Efs}).
Let $S_R=\dot{Q}_R+\vv_R\cdot\nabla Q^{\ve}+Q_R,~ Q^{\ve}=\widetilde{Q}+\ve^3Q_R$ and $Q_R=Q_R^{\top}+Q_R^{\bot}$ with $Q_R^{\top}\in\text{Ker}\MH_{\nn}$ and $Q_R^{\bot}\in(\text{Ker}\MH_{\nn})^{\perp}$. Then using Proposition \ref{linearized-oper-prop} we have
\begin{align*}
\widetilde{\Ef}_0(t)=&\frac12\int_{\BR}\Big(|\vv_R|^2+J|S_R|^2+(\mu_1-J)|Q_R|^2+\frac{1}{\ve}\MH^{\ve}_{\nn}(Q_R):Q_R\Big)\ud\xx\\
&+J\big\langle\MH^{-1}_{\nn}\CG(Q_R^{\top}),S_R\big\rangle
-J\big\langle\MH^{-1}_{\nn}\CG(Q_R^{\top}),Q^{\bot}_R\big\rangle\\
\geq&C\int_{\BR}\Big(|\vv_R|^2+|S_R|^2+\frac{\mu_1}{2}|Q_R|^2+\frac{1}{\ve}\MH^{\ve}_{\nn}(Q_R):Q_R\Big)\ud\xx\\
\geq&C(1-\ve^3 \|\nabla Q_R\|_{L^{\infty}})
\int_{\BR}\Big(|\vv_R|^2+|\dot{Q}_R|^2+|Q_R|^2+\frac{1}{\ve}\MH^{\ve}_{\nn}(Q_R):Q_R\Big)\ud\xx.
\end{align*}
Note that for $m=1,2$, by using H$\ddot{o}$lder inequality, we estimate
\begin{align*}
&\ve^{2m}\big\|\partial^m_i(\vv_R\cdot\nabla Q^{\ve})\big\|_{L^2}^2\\
&\le C\ve^{2m}\| \vv_R\|_{H^m}^2+C\ve^{2m+3}(\|\vv_R\|_{H^m}^2\|\nabla Q_R\|_{H^2}^2+\|\vv_R\|_{H^2}^2\|\nabla Q_R\|_{H^m}^2\big)\\
&\leq C(1+\ve\mathfrak{E})\mathfrak{E},
\end{align*}
which yields that
\begin{align*}
\widetilde{\Ef}_1(t)+\widetilde{\Ef}_2(t)
\geq C&\int_{\BR}\Big(|\partial_i\vv_R|^2+|\partial_i\dot{Q}_R|^2
+\frac{1}{\ve}\MH^{\ve}_{\nn}(\partial_iQ_R):\partial_iQ_R\\
&\quad+|\Delta\vv_R|^2+|\Delta\dot{Q}_R|^2+\frac{1}{\ve}\MH^{\ve}_{\nn}(\Delta Q_R):\Delta Q_R\Big)\ud\xx- \delta_0 (1+\ve\mathfrak{E})\mathfrak{E}.
\end{align*}
Therefore, there exists a constant $c_0>0$ such that $c_0(1-\ve\mathfrak{E}(t))\Ef(t)\leq \widetilde{\Ef}(t)$.
\end{proof}

\subsection{The proof Theorem \ref{thm:main1}}
Given the initial data $(\vv_0^{\ve},\partial_tQ_0^{\ve},\nabla Q_0^{\ve})\in H^2$, it can be proved from the similar energy method in \cite{DZ}
that there exists a maximal time $T_\ve>0$ and a unique solution $(\vv^\ve,Q^\ve)$ of the system (\ref{eq:QS-general-ve1})-(\ref{eq:QS-general-ve3}) such that
\begin{align*}
(\partial_tQ^\ve,\nabla Q^{\ve})\in L^{\infty}([0,T_\ve);H^2)\cap L^2(0,T_\ve;H^2),~~\vv^\ve\in L^{\infty}([0,T_\ve);H^2)\cap L^2(0,T_\ve;H^3).
\end{align*}
From Proposition \ref{prop:energy} and Lemma \ref{equiv-energy-functionals} we have
\begin{align}
\frac{d }{d t}\widetilde{\Ef}(t)+\Ff(t)\le C(1+\widetilde{\Ef}+\ve^2\widetilde{\Ef}^2+\ve^8\widetilde{\Ef}^5)+C(\ve+\ve^2\widetilde{\Ef}^{\frac12}+\ve^2\widetilde{\Ef})\Ff,\non
\end{align}
for any $t\in [0,T_\ve]$. Under the assumptions of Theorem \ref{thm:main1}, it follows that
$$\widetilde{\Ef}(0)\le C_1\Big(\|\vv_{R,0}^\ve\|_{H^2}+\|Q_{R,0}^\ve\|_{H^3}+\|\partial_tQ_{R,0}^{\ve}\|_{H^2}+\ve^{-1}\|\MP^{{ \rm out}}(Q^\ve_{R,0})\|_{L^2}\Big)\le C_1 E_0.$$
Let $\widetilde{E}_1=(1+C_1E_0){e}^{2CT}>\widetilde{\Ef}(0)$, and
$$T_1=\sup\{t\in[0,T_\ve]: \widetilde{\Ef}(t)\le \widetilde{E}_1\}.$$
If we take $\ve_0$ small enough such that
\begin{align*}
4\ve_0\widetilde{E}_1<c_0,\quad\ve^2_0\widetilde{E}_1+\ve^8_0\widetilde{E}^4_1\leq 1,\quad C(\ve_0+\ve^2_0\widetilde{E}^{\frac12}_1+\ve^2_0\widetilde{E}_1)\leq1/2
\end{align*}
then for $t\le T_1$, there holds
\begin{align*}
\frac{d}{dt}\widetilde{\Ef}(t)\le 2C(1+\widetilde{\Ef}).
\end{align*}
Therefore, we can infer by means of a continuous argument that $T_1=T_\ve$,  $T\le T_\ve$ and $\widetilde{\Ef}(t)\le \widetilde{E}_1$ for $t\in[0,T]$.
Moreover, as $c_0(1-\ve\mathfrak{E}(t))\Ef(t)\leq \widetilde{\Ef}(t)\le \widetilde{E}_1< c_0/(4\ve_0)$ and $\mathfrak{E}(t)$ is continuous, we know that
$\Ef(t) $ can not attain $1/(2\ve)$. Otherwise $\widetilde{E}_1\ge c_0/(4\ve)$ which yields a contradiction. Therefore, we have
$\Ef(t) \le 2\widetilde{E}_1/c_0 \triangleq E_1$ for $t\in[0,T]$.
This completes the proof of Theorem \ref{thm:main1}.

\section{Appendix}

\subsection{The energy dissipation relation}
~

\begin{lemma}\label{dissip-rel-QS}
Assume that $\beta_1,\beta_4,\mu_1>0$, and $\beta_4-\frac{\mu_2^2}{4\mu_1}>0$.  Then for any smooth solution $(\vv, Q)$ of the inertial Qian-Sheng system (\ref{eq:Q-general-intro1})-(\ref{eq:Q-general-intro3}), it holds that
\begin{align}\label{QS-dissip}
&\frac{\ud}{\ud t}\Big(\int_{\BR}\frac12\big(|\vv|^2+J|\dot{Q}|^2\big)\ud\xx+\CF(Q,\nabla Q)\Big)\nonumber\\
&=-\beta_1\|Q:\DD\|^2_{L^2}-\Big(\beta_4-\frac{\mu_2^2}{4\mu_1}\Big)\|\DD\|^2_{L^2}
-(\beta_5+\beta_6)\langle\DD\cdot Q,\DD\rangle\nonumber\\
&\quad-2\beta_7\|\DD\cdot Q\|^2_{L^2}
-\mu_1\Big\|\dot{Q}-[\BOm,Q]+\frac{\mu_2}{2\mu_1}\DD\Big\|^2_{L^2}.
\end{align}
Moreover, if one of the following assumptions holds: (i) $\beta_5+\beta_6=0$ if $\beta_7=0$; (ii)
$(\beta_5+\beta_6)^2<8\beta_7\big(\beta_4-\frac{\mu_2^2}{4\mu_1}\big)$ if $\beta_7\neq0$, then the right hand side in (\ref{QS-dissip}) is non-positive.
\end{lemma}

\begin{proof}
Firstly, using any one of the assumptions, it is easy to obtain  that
$$\Big(\beta_4-\frac{\mu_2^2}{4\mu_1}\Big)|\DD|^2+(\beta_5+\beta_6)(\DD\cdot Q: \DD) +2\beta_7|\DD\cdot Q|^2_{L^2}>c_0 |\DD|^2$$
for some $c_0>0$. Now we prove (\ref{QS-dissip}). Taking $L^2$-inner product with $\dot{Q}$ in the equation (\ref{eq:Q-general-intro1}), and taking $L^2$-inner product with $\vv$ in the equation (\ref{eq:Q-general-intro2}), we get
\begin{align*}
&J\langle\ddot{Q},\dot{Q}\rangle+\langle\partial_t\vv,\vv\rangle\\
&=-\mu_1\big\langle\dot{Q}-[\BOm,Q],\dot{Q}\big\rangle+\langle\HH,\dot{Q}\rangle-\frac{\mu_2}{2}\langle\DD,\dot{Q}\rangle
+\langle\nabla\cdot\sigma^d,\vv\rangle\\
&\quad-\Big\langle\beta_1Q(Q:\DD)+\beta_4\DD+\beta_5\DD\cdot Q+\beta_6 Q\cdot\DD
+\beta_7(\DD\cdot Q^2+Q^2\cdot\DD),\nabla\vv\Big\rangle\\
&\quad-\frac{\mu_2}{2}\langle\dot{Q}-[\BOm,Q],\nabla\vv\rangle
-\mu_1\Big\langle[Q,(\dot{Q}-[\BOm,Q])],\nabla\vv\Big\rangle\\
&\eqdefa I+II+III+IV+V+VI+VII.
\end{align*}
For terms $I$ and $VII$, we have
\begin{align*}
I+VII=&-\mu_1\big\langle\dot{Q}-[\BOm,Q],\dot{Q}\big\rangle
-\mu_1\Big\langle[Q,(\dot{Q}-[\BOm,Q])],\DD+\BOm\Big\rangle\\
=&-\mu_1\big\langle\dot{Q}-[\BOm,Q],\dot{Q}\big\rangle
-\mu_1\Big\langle[Q,\BOm],(\dot{Q}-[\BOm,Q])\Big\rangle\\
=&-\mu_1\|\dot{Q}-[\BOm,Q]\|^2_{L^2}.
\end{align*}
Recalling the relation $\beta_6-\beta_5=\mu_2$, we can deduce that
\begin{align*}
III+V+VI=&-\Big\langle\beta_1Q(Q:\DD)+\beta_4\DD+\beta_5\DD\cdot Q+\beta_6 Q\cdot\DD,\DD+\BOm\Big\rangle\\
&
-\Big\langle\beta_7(\DD\cdot Q^2+Q^2\cdot\DD),\DD\Big\rangle
-\frac{\mu_2}{2}\big\langle2\dot{Q}-[\BOm,Q],\DD\big\rangle\\
=&-\Big\langle\beta_1Q(Q:\DD)+\beta_4\DD+\frac{\beta_5+\beta_6}{2}(\DD\cdot Q+Q\cdot\DD),\DD\Big\rangle\\
&+\Big\langle\Big(\frac{\beta_5+\beta_6}{2}-\beta_5\Big)\DD\cdot Q+\Big(\frac{\beta_5+\beta_6}{2}-\beta_6\Big)Q\cdot\DD, \DD+\BOm\Big\rangle\\
&
-\Big\langle\beta_7(\DD\cdot Q^2+Q^2\cdot\DD),\DD\Big\rangle
-\frac{\mu_2}{2}\big\langle2\dot{Q}-[\BOm,Q],\DD\big\rangle\\
=&-\beta_1\|Q:\DD\|^2_{L^2}-\beta_4\|\DD\|^2_{L^2}
-(\beta_5+\beta_6)\langle\DD\cdot Q,\DD\rangle\\
&-2\beta_7\|\DD\cdot Q\|^2_{L^2}-\mu_2\big\langle\dot{Q}-[\BOm,Q],\DD\big\rangle.
\end{align*}
Further, it follows that
\begin{align*}
&I+III+V+VI+VII\\
&=-\beta_1\|Q:\DD\|^2_{L^2}-\beta_4\|\DD\|^2_{L^2}
-(\beta_5+\beta_6)\langle\DD\cdot Q,\DD\rangle\\
&\quad
-2\beta_7\|\DD\cdot Q\|^2_{L^2}
-\mu_2\big\langle\dot{Q}-[\BOm,Q],\DD\big\rangle
-\mu_1\|\dot{Q}-[\BOm,Q]\|^2_{L^2}\\
&=-\beta_1\|Q:\DD\|^2_{L^2}-\Big(\beta_4-\frac{\mu_2^2}{4\mu_1}\Big)\|\DD\|^2_{L^2}
-(\beta_5+\beta_6)\langle\DD\cdot Q,\DD\rangle\\
&\quad
-2\beta_7\|\DD\cdot Q\|^2_{L^2}
-\mu_1\Big\|\dot{Q}-[\BOm,Q]+\frac{\mu_2}{2\mu_1}\DD\Big\|^2_{L^2}.
\end{align*}
For the second term $II$, note that $\HH(Q)=-\frac{\delta\CF}{\delta Q}$ and $\nabla\cdot\vv=0$, we have
\begin{align*}
II=&-\Big\langle\frac{\delta\CF}{\delta Q},\partial_tQ\Big\rangle+\langle\HH(Q),\vv\cdot\nabla Q\rangle\\
=&-\frac{\ud}{\ud t}\CF(Q,\nabla Q)+\langle\HH(Q),\vv\cdot\nabla Q\rangle.
\end{align*}
Using the definition of the distortion stress $\sigma^d$, we can infer that
\begin{align*}
IV=&-\int_{\BR}\partial_j\Big(\frac{\partial\CF}{\partial Q_{kl,j}}Q_{kl,i}\Big)v_i\ud\xx\\
=&-\int_{\BR}\Big(\partial_j\Big(\frac{\partial\CF}{\partial Q_{kl,j}}\Big)Q_{kl,i}
+\frac{\partial\CF}{\partial Q_{kl,j}}Q_{kl,ij}\Big)v_i\ud\xx\\
=&-\int_{\BR}\Big(H_{kl}(Q)Q_{kl,i}+\frac{\partial\CF}{\partial Q_{kl}}Q_{kl,i}
+\frac{\partial\CF}{\partial Q_{kl,j}}Q_{kl,ij}\Big)v_i\ud\xx\\
=&-\int_{\BR}\Big(H_{kl}(Q)Q_{kl,i}+\partial_i\CF(Q,\nabla Q)\Big)v_i\ud\xx\\
=&-\langle\HH(Q),\vv\cdot\nabla Q\rangle.
\end{align*}
In conclusion, under the assumptions of Lemma \ref{dissip-rel-QS}, we obtain (\ref{QS-dissip}).
\end{proof}

\subsection{The estimate of $\ve^2\langle\partial_i\widetilde{\FF}_R,\partial_i\dot{Q}_R\rangle$}
Similar arguments for Lemma \ref{lem:FR-2} will be applied to the estimate of higher order derivative terms. First of all, note that
$\langle\partial_i\vv_R\cdot\nabla Q_0,\MH_{\nn}(\partial_iQ_R)\rangle=0$, then we have
\begin{align}\label{vvR-MHQR}
&-\ve^2\Big\langle\partial_i(\vv_R\cdot\nabla\widetilde{Q}),\frac{1}{\ve}\partial_i\MH^{\ve}_{\nn}(Q_R)\Big\rangle\nonumber\\
&\leq-\ve^2\Big\langle\partial_i\vv_R\cdot\nabla(Q_0+\ve\widehat{Q}^{\ve}),\frac{1}{\ve}\partial_i\MH^{\ve}_{\nn}(Q_R)\Big\rangle\nonumber\\
&\quad+C\ve\|\vv_R\|_{L^2}\big(\|\partial_i\MH_{\nn}(Q_R)\|_{L^2}+\ve\|\partial_i\ML(Q_R)\|_{L^2}\big)\nonumber\\
&\leq-\ve^3\Big\langle\partial_i\vv_R\cdot\nabla \widehat{Q}^{\ve},\frac{1}{\ve}\partial_i\MH^{\ve}_{\nn}(Q_R)\Big\rangle+C\Ef\nonumber\\
&\leq C\ve^2\|\partial_i\vv_R\|_{L^2}\big(\|\partial_i\MH_{\nn}(Q_R)\|_{L^2}+\ve\|\partial_i\ML(Q_R)\|_{L^2}\big)
+C\Ef\nonumber\\
&\leq C\Ef.
\end{align}
Recalling the equation (\ref{remainder-Q-R}), we derive from the integration by parts over $\xx\in\BR$ that
\begin{align}\label{part1-FFdotQ-R}
\ve^2\langle\partial_i\widetilde{\FF}_1,\partial_i\dot{Q}_R\rangle
&=
-\ve^2J\frac{\ud}{\ud t}\Big\langle\partial_i(\vv_R\cdot\nabla\widetilde{Q}),\partial_i\dot{Q}_R\Big\rangle
-\ve^2J\langle\partial_i\tilde{\vv}\cdot\nabla(\vv_R\cdot\nabla\widetilde{Q}),\partial_i\dot{Q}_R\rangle\nonumber\\
&\quad+\ve^2J\langle\partial_i(\vv_R\cdot\nabla\widetilde{Q}),\partial_i\ddot{Q}_R\rangle
-\ve^2J\langle\partial_i(\vv_R\cdot\nabla\widetilde{Q}),\partial_i\tilde{\vv}\cdot\nabla\dot{Q}_R\rangle\nonumber\\
&\leq
-\ve^2J\frac{\ud}{\ud t}\Big\langle\partial_i(\vv_R\cdot\nabla\widetilde{Q}),\partial_i\dot{Q}_R\Big\rangle
+\ve^2\langle\partial_i(\vv_R\cdot\nabla\widetilde{Q}),\partial_i\widetilde{\FF}_R\rangle\nonumber\\
&\quad+C\Ef^{\frac12}\|\ve\partial_i\FF_R\|_{L^2}
+C(\Ef+\Ef^{\frac12}\Ff^{\frac12}),
\end{align}
where we have applied Lemma \ref{lem:energy} and (\ref{vvR-MHQR}), and the following estimates
\begin{align*}
-\ve^2J\langle\partial_i\tilde{\vv}\cdot\nabla(\vv_R\cdot\nabla\widetilde{Q}),\partial_i\dot{Q}_R\rangle
\leq&C\ve^2\|\vv_R\|_{H^1}\|\partial_i\dot{Q}_R\|_{L^2}
\leq C\Ef,\\
-\ve^2J\langle\partial_i(\vv_R\cdot\nabla\widetilde{Q}),\partial_i\tilde{\vv}\cdot\nabla\dot{Q}_R\rangle
\leq&C\ve^2\|\vv_R\|_{H^1}\|\nabla\dot{Q}_R\|_{L^2}
\leq C\Ef,
\end{align*}
and
\begin{align}\label{ve2-vR-ddQR-d2}
&\ve^2J\langle\partial_i(\vv_R\cdot\nabla\widetilde{Q}),\partial_i\ddot{Q}_R\rangle\nonumber\\
&=-\mu_1\ve^2\langle\partial_i(\vv_R\cdot\nabla\widetilde{Q}),\partial_i\dot{Q}_R\rangle
-\ve^2\Big\langle\partial_i(\vv_R\cdot\nabla\widetilde{Q}),\frac{1}{\ve}\partial_i\MH^{\ve}_{\nn}(Q_R)\Big\rangle\nonumber\\
&\quad+\ve^2\Big\langle\partial_i(\vv_R\cdot\nabla\widetilde{Q}),
-\frac{\mu_2}{2}\partial_i\DD_R+\mu_1\partial_i[\BOm_R,Q_0]\Big\rangle\nonumber\\
&\quad
+\ve^2\langle\partial_i(\vv_R\cdot\nabla\widetilde{Q}),\partial_i\FF_R+\partial_i\widetilde{\FF}_R\rangle\nonumber\\
&\leq C\ve^2\|\vv_R\|_{H^1}(\|\partial_i\dot{Q}_R\|_{L^2}+\|\partial_i\nabla\vv_R\|_{L^2})
+C\Ef\nonumber\\
&\quad+C\ve\|\vv_R\|_{H^1}\|\ve\partial_i\FF_R\|_{L^2}
+\ve^2\langle\partial_i(\vv_R\cdot\nabla\widetilde{Q}),\partial_i\widetilde{\FF}_R\rangle\nonumber\\
&\leq
C(\Ef+\Ef^{\frac12}\Ff^{\frac12})
+C\Ef^{\frac12}\|\ve\partial_i\FF_R\|_{L^2}
+\ve^2\langle\partial_i(\vv_R\cdot\nabla\widetilde{Q}),\partial_i\widetilde{\FF}_R\rangle.
\end{align}
We proceed to deal with the term $\ve^2\langle\partial_i(\vv_R\cdot\nabla\widetilde{Q}),\partial_i\widetilde{\FF}_R\rangle$.
Using integration by parts yields
\begin{align}\label{ve2-vv-WFF1}
&\ve^2\langle\partial_i(\vv_R\cdot\nabla\widetilde{Q}),\partial_i\widetilde{\FF}_1\rangle\nonumber\\
&=
-\ve^2J\frac{\ud}{\ud t}\|\partial_i(\vv_R\cdot\nabla\widetilde{Q})\|^2_{L^2}
-\ve^2J\Big\langle\partial_i(\vv_R\cdot\nabla\widetilde{Q}),
\partial_i\tilde{\vv}\cdot\nabla(\vv_R\cdot\nabla\widetilde{Q})\Big\rangle\nonumber\\
&\leq
-\ve^2J\frac{\ud}{\ud t}\|\partial_i(\vv_R\cdot\nabla\widetilde{Q})\|^2_{L^2}
+C\Ef.
\end{align}
It is obvious from integration by parts that
\begin{align*}
&-\ve^5J\Big\langle\partial_i(\vv_R\cdot\nabla\widetilde{Q}),
(\partial_t+\tilde{\vv}\cdot\nabla)\partial_i(\vv_R\cdot\nabla Q_R)\Big\rangle\\
&\quad=
-\ve^5J\frac{\ud}{\ud t}\Big\langle\partial_i(\vv_R\cdot\nabla\widetilde{Q}),\partial_i(\vv_R\cdot\nabla Q_R)\Big\rangle\\
&\qquad+\underbrace{\ve^5J\Big\langle(\partial_t+\tilde{\vv}\cdot\nabla)\partial_i(\vv_R\cdot\nabla\widetilde{Q}),
\partial_i(\vv_R\cdot\nabla Q_R)\Big\rangle}_{\SSS_2},
\end{align*}
Then by Lemma \ref{lem:energy} we have
\begin{align}\label{ve2-vv-WFF2}
&\ve^2\langle\partial_i(\vv_R\cdot\nabla\widetilde{Q}),\partial_i\widetilde{\FF}_2\rangle\nonumber\\
&=
-\ve^5J\frac{\ud}{\ud t}\Big\langle\partial_i(\vv_R\cdot\nabla\widetilde{Q}),\partial_i(\vv_R\cdot\nabla Q_R)\Big\rangle
+\SSS_2\nonumber\\
&\quad
-\ve^5J\Big\langle\partial_i(\vv_R\cdot\nabla\widetilde{Q}),
\partial_i\tilde{\vv}\cdot\nabla(\vv_R\cdot\nabla Q_R)\Big\rangle
+\ve^5J\Big\langle\Delta(\vv_R\cdot\nabla\widetilde{Q}),
\vv_R\cdot\nabla \dot{Q}_R\Big\rangle\nonumber\\
&\quad
-\ve^8J\Big\langle(\vv_R\cdot\nabla)\Delta(\vv_R\cdot\nabla\widetilde{Q}),\vv_R\cdot\nabla Q_R\Big\rangle\nonumber\\
&\leq
-\ve^5J\frac{\ud}{\ud t}\Big\langle\partial_i(\vv_R\cdot\nabla\widetilde{Q}),\partial_i(\vv_R\cdot\nabla Q_R)\Big\rangle+\SSS_2
+C\ve^5\|\vv_R\|_{H^1}\|\vv_R\|_{H^2}\|\nabla Q_R\|_{H^1}\nonumber\\
&\quad
+C\ve^5\|\vv_R\|^2_{H^2}\|\nabla\dot{Q}_R\|_{L^2}
+C\ve^8\|\vv_R\|^2_{H^2}\|\vv_R\|_{H^3}\|\nabla Q_R\|_{L^2}\nonumber\\
&\leq
-\ve^5J\frac{\ud}{\ud t}\Big\langle\partial_i(\vv_R\cdot\nabla\widetilde{Q}),\partial_i(\vv_R\cdot\nabla Q_R)\Big\rangle+\SSS_2\nonumber\\
&\quad+C(\ve\Ef^{\frac32}+\ve^2\Ef^2+\ve\Ef\Ff^{\frac12}+\ve^2\Ef^{\frac32}\Ff^{\frac12}).
\end{align}
Thus from (\ref{ve2-vv-WFF1}) and (\ref{ve2-vv-WFF2}) we conclude that
\begin{align}\label{P-ivvRQ-FFR}
&\ve^2\langle\partial_i(\vv_R\cdot\nabla\widetilde{Q}),\partial_i\widetilde{\FF}_R\rangle\nonumber\\
&\leq-\ve^2J\frac{\ud}{\ud t}\|\partial_i(\vv_R\cdot\nabla\widetilde{Q})\|^2_{L^2}
-\ve^5J\frac{\ud}{\ud t}\Big\langle\partial_i(\vv_R\cdot\nabla\widetilde{Q}),\partial_i(\vv_R\cdot\nabla Q_R)\Big\rangle\nonumber\\
&\quad
+\SSS_2+C(\Ef+\ve^2\Ef^2+\ve^2\Ef\Ff).
\end{align}

We are now in a position to estimate the term $\ve^2\langle\partial_i\widetilde{\FF}_2,\partial_i\dot{Q}_R\rangle$. First, via employing integration by parts we find
\begin{align*}
&\ve^2\langle\partial_i\widetilde{\FF}_2,\partial_i\dot{Q}_R\rangle\\
&=
-\ve^5J\Big\langle\partial_i(\partial_t+\tilde{\vv}\cdot\nabla)(\vv_R\cdot\nabla Q_R),
\partial_i\dot{Q}_R\Big\rangle
-\ve^5J\Big\langle\partial_i(\vv_R\cdot\nabla \dot{Q}_R),
\partial_i\dot{Q}_R\Big\rangle\\
&\quad
-\ve^8J\Big\langle\partial_i(\vv_R\cdot\nabla(\vv_R\cdot\nabla Q_R)),\partial_i\dot{Q}_R\Big\rangle\\
&=
-\ve^5J\frac{\ud}{\ud t}\Big\langle\partial_i(\vv_R\cdot\nabla Q_R),\partial_i\dot{Q}_R\Big\rangle
+\ve^5J\Big\langle\partial_i(\vv_R\cdot\nabla Q_R),\partial_i\ddot{Q}_R\Big\rangle\\
&\quad
\underbrace{-\ve^5J\Big\langle\partial_i\tilde{\vv}\cdot\nabla(\vv_R\cdot\nabla Q_R),
\partial_i\dot{Q}_R\Big\rangle}_{\mathcal{B}_1}
\underbrace{-\ve^5J\Big\langle\partial_i(\vv_R\cdot\nabla Q_R),\partial_i\tilde{\vv}\cdot\nabla\dot{Q}_R\Big\rangle}_{\mathcal{B}_2}\\
&\quad\underbrace{-\ve^5J\langle\partial_i\vv_R\cdot\nabla\dot{Q}_R,\partial_i\dot{Q}_R\rangle}_{\mathcal{B}_3}
\underbrace{-\ve^8J\Big\langle\partial_i(\vv_R\cdot\nabla(\vv_R\cdot\nabla Q_R)),\partial_i\dot{Q}_R\Big\rangle}_{\mathcal{B}_4}.
\end{align*}
Using Lemma \ref{lem:energy}, we have
\begin{align*}
\mathcal{B}_1
\leq&C\ve^5\|\vv_R\|_{H^2}\|\nabla Q_R\|_{H^1}\|\partial_i\dot{Q}_R\|_{L^2}\leq C\ve\Ef^{\frac32},\\
\mathcal{B}_2
\leq&C\ve^5\|\vv_R\|_{H^2}\|\nabla Q_R\|_{H^1}\|\nabla\dot{Q}_R\|_{L^2}\leq C\ve\Ef^{\frac32},\\
\mathcal{B}_3
\leq&C\ve^5\|\partial_i\vv_R\|_{H^2}\|\nabla\dot{Q}_R\|_{L^2}\|\partial_i\dot{Q}_R\|_{L^2}
\leq C\ve\Ef\Ff^{\frac12},
\end{align*}
and
\begin{align*}
\mathcal{B}_4
=&-\ve^8J\Big\langle\partial_i\vv_R\cdot\nabla(\vv_R\cdot\nabla Q_R),\partial_i\dot{Q}_R\Big\rangle
-\underbrace{\ve^8J\Big\langle(\vv_R\cdot\nabla)\partial_i(\vv_R\cdot\nabla Q_R),\partial_i\dot{Q}_R\Big\rangle}_{\WW_2}\\
\leq&-\WW_2+C\ve^8\|\partial_i\vv_R\|_{H^2}\|\vv_R\|_{H^2}\|\nabla Q_R\|_{H^1}\|\partial_i\dot{Q}_R\|_{L^2}\\
\leq&-\WW_2+C(\ve^2\Ef^2+\ve^2\Ef^{\frac32}\Ff^{\frac12}).
\end{align*}
Similar to the estimate of (\ref{ve2-vR-ddQR-d2}), from the equation (\ref{remainder-Q-R}) we get
\begin{align*}
&\ve^5J\langle\partial_i(\vv_R\cdot\nabla Q_R),\partial_i\ddot{Q}_R\rangle\nonumber\\
&=-\ve^5\mu_1\langle\partial_i(\vv_R\cdot\nabla Q_R),\partial_i\dot{Q}_R\rangle
-\ve^5\Big\langle\partial_i(\vv_R\cdot\nabla Q_R),\frac{1}{\ve}\partial_i\MH^{\ve}_{\nn}(Q_R)\Big\rangle\nonumber\\
&\quad+\ve^5\Big\langle\partial_i(\vv_R\cdot\nabla Q_R),-\frac{\mu_2}{2}\partial_i\DD_R+\mu_1\partial_i[\BOm_R,Q_0]\Big\rangle\nonumber\\
&\quad+\ve^5\langle\partial_i(\vv_R\cdot\nabla Q_R),\partial_i\FF_R+\partial_i\widetilde{\FF}_R\rangle\nonumber\\
&\leq C\ve^5\|\vv_R\|_{H^2}\|\nabla Q_R\|_{H^1}
\big(\|\partial_i\dot{Q}_R\|_{L^2}+\|\partial_i\nabla\vv_R\|_{L^2}+\|\partial_i\FF_R\|_{L^2}\big)\nonumber\\
&\quad+C\ve^4\|\vv_R\|_{H^2}\|\nabla Q_R\|_{H^1}
\big(\|\partial_i\MH_{\nn}(Q_R)\|_{L^2}+\ve\|\partial_i\ML(Q_R)\|_{L^2}\big)\nonumber\\
&\quad+\ve^5\langle\partial_i(\vv_R\cdot\nabla Q_R),\partial_i\widetilde{\FF}_R\rangle\nonumber\\
&\leq\ve^5\langle\partial_i(\vv_R\cdot\nabla Q_R),\partial_i\widetilde{\FF}_R\rangle
+C(\ve\Ef^{\frac32}+\ve\Ef\Ff^{\frac12})+C\ve\Ef\|\ve\partial_i\FF_R\|_{L^2}.
\end{align*}
Thus collecting the above estimates, we can deduce that
\begin{align}\label{PFF4-dQR}
\ve^2\langle\partial_i\widetilde{\FF}_2,\partial_i\dot{Q}_R\rangle
&\leq
-\ve^5J\frac{\ud}{\ud t}\Big\langle\partial_i(\vv_R\cdot\nabla Q_R),\partial_i\dot{Q}_R\Big\rangle
+\ve^5\langle\partial_i(\vv_R\cdot\nabla Q_R),\partial_i\widetilde{\FF}_R\rangle\nonumber\\
&\quad-\WW_2+C(1+\Ef+\ve^2\Ef^2+\ve^8\Ef^5)+C\ve^2\Ef\Ff.
\end{align}

Our next task is to calculate the term $\ve^5\langle\partial_i(\vv_R\cdot\nabla Q_R),\partial_i\widetilde{\FF}_R\rangle$.
It is evident to see from integration by parts that
\begin{align*}
\ve^5\langle\partial_i(\vv_R\cdot\nabla Q_R),\partial_i\widetilde{\FF}_1\rangle
&=
-\SSS_2-J\ve^5\Big\langle\partial_i(\vv_R\cdot\nabla Q_R),\partial_i\tilde{\vv}\cdot\nabla(\vv_R\cdot\nabla\widetilde{Q})\Big\rangle\\
&\leq
-\SSS_2+C\ve^5\|\vv_R\|_{H^2}\|\nabla Q_R\|_{H^1}\|\vv_R\|_{H^1}\\
&\leq-\SSS_2+C\ve\Ef^{\frac32}.
\end{align*}
In addition, by integrating by parts we also have
\begin{align*}
&\ve^5\langle\partial_i(\vv_R\cdot\nabla Q_R),\partial_i\widetilde{\FF}_2\rangle\\
&=
-\ve^8J\frac{\ud}{\ud t}\|\partial_i(\vv_R\cdot\nabla Q_R)\|^2_{L^2}
-\ve^8J\Big\langle\partial_i(\vv_R\cdot\nabla Q_R),\partial_i\tilde{\vv}\cdot\nabla(\vv_R\cdot\nabla Q_R)\Big\rangle\\
&\quad
+\WW_2-\ve^8J\langle\partial_i(\vv_R\cdot\nabla Q_R),\partial_i\vv_R\cdot\nabla \dot{Q}_R\rangle\\
&\quad
-\ve^{11}J\Big\langle\partial_i(\vv_R\cdot\nabla Q_R),\partial_i(\vv_R\cdot\nabla(\vv_R\cdot\nabla Q_R))\Big\rangle\\
&\leq
-\ve^8J\frac{\ud}{\ud t}\|\partial_i(\vv_R\cdot\nabla Q_R)\|^2_{L^2}
+\WW_2
+C\ve^8\|\vv_R\|^2_{H^2}\|\nabla Q_R\|^2_{H^1}\\
&\quad+C\ve^8\|\vv_R\|^2_{H^2}\|\nabla Q_R\|_{H^2}\|\nabla\dot{Q}_R\|_{L^2}
+C\ve^{11}\|\vv_R\|^2_{H^2}\|\nabla Q_R\|^2_{H^1}\|\partial_i\vv_R\|_{H^2}\\
&\leq
-\ve^8J\frac{\ud}{\ud t}\|\partial_i(\vv_R\cdot\nabla Q_R)\|^2_{L^2}
+\WW_2
+C(\ve^2\Ef^2+\ve^4\Ef^{\frac52}+\ve^2\Ef^{\frac32}\Ff^{\frac12}+\ve^3\Ef^2\Ff^{\frac12}).
\end{align*}

Thus, combining the  latest two bounds with (\ref{part1-FFdotQ-R}) and (\ref{P-ivvRQ-FFR})-(\ref{PFF4-dQR}), it follows that
\begin{align*}
\ve^2\langle\partial_i\FF_R,\partial_i\dot{Q}_R\rangle
&\leq
-\ve^2\frac{J}{2}\frac{\ud}{\ud t}\|\partial_i(\vv_R\cdot Q^{\ve})\|^2_{L^2}
-\ve^2J\frac{\ud}{\ud t}\Big\langle\partial_i(\vv_R\cdot\nabla Q^{\ve}),\partial_i\dot{Q}_R\Big\rangle\\
&\quad
+C(1+\Ef+\ve^2\Ef^2+\ve^8\Ef^5)+(\delta+C\ve^2\Ef)\Ff.
\end{align*}

\subsection{The estimate of $\ve^4\langle\Delta\widetilde{\FF}_R,\Delta\dot{Q}_R\rangle$}

First note that
\begin{align*}
\Delta(\partial_t+\tilde{\vv}\cdot\nabla)
=(\partial_t+\tilde{\vv}\cdot\nabla)\Delta+\Delta\tilde{\vv}\cdot\nabla+2\partial_i\tilde{\vv}\cdot\nabla\partial_i,
\end{align*}
then from integration by parts we obtain
\begin{align}
\ve^4\langle\Delta\widetilde{\FF}_1,\Delta\dot{Q}_R\rangle
&=
-\ve^4J\Big\langle\Delta(\partial_t+\tilde{\vv}\cdot\nabla)(\vv_R\cdot\nabla\widetilde{Q}),\Delta\dot{Q}_R\Big\rangle\nonumber\\
&=
-\ve^4J\frac{\ud}{\ud t}\Big\langle\Delta(\vv_R\cdot\nabla\widetilde{Q}),\Delta\dot{Q}_R\Big\rangle
+\ve^4J\Big\langle\Delta(\vv_R\cdot\nabla\widetilde{Q}),\Delta\ddot{Q}_R\Big\rangle\nonumber\\
&\quad\underbrace{-\ve^4J\Big\langle(\Delta\tilde{\vv}\cdot\nabla+2\partial_i\tilde{\vv}\cdot\nabla\partial_i)
(\vv_R\cdot\nabla\widetilde{Q}),\Delta\dot{Q}_R\Big\rangle}_{\mathcal{C}_1}\nonumber\\
&\quad\underbrace{-\ve^4J\Big\langle\Delta(\vv_R\cdot\nabla\widetilde{Q}),
(\Delta\tilde{\vv}\cdot\nabla+2\partial_i\tilde{\vv}\cdot\nabla\partial_i)\dot{Q}_R\Big\rangle}_{\mathcal{C}_2}.\nonumber
\end{align}
It can be estimated by Lemma \ref{lem:energy} that
\begin{align*}
\mathcal{C}_1\leq&C\ve^4\|\vv_R\|_{H^2}\|\Delta\dot{Q}_R\|_{L^2}\leq C\Ef,\\
\mathcal{C}_2\leq&C\ve^4\|\vv_R\|_{H^2}\|\dot{Q}_R\|_{H^2}\leq C\Ef.
\end{align*}
Keeping the equation (\ref{remainder-Q-R}) in mind, we can deduce that
\begin{align}
&\ve^4J\Big\langle\Delta(\vv_R\cdot\nabla\widetilde{Q}),\Delta\ddot{Q}_R\Big\rangle\nonumber\\
&=
\underbrace{-\ve^4\mu_1\langle\Delta(\vv_R\cdot\nabla\widetilde{Q}),\Delta\dot{Q}_R\rangle}_{\mathcal{C}_3}
\underbrace{-\ve^4\Big\langle\Delta(\vv_R\cdot\nabla \widetilde{Q}),\frac{1}{\ve}\Delta\MH^{\ve}_{\nn}(Q_R)\Big\rangle}_{\mathcal{C}_4}\nonumber\\
&\quad+\underbrace{\ve^4\Big\langle\Delta(\vv_R\cdot\nabla\widetilde{Q}),
-\frac{\mu_2}{2}\Delta\DD_R+\mu_1\Delta[\BOm_R,Q_0]\Big\rangle}_{\mathcal{C}_5}
+\ve^4\langle\Delta(\vv_R\cdot\nabla\widetilde{Q}),\Delta\FF_R+\Delta\widetilde{\FF}_R\rangle.\nonumber
\end{align}
Using Lemma \ref{lem:energy}, we have
\begin{align*}
\mathcal{C}_3\leq&C\ve^4\|\vv_R\|_{H^2}\|\Delta\dot{Q}_R\|_{L^2}\leq C\Ef,\\
\mathcal{C}_5\leq&C\ve^4\|\vv_R\|_{H^2}\|\nabla\Delta\vv_R\|_{L^2}\leq C\Ef^{\frac12}\Ff^{\frac12},\\
\mathcal{C}_4=&\ve^3\Big\langle\partial_i\Delta(\vv_R\cdot\nabla \widetilde{Q}),\partial_i\MH^{\ve}_{\nn}(Q_R)\Big\rangle\\
\leq&C\ve^3\|\vv_R\|_{H^3}\big(\|\partial_i\MH_{\nn}(Q_R)\|_{L^2}+\ve\|\partial_i\ML(Q_R)\|_{L^2}\big)
\leq C(\Ef+\Ef^{\frac12}\Ff^{\frac12}).
\end{align*}
Then we get
\begin{align}\label{DeFF1-DdQR}
\ve^4\langle\Delta\widetilde{\FF}_1,\Delta\dot{Q}_R\rangle
&\leq
-\ve^4J\frac{\ud}{\ud t}\Big\langle\Delta(\vv_R\cdot\nabla\widetilde{Q}),\Delta\dot{Q}_R\Big\rangle
+\ve^4\langle\Delta(\vv_R\cdot\nabla\widetilde{Q}),\Delta\widetilde{\FF}_R\rangle\nonumber\\
&\quad+C\Ef^{\frac12}\|\ve^2\Delta\FF_R\|_{L^2}
+C(\Ef+\Ef^{\frac12}\Ff^{\frac12}).
\end{align}
Next, we estimate the quantity $\ve^4\langle\Delta(\vv_R\cdot\nabla\widetilde{Q}),\Delta\widetilde{\FF}_R\rangle$. Direct calculations yield that
\begin{align*}
&\ve^4\langle\Delta(\vv_R\cdot\nabla\widetilde{Q}),\Delta\widetilde{\FF}_1\rangle\\
&=
-\ve^4J\frac{\ud}{\ud t}\|\Delta(\vv_R\cdot\nabla\widetilde{Q})\|^2_{L^2}
-\ve^4J\Big\langle\Delta(\vv_R\cdot\nabla\widetilde{Q}),
\Delta\tilde{\vv}\cdot\nabla(\vv_R\cdot\nabla\widetilde{Q})\Big\rangle\\
&\quad
-\ve^4J\Big\langle\Delta(\vv_R\cdot\nabla\widetilde{Q}),
2\partial_i\tilde{\vv}\cdot\nabla\partial_i(\vv_R\cdot\nabla\widetilde{Q})\Big\rangle\\
&\leq
-\ve^4J\frac{\ud}{\ud t}\|\Delta(\vv_R\cdot\nabla\widetilde{Q})\|^2_{L^2}
+C\Ef.
\end{align*}
Using integration by parts, we derive the following bound
\begin{align*}
&\ve^4\langle\Delta(\vv_R\cdot\nabla\widetilde{Q}),\Delta\widetilde{\FF}_2\rangle\\
&=
-\ve^7J\frac{\ud}{\ud t}\Big\langle\Delta(\vv_R\cdot\nabla\widetilde{Q}),\Delta(\vv_R\cdot\nabla Q_R)\Big\rangle
+\underbrace{\ve^7J\langle\partial_i\Delta(\vv_R\cdot\nabla\widetilde{Q}),
\partial_i(\vv_R\cdot\nabla\dot{Q}_R)\rangle}_{\mathcal{C}_6}\\
&\quad
+\underbrace{\ve^7J\Big\langle(\partial_t+\tilde{\vv}\cdot\nabla)\Delta(\vv_R\cdot\nabla\widetilde{Q}),
\Delta(\vv_R\cdot\nabla Q_R)\Big\rangle}_{\SSS_3}\\
&\quad
\underbrace{-\ve^7J\Big\langle\Delta(\vv_R\cdot\nabla\widetilde{Q}),(\Delta\tilde{\vv}\cdot\nabla+2\partial_i\tilde{\vv}\cdot\nabla\partial_i)(\vv_R\cdot\nabla Q_R)\Big\rangle}_{\mathcal{C}_7}\\
&\quad
+\underbrace{\ve^{10}J\Big\langle\partial_i\Delta(\vv_R\cdot\nabla\widetilde{Q}),\partial_i(\vv_R\cdot\nabla(\vv_R\cdot\nabla Q_R))\Big\rangle}_{\mathcal{C}_8}.
\end{align*}
According to Lemma \ref{lem:energy}, we obtain
\begin{align*}
\mathcal{C}_6\leq&C\ve^7\|\vv_R\|_{H^3}\|\vv_R\|_{H^2}\|\nabla\dot{Q}_R\|_{H^1}
\leq C(\ve\Ef^{\frac32}+\ve\Ef\Ff^{\frac12}),\\
\mathcal{C}_7\leq&C\ve^7\|\vv_R\|^2_{H^2}\|\nabla Q_R\|_{H^2}
\leq C\ve\Ef^{\frac32},\\
\mathcal{C}_8\leq &C\ve^{10}\|\vv_R\|_{H^3}\|\vv_R\|^2_{H^2}\|\nabla Q_R\|_{H^2}
\leq C(\ve^2\Ef^2+\ve^2\Ef^{\frac32}\Ff^{\frac12}).
\end{align*}
Then we have
\begin{align}\label{ve4-vv-FFR}
&\ve^4\langle\Delta(\vv_R\cdot\nabla\widetilde{Q}),\Delta\widetilde{\FF}_R\rangle\nonumber\\
&\leq
-\ve^4J\frac{\ud}{\ud t}\|\Delta(\vv_R\cdot\nabla\widetilde{Q})\|^2_{L^2}
-\ve^7J\frac{\ud}{\ud t}\Big\langle\Delta(\vv_R\cdot\nabla\widetilde{Q}),\Delta(\vv_R\cdot\nabla Q_R)\Big\rangle\nonumber\\
&\quad
+\SSS_3
+C(\Ef+\ve\Ef^{\frac32}+\ve^2\Ef^2+\ve\Ef\Ff^{\frac12}+\ve^2\Ef^{\frac32}\Ff^{\frac12}).
\end{align}

Finally, it remains to estimate $\ve^4\langle\Delta\widetilde{\FF}_2,\Delta\dot{Q}_R\rangle$.
By integration by parts, we have
\begin{align*}
&\ve^4\langle\Delta\widetilde{\FF}_2,\Delta\dot{Q}_R\rangle\\
&=-\ve^7J\Big\langle\Delta(\partial_t+\tilde{\vv}\cdot\nabla)(\vv_R\cdot\nabla Q_R),
\Delta\dot{Q}_R\Big\rangle
-\ve^7J\langle\Delta(\vv_R\cdot\nabla\dot{Q}_R),\Delta\dot{Q}_R\rangle\\
&\quad
-\ve^{10}J\Big\langle\Delta(\vv_R\cdot\nabla(\vv_R\cdot\nabla Q_R)),\Delta\dot{Q}_R\Big\rangle\\
&=
-\ve^7J\frac{\ud}{\ud t}\Big\langle\Delta(\vv_R\cdot\nabla Q_R),\Delta\dot{Q}_R\Big\rangle
+\ve^7J\Big\langle\Delta(\vv_R\cdot\nabla Q_R),\Delta\ddot{Q}_R\Big\rangle\\
&\quad
\underbrace{-\ve^7J\Big\langle\Delta\tilde{\vv}\cdot\nabla(\vv_R\cdot\nabla Q_R),
\Delta\dot{Q}_R\Big\rangle}_{\mathcal{D}_1}
\underbrace{-\ve^7J\Big\langle2\partial_i\tilde{\vv}\cdot\nabla\partial_i(\vv_R\cdot\nabla Q_R),
\Delta\dot{Q}_R\Big\rangle}_{\mathcal{D}_2}\\
&\quad
\underbrace{-\ve^7J\Big\langle\Delta(\vv_R\cdot\nabla Q_R),
(\Delta\tilde{\vv}\cdot\nabla+2\partial_i\tilde{\vv}\cdot\nabla\partial_i)\dot{Q}_R\Big\rangle}_{\mathcal{D}_3}\\
&\quad
\underbrace{-\ve^7J\Big\langle\Delta\vv_R\cdot\nabla\dot{Q}_R+2\partial_i\vv_R\cdot\nabla\partial_i\dot{Q}_R,
\Delta\dot{Q}_R\Big\rangle}_{\mathcal{D}_4}\\
&\quad
\underbrace{-\ve^{10}J\Big\langle\Delta(\vv_R\cdot\nabla(\vv_R\cdot\nabla Q_R)),\Delta\dot{Q}_R\Big\rangle}_{\mathcal{D}_5}.
\end{align*}
Using Lemma \ref{lem:energy}, we get
\begin{align*}
\mathcal{D}_1\leq&C\ve^7\|\vv_R\|_{H^2}\|\nabla Q_R\|_{H^1}\|\Delta\dot{Q}_R\|_{L^2}
\leq C\ve^2\Ef^{\frac32},\\
\mathcal{D}_2\leq&C\ve^7\|\vv_R\|_{H^2}\|\nabla Q_R\|_{H^2}\|\Delta\dot{Q}_R\|_{L^2}
\leq C\ve\Ef^{\frac32},\\
\mathcal{D}_3\leq&C\ve^7\|\vv_R\|_{H^2}\|\nabla Q_R\|_{H^2}\|\dot{Q}_R\|_{H^2}
\leq C\ve\Ef^{\frac32},\\
\mathcal{D}_4\leq&C\ve^7\|\vv_R\|_{H^3}\|\dot{Q}_R\|_{H^2}\|\Delta\dot{Q}_R\|_{L^2}
+C\ve^7\|\partial_i\vv_R\|_{H^2}\|\nabla\partial_i\dot{Q}_R\|_{L^2}\|\Delta\dot{Q}_R\|_{L^2}\\
\leq& C(\ve\Ef^{\frac32}+\ve\Ef\Ff^{\frac12}),
\end{align*}
and
\begin{align*}
\mathcal{D}_5=&-\ve^{10}J\Big\langle\Delta\vv_R\cdot\nabla(\vv_R\cdot\nabla Q_R),\Delta\dot{Q}_R\Big\rangle
-\ve^{10}J\Big\langle2\partial_i\vv_R\cdot\nabla\partial_i(\vv_R\cdot\nabla Q_R),\Delta\dot{Q}_R\Big\rangle\\
&
-\underbrace{\ve^{10}J\Big\langle(\vv_R\cdot\nabla)\Delta(\vv_R\cdot\nabla Q_R),\Delta\dot{Q}_R\Big\rangle}_{\WW_3}\\
\leq&-\WW_3+C\ve^{10}\|\vv_R\|_{H^3}\|\vv_R\|_{H^2}\|\nabla Q_R\|_{H^2}\|\Delta\dot{Q}_R\|_{L^2}\\
\leq&-\WW_3+C(\ve^2\Ef^2+\ve^2\Ef^{\frac32}\Ff^{\frac12}).
\end{align*}
From the equation (\ref{remainder-Q-R}), we obtain
\begin{align*}
&\ve^7J\langle\Delta(\vv_R\cdot\nabla Q_R),\Delta\ddot{Q}_R\rangle\nonumber\\
&=\underbrace{-\ve^7\mu_1\langle\Delta(\vv_R\cdot\nabla Q_R),\Delta\dot{Q}_R\rangle}_{\mathcal{D}_6}
\underbrace{-\ve^7\Big\langle\Delta(\vv_R\cdot\nabla Q_R),\frac{1}{\ve}\Delta\MH^{\ve}_{\nn}(Q_R)\Big\rangle}_{\mathcal{D}_7}\nonumber\\
&\quad\underbrace{+\ve^7\Big\langle\Delta(\vv_R\cdot\nabla Q_R),-\frac{\mu_2}{2}\Delta\DD_R+\mu_1\Delta[\BOm_R,Q_0]\Big\rangle}_{\mathcal{D}_8}
+\underbrace{\ve^7\langle\Delta(\vv_R\cdot\nabla Q_R),\Delta\FF_R\rangle}_{\mathcal{D}_{9}}\nonumber\\
&\quad
+\ve^7\langle\Delta(\vv_R\cdot\nabla Q_R),\Delta\widetilde{\FF}_R\rangle.
\end{align*}
Likewise, applying Lemma \ref{lem:energy} leads to
\begin{align*}
\mathcal{D}_6\leq&C\ve^7\|\vv_R\|_{H^2}\|\nabla Q_R\|_{H^2}\|\Delta \dot{Q}_R\|_{L^2}
\leq C\ve\Ef^{\frac32},\\
\mathcal{D}_8\leq&C\ve^7\|\vv_R\|_{H^2}\|\nabla Q_R\|_{H^2}\|\nabla\Delta\vv_R\|_{L^2}
\leq C\ve\Ef\Ff^{\frac12},\\
\mathcal{D}_{9}\leq&C\ve^7\|\vv_R\|_{H^2}\|\nabla Q_R\|_{H^2}\|\Delta\FF_R\|_{L^2}
\leq C\ve\Ef\|\ve^2\Delta\FF_R\|_{L^2}.
\end{align*}
Notice that if we replace $\vv_0$ and $Q$ with $\vv_R$ and $\Delta Q_R$ in (\ref{MLQ-vNQ}), respectively, then it follows that
\begin{align*}
-\ve^7\Big\langle(\vv_R\cdot\nabla)\Delta Q_R,\ML(\Delta Q_R)\Big\rangle
\leq C\ve^7\|\nabla\vv_R\|_{H^2}\|\Delta Q_R\|^2_{H^1}
\leq C(\ve\Ef^{\frac32}+\ve\Ef\Ff^{\frac12}).
\end{align*}
Then we have
\begin{align*}
\mathcal{D}_7=&
-\ve^6\Big\langle\big(\Delta\vv_R\cdot\nabla+2\partial_i\vv_R\cdot\nabla\partial_i\big)Q_R,\Delta\MH_{\nn}(Q_R)\Big\rangle\\
&-\ve^6\Big\langle(\vv_R\cdot\nabla)\Delta Q_R,\Delta\MH_{\nn}(Q_R)\Big\rangle\\
&
+\ve^7\Big\langle\partial_j\big(\Delta\vv_R\cdot\nabla+2\partial_i\vv_R\cdot\nabla\partial_i\big)Q_R,\partial_j\ML(Q_R)\Big\rangle\\
&-\ve^7\Big\langle(\vv_R\cdot\nabla)\Delta Q_R,\ML(\Delta Q_R)\Big\rangle\\
\leq&C\ve^6\|\vv_R\|_{H^3}\|Q\|^2_{H^2}
+C\ve^6\|\vv_R\|_{H^2}\|\nabla\Delta Q_R\|_{L^2}\|Q_R\|_{H^2}\\
&+C\ve^7\|\vv_R\|_{H^3}\|Q_R\|^2_{H^3}
+C(\ve\Ef^{\frac32}+\ve\Ef\Ff^{\frac12})\\
\leq& C(\ve\Ef^{\frac32}+\ve\Ef\Ff^{\frac12}).
\end{align*}
Thus the following bound holds
\begin{align}\label{ve4-FF2-dQR}
\ve^4\langle\Delta\widetilde{\FF}_2,\Delta\dot{Q}_R\rangle
&\leq
-\ve^7J\frac{\ud}{\ud t}\Big\langle\Delta(\vv_R\cdot\nabla Q_R),\Delta\dot{Q}_R\Big\rangle
+\ve^7\langle\Delta(\vv_R\cdot\nabla Q_R),\Delta\widetilde{\FF}_R\rangle-\WW_3\nonumber\\
&\quad
+C(\ve\Ef^{\frac32}+\ve^2\Ef^2+\ve\Ef\Ff^{\frac12}+\ve^2\Ef^{\frac32}\Ff^{\frac12})
+C\ve\Ef\|\ve^2\Delta\FF_R\|_{L^2}.
\end{align}

We next deal with the term $\ve^7\langle\Delta(\vv_R\cdot\nabla Q_R),\Delta\widetilde{\FF}_R\rangle$.
It is easy to see that
\begin{align*}
\ve^7\langle\Delta(\vv_R\cdot\nabla Q_R),\Delta\widetilde{\FF}_1\rangle
&=
-\SSS_3
-\ve^7J\Big\langle\Delta(\vv_R\cdot\nabla Q_R),
(\Delta\tilde{\vv}\cdot\nabla+2\partial_i\tilde{\vv}\cdot\nabla\partial_i)(\vv_R\cdot\nabla\widetilde{Q})\Big\rangle\\
&\leq-\SSS_3+C\ve^7\|\vv_R\|^2_{H^2}\|\nabla Q_R\|_{H^2}\\
&\leq-\SSS_3+C\ve\Ef^{\frac32}.
\end{align*}
By a straightforward computation, one checks that
\begin{align*}
&\ve^7\langle\Delta(\vv_R\cdot\nabla Q_R),\Delta\widetilde{\FF}_2\rangle\\
&=-\ve^{10}\frac{J}{2}\frac{\ud}{\ud t}\|\Delta(\vv_R\cdot\nabla Q_R)\|^2_{L^2}
\underbrace{-\ve^{10}J\Big\langle\Delta(\vv_R\cdot\nabla Q_R),\Delta(\vv_R\cdot\nabla\dot{Q}_R)\Big\rangle}_{\mathcal{D}_{10}}\\
&\quad
\underbrace{-\ve^{10}J\Big\langle\Delta(\vv_R\cdot\nabla Q_R),
(\Delta\tilde{\vv}\cdot\nabla+2\partial_i\tilde{\vv}\cdot\nabla\partial_i)(\vv_R\cdot\nabla Q_R)\Big\rangle}_{\mathcal{D}_{11}}\\
&\quad
\underbrace{-\ve^{13}J\Big\langle\Delta(\vv_R\cdot\nabla Q_R),\Delta(\vv_R\cdot\nabla(\vv_R\cdot\nabla Q_R))\Big\rangle}_{\mathcal{D}_{12}}.
\end{align*}
Similarly, by Lemma \ref{lem:energy} we have
\begin{align*}
\mathcal{D}_{10}=&\WW_3-\ve^{10}J\Big\langle\Delta(\vv_R\cdot\nabla Q_R),
(\Delta\vv_R\cdot\nabla+2\partial_i\vv_R\cdot\nabla\partial_i)\dot{Q}_R\Big\rangle\\
\leq&\WW_3+C\ve^{10}\|\vv_R\|_{H^2}\|\nabla Q_R\|_{H^2}\|\vv_R\|_{H^3}\|\dot{Q}_R\|_{H^2}\\
\leq&\WW_3+C(\ve^2\Ef^2+\ve^2\Ef^{\frac32}\Ff^{\frac12}),\\
\mathcal{D}_{11}\leq&C\ve^{10}\|\vv_R\|^2_{H^2}\|\nabla Q_R\|^2_{H^2}\leq C\ve^2\Ef^2,\\
\mathcal{D}_{12}=&-\ve^{13}J\Big\langle\Delta(\vv_R\cdot\nabla Q_R),
(\Delta\vv_R\cdot\nabla+2\partial_i\vv_R\cdot\nabla\partial_i)(\vv_R\cdot\nabla Q_R)\Big\rangle\\
\leq&\ve^{13}\|\vv_R\|^2_{H^2}\|\nabla Q_R\|^2_{H^2}\|\vv_R\|_{H^3}
\leq C(\ve^4\Ef^{\frac52}+\ve^3\Ef^2\Ff^{\frac12}).
\end{align*}
Thus we get
\begin{align}\label{ve7-vvQR-FFR}
\ve^7\langle\Delta(\vv_R\cdot\nabla Q_R),\Delta\widetilde{\FF}_R\rangle
&\leq
-\ve^{10}\frac{J}{2}\frac{\ud}{\ud t}\|\Delta(\vv_R\cdot\nabla Q_R)\|^2_{L^2}
-\SSS_3+\WW_3\nonumber\\
&\quad+C(\ve\Ef^{\frac32}+\ve^2\Ef^2+\ve^4\Ef^{\frac52}+\ve^2\Ef^{\frac32}\Ff^{\frac12}+\ve^3\Ef^2\Ff^{\frac12}).
\end{align}
In conclusion, putting together these estimates (\ref{DeFF1-DdQR})-(\ref{ve7-vvQR-FFR}) and discarding the cancelation terms, we obtain the following estimate
\begin{align*}
\ve^4\langle\Delta\widetilde{\FF}_R,\Delta\dot{Q}_R\rangle
&\leq
-\ve^4\frac{J}{2}\frac{\ud}{\ud t}\|\Delta(\vv_R\cdot\nabla Q^{\ve})\|^2_{L^2}
-\ve^4J\frac{\ud}{\ud t}\Big\langle\Delta(\vv_R\cdot\nabla Q^{\ve}),\Delta\dot{Q}_R\Big\rangle\nonumber\\
&\quad
+C(1+\Ef+\ve^2\Ef^2+\ve^8\Ef^5)+(\delta+C\ve^2\Ef)\Ff.
\end{align*}

\subsection{$H^2$-estimate in Proposition \ref{prop:energy}}
We first act the derivative operator $\Delta$ on the equation (\ref{remainder-Q-R}), then multiply $\Delta\dot{Q}_R$ and integrate the resulting identity on $\BR$ with respect to $\xx$. Again applying the operator $\Delta$ on the equation (\ref{remainder-v-R}) and taking $L^2$-inner product with $\Delta\vv_R$ enable us to derive the following equality:
\begin{align*}
&\ve^4\Big\langle\partial_t(\Delta\vv_R),\Delta\vv_R\Big\rangle
+\ve^4J\Big\langle\partial_t(\Delta\dot{Q}_R),\Delta\dot{Q}_R\Big\rangle\\
&=
\underbrace{-\ve^4\bigg\langle\Delta\Big(\beta_1Q_0(Q_0:\DD_R)+\beta_4\DD_R
+\beta_5\DD_R\cdot Q_0
+\beta_6Q_0\cdot\DD_R\Big),\nabla\Delta\vv_R\bigg\rangle}_{\mathcal{K}_{1}}\\
&\quad
\underbrace{-\ve^4\beta_7\Big\langle\Delta(\DD_R\cdot Q^2_0+Q^2_0\cdot\DD_R),\nabla\Delta\vv_R\Big\rangle}_{\mathcal{K}_{2}}\\
&\quad
\underbrace{-\ve^4\frac{\mu_2}{2}\Big\langle\Delta(\dot{Q}_R-[\BOm_R,Q_0]),\nabla\Delta\vv_R\Big\rangle}_{\mathcal{K}_3}
\underbrace{-\ve^4\mu_1\Big\langle\Delta\big[Q_0,(\dot{Q}_R-[\BOm_R,Q_0])\big],\nabla\Delta\vv_R\Big\rangle}_{\mathcal{K}_4}\\
&\quad
\underbrace{-\ve^4\langle\Delta\tilde{\vv}\cdot\nabla\vv_R+2\partial_i\tilde{\vv}\cdot\nabla\partial_i\vv_R, \Delta\vv_R\rangle
-\ve^4\langle\Delta\GG_R,\nabla\Delta\vv_R\rangle+\ve^4\langle\Delta\GG'_R,\Delta\vv_R\rangle}_{\mathcal{K}_5}\\
&\quad
\underbrace{-\ve^4\frac{\mu_2}{2}\langle\Delta\DD_R,\Delta\dot{Q}_R\rangle}_{\mathcal{K}_6}
\underbrace{-\ve^4\mu_1\langle\Delta\dot{Q}_R-\Delta[\BOm_R,Q_0],\Delta\dot{Q}_R\rangle^2_{L^2}}_{\mathcal{K}_7}
\underbrace{-\ve^4\Big\langle\frac{1}{\ve}\Delta\MH^{\ve}_{\nn}(Q_R),\Delta\dot{Q}_R\Big\rangle}_{\mathcal{K}_8}\\
&\quad
\underbrace{-\ve^4\Big\langle\Delta\tilde{\vv}\cdot\nabla\dot{Q}_R+2\partial_i\tilde{\vv}\cdot\nabla\partial_i\dot{Q}_R,\Delta\dot{Q}_R\Big\rangle}_{\mathcal{K}_9}
+\ve^4\langle\Delta\FF_R+\Delta\widetilde{\FF}_R,\Delta\dot{Q}_R\rangle.
\end{align*}
The terms on the right-hand sides can be estimated as follows. By the analysis for the construction of the terms $\mathcal{K}_1$ and $\mathcal{K}_2$, we have
\begin{align*}
\mathcal{K}_{1}+\mathcal{K}_2
\leq&-\ve^4\Big\langle\beta_1Q_0(Q_0:\Delta\DD_R)+\beta_4\Delta\DD_R
+\beta_5\Delta\DD_R\cdot Q_0+\beta_6Q_0\cdot\Delta\DD_R,\nabla\Delta\vv_R\Big\rangle\\
&
-\ve^4\Big\langle\beta_7(\Delta\DD_R\cdot Q^2_0+Q^2_0\cdot\Delta\DD_R),\nabla\Delta\vv_R\Big\rangle
+C\|\ve^2\nabla\vv_R\|_{H^1}\|\ve^2\nabla\Delta\vv_R\|_{L^2}\\
\leq&-\ve^4\Big\langle\beta_1Q_0(Q_0:\Delta\DD_R)+\beta_4\Delta\DD_R
+\frac{\beta_5+\beta_6}{2}(Q_0\cdot\Delta\DD_R+\Delta\DD_R\cdot Q_0),\Delta\DD_R\Big\rangle\\
&
-\ve^4\Big\langle\beta_7(\Delta\DD_R\cdot Q^2_0+Q^2_0\cdot\Delta\DD_R),\Delta\DD_R\Big\rangle
+\underbrace{\ve^4\frac{\mu_2}{2}\langle[\Delta\DD_R,Q_0],\nabla\Delta\vv_R\rangle}_{\mathcal{K}'_{1}}
+C\Ef^{\frac12}\Ff^{\frac12}.
\end{align*}
It can be easy to observe that
\begin{align*}
\mathcal{K}'_{1}+\mathcal{K}_3+\mathcal{K}_6
&\leq
\ve^4\frac{\mu_2}{2}\langle[\Delta\DD_R,Q_0],\Delta\BOm_R\rangle
-\ve^4\mu_2\langle\Delta\DD_R,\Delta\dot{Q}_R\rangle\\
&\quad+\ve^4\frac{\mu_2}{2}\langle[\Delta\BOm_R,Q_0],\Delta\DD_R\rangle
+C\|\ve^2\nabla\vv_R\|_{H^1}\|\ve^2\nabla\Delta\vv_R\|_{L^2}\\
&\leq -\ve^4\mu_2\langle\Delta\dot{Q}_R-[\Delta\BOm_R,Q_0],\Delta\DD_R\rangle
+C\Ef^{\frac12}\Ff^{\frac12}.
\end{align*}
The terms $\mathcal{K}_4$ and $\mathcal{K}_7$ can be estimated as
\begin{align*}
\mathcal{K}_4+\mathcal{K}_7\leq&
-\ve^4\mu_1\Big\langle\big[Q_0,(\Delta\dot{Q}_R-[\Delta\BOm_R,Q_0])\big],\nabla\Delta\vv_R\Big\rangle\\
&-\ve^4\mu_1\langle\Delta\dot{Q}_R-[\Delta\BOm_R,Q_0],\Delta\dot{Q}_R\rangle\\
&+C\Big(\ve\|\ve\dot{Q}_R\|_{H^1}+\|\ve^2\nabla\vv_R\|_{H^1}\Big)\|\ve^2\nabla\Delta\vv_R\|_{L^2}\\
&+C\|\ve^2\nabla\vv_R\|_{H^1}\|\ve^2\Delta\dot{Q}_R\|_{L^2}\\
\leq&
-\ve^4\mu_1\big\|\Delta\dot{Q}_R-[\Delta\BOm_R,Q_0]\big\|^2_{L^2}
+C(\Ef^{\frac12}\Ff^{\frac12}+\Ef).
\end{align*}
Combining with the following equality
\begin{align*}
&-\ve^4\mu_1\big\|\Delta\dot{Q}_R-[\Delta\BOm_R,Q_0]\big\|^2_{L^2}
-\ve^4\mu_2\langle\Delta\dot{Q}_R-[\Delta\BOm_R,Q_0],\Delta\DD_R\rangle\\
&\quad=-\ve^4\mu_1\big\|\Delta\dot{Q}_R-[\Delta\BOm_R,Q_0]+\frac{\mu_2}{2\mu_1}\Delta\DD_R\big\|^2_{L^2}
+\frac{\mu^2_2}{4\mu_1}\|\Delta\DD_R\|^2_{L^2},
\end{align*}
and by using the dissipation relation (\ref{diss:ineq}), then we have the following estimate
\begin{align*}
&\mathcal{K}_1+\mathcal{K}_2+\mathcal{K}_{3}+\mathcal{K}_4+\mathcal{K}_6+\mathcal{K}_7\\
&\leq
-\ve^4\beta_1s^2\|\nn\nn:\Delta\DD_R\|^2_{L^2}
-\ve^4\Big(\beta_4-\frac{s(\beta_5+\beta_6)}{3}+\frac29\beta_7s^2\Big)\|\Delta\DD_R\|^2_{L^2}\\
&\quad
-\ve^4\Big(s(\beta_5+\beta_6)+\frac23\beta_7s^2\Big)\|\nn\cdot\Delta\DD_R\|^2_{L^2}
-\ve^4\mu_2\langle\Delta\dot{Q}_R-[\Delta\BOm_R,Q_0],\Delta\DD_R\rangle\\
&\quad-\ve^4\mu_1\big\|\Delta\dot{Q}_R-[\Delta\BOm_R,Q_0]\big\|^2_{L^2}
+C(\Ef^{\frac12}\Ff^{\frac12}+\Ef)\\
&\leq
-\ve^4\tilde{\beta}_1\|\nn\nn:\Delta\DD_R\|^2_{L^2}-\ve^4\tilde{\beta}_2\|\Delta\DD_R\|^2_{L^2}
-\ve^4\tilde{\beta}_3\|\nn\cdot\Delta\DD_R\|^2_{L^2}\\
&\quad
-4\ve^4\delta\|\Delta\DD_R\|^2_{L^2}
+C(\Ef^{\frac12}\Ff^{\frac12}+\Ef)\\
&\leq
-4\ve^4\delta\|\nabla\Delta\vv_R\|^2_{L^2}+C\Ef+\delta\Ff,
\end{align*}
where $\delta>0$ is small enough, such that the coefficients $\tilde{\beta}_i(i=1,2,3)$ given by (\ref{tilde-beta}) satisfy the relation (\ref{diss:coeff}).
As for the estimates of the terms $\mathcal{K}_{5}$ and $\mathcal{K}_{9}$, it is easy to obtain
\begin{align*}
\mathcal{K}_{5}+\mathcal{K}_{9}
\leq&C\Big(\|\ve^2\nabla\vv_R\|_{H^1}\|\ve^2\Delta\vv_R\|_{L^2}
+\|\ve^2\Delta\GG_R\|_{L^2}\|\ve^2\nabla\Delta\vv_R\|_{L^2}\\
&+\|\ve^2\Delta\GG'_R\|_{L^2}\|\ve^2\Delta\vv_R\|_{L^2}
+\|\ve^2\nabla\dot{Q}_R\|_{H^1}\|\ve^2\Delta\dot{Q}_R\|_{L^2}\Big)\\
\leq&C\Ef+C(\|\ve^2\Delta\GG_R\|_{L^2}\Ff^{\frac12}+\|\ve^2\Delta\GG'_R\|_{L^2}\Ef^{\frac12}).
\end{align*}
Similar to the derivation of (\ref{mathcalJ-8}), the term $\mathcal{K}_{8}$ can be handled as
\begin{align*}
\mathcal{K}_8\leq-\frac{\ve^3}{2}\frac{\ud}{\ud t}\langle\MH^{\ve}_{\nn}(\Delta Q_R),\Delta Q_R\rangle
+\delta\Ff+C\Ef.
\end{align*}

As a consequence, from the above estimates, we can conclude that
\begin{align*}
&\ve^4\Big\langle\partial_t(\Delta\vv_R),\Delta\vv_R\Big\rangle
+\ve^4J\Big\langle\partial_t(\Delta\dot{Q}_R),\Delta\dot{Q}_R\Big\rangle\nonumber\\
&\quad+\frac{\ve^3}{2}\frac{\ud}{\ud t}\langle\MH^{\ve}_{\nn}(\Delta Q_R),\Delta Q_R\rangle
+4\ve^4\delta\|\nabla\Delta\vv_R\|^2_{L^2}\nonumber\\
&\leq
C\Big(\|\ve^2\Delta\GG_R\|_{L^2}\Ff^{\frac12}+\|\ve^2\Delta\GG'_R\|_{L^2}\Ef^{\frac12}
+\|\ve^2\Delta\FF_R\|_{L^2}\Ef^{\frac12}\Big)\nonumber\\
&\quad
+\ve^4\langle\Delta\widetilde{\FF}_R,\Delta\dot{Q}_R\rangle+C\Ef+\delta\Ff.
\end{align*}

\bigskip

\noindent{\bf Acknowledgments.}
Sirui Li is supported by the NSF of China under Grant No. 11601099 and by the Science and Technology Foundation of Guizhou Province of China under Grant No. [2017]1032.
 Wei Wang is supported by NSF of China under Grant No. 11922118 and 11771388, and the Young Elite Scientists Sponsorship Program by CAST.

\end{document}